\documentclass[11pt]{article}
\usepackage{fullpage}
\usepackage{amsmath}
\usepackage{amsfonts}
\usepackage{amsthm}
\usepackage{graphicx}
\usepackage{subcaption}
\usepackage{framed}
\usepackage{bm}

\newtheorem{proposition}{Proposition}

\newtheorem{definition}{Definition}

\title{Jet Lag Recovery: Synchronization of Circadian Oscillators as a Mean Field Game}
\author{
	Ren\'{e} Carmona\thanks{Operations Research and Financial Engineering, Princeton University, Partially supported by NSF \#DMS-1716673 and ARO \#W911NF-17-1-0578} \and Christy V. Graves\thanks{Program in Applied and Computational Mathematics, Princeton University, Partially supported by NSF GRFP}\\
}

\begin{document}

\maketitle

\begin{abstract}
	The Suprachiasmatic Nucleus (SCN) is a region in the brain that is responsible for controlling circadian rhythms. The SCN contains on the order of $10^4$ neuronal oscillators which have a preferred period slightly longer than $24$ hours. The oscillators try to synchronize with each other as well as responding to external stimuli such as sunlight exposure.  A mean field game model for these neuronal oscillators is formulated with two goals in mind: 1) to understand the long time behavior of the oscillators when an individual remains in the same time zone, and 2) to understand how the oscillators recover from jet lag when the individual has traveled across time zones. In particular, we would like to study the claim that jet lag is worse after traveling east than west. Finite difference schemes are used to find numerical approximations to the mean field game solutions. Numerical results are presented and conjectures are posed. The numerics suggest the time to recover from jet lag is about the same for east versus west trips, but the cost the oscillators accrue while recovering is larger for eastward trips.
\end{abstract}

\section{Introduction}
\hspace{5mm} Circadian rhythm refers to the oscillatory behavior of certain biological processes occurring with a period close to $24$ hours. Examples of circadian rhythms in animals include sleep/wake patterns, eating schedules, bodily temperatures, hormone production, and brain activity. These oscillations can be entrained to the 24 hour cycle of sunlight exposure. Abrupt disruptions of such circadian rhythms can occur, such as when an individual travels across time zones, resulting in jet lag.

The Suprachiasmatic Nucleus (SCN) is a region in the brain that is responsible for controlling circadian rhythms \cite{bio1}\cite{bio2}. The SCN contains on the order of $10^4$ neuronal oscillator cells, each of which has a preferred frequency corresponding to a period slightly longer than 24 hours \cite{bio3}. We model the oscillators as rational players wishing to synchronize with each other as well as the natural $24$ hour sunlight cycle, while minimizing their effort. Since there are a large number of SCN cells, we consider the limit as the number of players tends to infinity and view the game from the perspective of mean field games (MFG). Mean field games were proposed independently by Lasry and Lions \cite{lasry2007mean}, and Caines and his collaborators \cite{huang2006large}, to approximate games with a large number of players with symmetric interactions of a mean field nature. Finite difference schemes for numerically solving the PDE formulations of mean field games are discussed in the works of Achdou and collaborators \cite{achdou2013finite}\cite{achdou2013mean}\cite{achdou2010mean}\cite{achdou2012iterative}.

The goal of this study is to understand the behavior of the oscillators in two settings: 1) the long time behavior of the oscillators for an individual that is entrained to the 24 hour light/dark rhythm, and 2) how the oscillators resynchronize to a shifted 24 hour light/dark rhythm after travel across time zones.

This project was inspired by the work of Lu, Klein-Carde\~{n}a, Lee, Antonsen, Girvan, and Ott, in which they formulated a model of jet lag for SCN oscillators using the classical Kuramoto model \cite{lu2016resynchronization}. In their model, each oscillator has a random preferred frequency, and their phases evolve forward in time according to deterministic coupled ODEs. By making an ansatz and turning to the limit as the number of oscillators tends to infinity, they reduce the dynamics to an ODE for a compex order parameter. They are able to show a larger recovery time for eastward travel. Our model differs in that we use a game theoretic approach: the oscillators choose their controls to minimize a cost objective.

Other relevant work includes that of Yin, Mehta, Meyn, and Shanbhag \cite{yin2012synchronization}, in which they formulated a mean field game model for the synchronization of oscillators. In their model, the oscillators choose their control to minimize a cost objective which encourages synchronization of the oscillators with each other. There is no external forcing term, such as the natural 24 hour light/dark cycle, which is relevant for SCN oscillators.

Our main finding from the numerical results is that the time to recover from jet lag is about the same for east versus west travels. However, there is a larger cost associated with recovery from jet lag after traveling east.

Our problem formulation is described in Section \ref{sec:Problem_Formulation}. Our model of long time behavior of the oscillators is provided in Section \ref{sec:Stationary_Solutions}, while Section \ref{sec:Recovery_Problem} provides two models for resynchronization after travel. A comment on existence and uniqueness for these models is provided in Section \ref{sec:Existence_Uniqueness}. Jet lag recovery is described in Section \ref{sec:Jet_Lag_Recovery} where we provide a few notions of jet lag recovery time and jet lag recovery cost.

A finite difference approach is used to find numerical approximations to the solutions. The numerical methods are described in Section \ref{sec:Numerical_Methods} and results are presented in Section \ref{sec:Numerical_Results}. From these numerics, we have formulated conjectures in Section \ref{sec:Conjectures}. Section \ref{sec:Conclusion} concludes the paper.

\section{Model Formulation}\label{sec:Problem_Formulation}
\subsection{$N$ Player Game Formulation}
\hspace{5mm} We model the SCN as $N$ coupled oscillators, $i \in \{1, \dots N \}$, with phases $\Theta^i_t$ in the periodic domain $[0,2\pi)$. Without any external stimuli, SCN oscillators are believed to have an intrinsic frequency $\omega^i$ corresponding to a period slightly longer than $24$ hours \cite{bio3}. We will take the frequency to be the same for all of the oscillators: $\omega^i=\omega_0,\ \forall i  \in \{1, \dots N \}$. Note that this differs from the formulation in \cite{lu2016resynchronization}, where the $\omega^i$ are taken to be random. We model oscillator $i$ as advancing it's phase, $\Theta^i_t$, forward in time according to the It\^{o} equation:
\begin{equation*}
d\Theta^i_t=(\omega_0+\alpha^i_t)dt+\sigma dB^i_t,
\end{equation*}
where $B^i_t$ are independent standard Wiener processes. The control of oscillator $i$, $\bm{\alpha^i}=(\alpha_t^i)_{t\geq0} \in \mathbb{A}$, is chosen to minimize the long run average cost:
		\begin{equation*}
		J^{\mu^N}(\bm{\alpha}):=\limsup_{T \rightarrow \infty} \frac{1}{T}\int\limits_0^T \left(\frac{1}{2} \alpha_t^2+K\bar{c}^N(\Theta^i_t,\mu^N(t,\cdot))+Fc_{sun}(t,\Theta^i_t,\rho(t)) \right) dt,
		\end{equation*}
where $\mathbb{A}$ is the set of locally square integrable stochastic processes that are adapted to $B^i_t$. $\mu^N(t,\cdot))=\frac{1}{N}\sum_{j=1}^{N}\delta_{\Theta^i_t}$ is the empirical measure of the $N$ oscillators at time $t$. The cost function $\bar{c}^N$ encourages synchronization of the oscillators with each other through their empirical measure. $\bar{c}^N$ is given by:
		\begin{equation*}
		\bar{c}^N(\theta,\nu)=\frac{1}{2}\int_0^{2\pi}\sin^2\left(\frac{\theta'-\theta}{2}\right) d\nu(\theta').
		\end{equation*}
The cost function $c_{sun}$ encourages alignment of the oscillators with the natural 24 hour sunlight cycle and is given by:
		\begin{equation*}
		c_{sun}(t,\theta,\rho(t))=\frac{1}{2}\sin^2\left(\frac{\omega_St+\rho(t)-\theta}{2}\right),
		\end{equation*}
where $\omega_S=\frac{2\pi}{24}$ radians per hour is the frequency of the 24 hour sunlight cycle, and $\rho(t):\mathbb{R}^{+} \rightarrow [0,2\pi)$ is a phase shift accounting for the time zone angle at time $t$. $\rho(t)=p$ is used for an individual that stays in their time zone forever, whereas $\rho(t)$ increases for eastward travel and decreases for westward travel. The two constants, $K, F \geq 0$ are used to weigh the three components of the cost.

\subsection{Mean Field Limit Formulation}
\hspace{5mm} By considering the limit $N \rightarrow \infty$, we reformulate the problem as a mean field game:

\begin{enumerate}
	\item Fix a deterministic flow of measures, $\bm{\mu}=(\mu(t,\cdot))_{t\geq0}$.
	\item Solve the standard control problem of finding $\bm{\alpha^{\bm{\mu}}}=(\alpha_t^{\bm{\mu}})_{t\geq0}\in\mathbb{A}$ to minimize the long run average (LRA) cost:
		\begin{equation*}
		J^{\bm{\mu}}(\bm{\alpha})=\limsup_{T \rightarrow \infty} \frac{1}{T}\int_0^T \left(\frac{1}{2} \alpha_t^2+K\bar{c}(\Theta_t^{\bm{\alpha}},\mu(t,\cdot))+Fc_{sun}(t,\Theta_t,\rho(t)) \right) dt,
		\end{equation*}
		\begin{equation*}
		\bar{c}(\theta,\nu)=\frac{1}{2}\int_0^{2\pi}\sin^2\left(\frac{\theta'-\theta}{2}\right) d\nu(\theta'),
		\end{equation*}
		\begin{equation*}
		c_{sun}(t,\theta,\rho(t))=\frac{1}{2}\sin^2\left(\frac{\omega_St+\rho(t)-\theta}{2}\right),
		\end{equation*}
		subject to the dynamical constraint:
				\begin{equation*}
				d\Theta_t^{\bm{\alpha}}=(\omega_0+\alpha_t)dt+\sigma dB_t.
				\end{equation*}
	\item Find a fixed point such that $\mu(t,\theta)=\mathcal{L}(\Theta_t^{\bm{\alpha}^{\bm{\mu}}}),\ \forall t$.
\end{enumerate}	
The convergence of the $N$ player game to the mean field game is outside the focus of this report. For the remainder of the paper, we will work with the mean field game.

\subsection{Change of Variables}
\hspace{5mm} For now, let $\rho(t)=p$ be constant. Note that the time dependency of $c_{sun}$ can be removed if we make the change of variables: $\Phi_t=\Theta_t-\omega_St$, where $\Phi_t$ is also in the periodic domain $[0,2\pi)$. Now the optimal control part of the mean field game is the following. A generic oscillator evolves it's phase according to:
	\begin{equation}
	d\Phi_t=(\omega_0-\omega_S+\alpha_t)dt+\sigma dB_t,
	\label{eq:dynamics_phi}
	\end{equation}
where $\bm\alpha\in\mathbb{A}$ is chosen to minimize the long run average (LRA) cost:
		\begin{equation}
		J^{\bm{\mu}}(\bm\alpha)=\limsup_{T \rightarrow \infty} \frac{1}{T}\int_0^T \left(\frac{1}{2} \alpha_t^2+K\bar{c}(\Phi_t,\mu(t,\cdot))+Fc_{sun}(\Phi_t,p) \right) dt,
		\label{eq:cost_phi}
		\end{equation}
where:
		\begin{equation}
		\bar{c}(\phi,\nu)=\frac{1}{2}\int_0^{2\pi}\sin^2\left(\frac{\phi'-\phi}{2}\right) d\nu(\phi'),
		\label{eq:cost_cbar_phi}
		\end{equation}
		\begin{equation}
		c_{sun}(\phi,p)=\frac{1}{2}\sin^2\left(\frac{p-\phi}{2}\right).
		\label{eq:cost_csun_phi}
		\end{equation}
		
Now that the cost functions no longer depend on $t$, we can look at the long time behavior of the oscillators when an individual remains in the same time zone angle $p$. Thus, the model formulation in the new variable, $\Phi$, will be useful to address our first problem of interest: to understand the long time behavior of the oscillators when an individual remains in the same time zone. In fact, this formulation will also be used to understand the second problem of interest as well: to understand how the oscillators recover from jet lag when the individual has traveled across time zones. Thus, for the remainder of the paper, we will refer to $\Phi_t$ as the phase of a representative oscillator at time $t$.
		
\section{Long Time Behavior in One Time Zone}\label{sec:Stationary_Solutions}
\hspace{5mm} When an individual remains in the same time zone angle $p$ for a long time, we expect the distribution of their oscillators to be a stationary solution to the mean field game posed in equations (\ref{eq:dynamics_phi})-(\ref{eq:cost_csun_phi}). Using the analytic (PDE) approach to solving mean field games, this stationary solution is given by $(\mu^*_p(\phi),V^*_p(\phi),\lambda^*_p)$ solving the coupled ergodic Hamilton-Jacobi-Bellman (HJB) and Poisson equations:
	\begin{equation*}
	(\omega_0-\omega_S) \partial_\phi V^*_p-\frac{1}{2}(\partial_\phi V^*_p)^2+\frac{\sigma^2}{2}\partial_{\phi \phi}^2 V^*_p=\lambda^*_p-K\bar{c}(\phi,\mu^*_p(\cdot))-Fc_{sun}(\phi,p),
	\end{equation*}
	\begin{equation*}
	(\omega_0-\omega_S) \partial_{\phi} \mu^*_p- \partial_{\phi}\left[\mu^*_p(\partial_\phi V^*_p) \right]- \frac{\sigma^2}{2} \partial_{\phi \phi}^2 \mu^*_p=0.
	\end{equation*}
Since $V^*_p$ is defined up to a constant, without loss of generality, we add the constraint $\int_0^{2\pi}V^*_p=0$. Also, $\mu_p^*$ needs to be a probability measure. We require $\mu^*_p(\phi)\geq0$ and $\int_0^{2\pi}d\mu^*_p(\phi)=1$. Note that $\lambda_p^*$ is given by the ergodic average cost, i.e.:
\begin{equation*}
\lambda_p^*=\int_0^{2\pi} \left[\frac{1}{2}(\partial_\phi V^*_p)^2+K\bar{c}(\phi,\mu^*_p(\cdot))+Fc_{sun}(\phi,p) \right] d\mu^*_p(\phi).
\end{equation*}

Let $(\mu^*(\phi),V^*(\phi),\lambda^*)$ without subscripts denote a solution when $p=0$. It is easy to check that for $p \neq 0$, $(\mu^*_p(\phi),V^*_p(\phi),\lambda^*_p)=(\mu^*(\phi-p),V^*(\phi-p),\lambda^*)$ is a solution. Thus, without loss of generality, let $p=0$ and pose the \textit{Ergodic Mean Field Game Problem}.
\begin{framed}
\textit{Ergodic Mean Field Game Problem}: Find $(\mu^*(\phi),V^*(\phi),\lambda^*)$ solving
	\begin{equation}
	(\omega_0-\omega_S) \partial_\phi V^*-\frac{1}{2}(\partial_\phi V^*)^2+\frac{\sigma^2}{2}\partial_{\phi \phi}^2 V^*=\lambda^*-K\bar{c}(\phi,\mu^*(\cdot))-Fc_{sun}(\phi,0),
	\label{eq:ergodic_HJB}
	\end{equation}
	\begin{equation}
	(\omega_0-\omega_S) \partial_{\phi} \mu^*- \partial_{\phi}\left[\mu^*(\partial_\phi V^*) \right]- \frac{\sigma^2}{2} \partial_{\phi \phi}^2 \mu^*=0,
	\label{eq:Poisson}
	\end{equation}
		\begin{equation*}
		\bar{c}(\phi,\mu^*(\cdot))=\frac{1}{2}\int_0^{2\pi}\sin^2\left(\frac{\phi'-\phi}{2}\right) d\mu^*(\phi'),
		\end{equation*}
		\begin{equation*}
		c_{sun}(\phi,0)=\frac{1}{2}\sin^2\left(\frac{\phi}{2}\right),
		\end{equation*}
	\begin{equation}
	\int_0^{2\pi}V^*(\phi)d\phi=0,
	\label{eq:normalization}
	\end{equation}
	\begin{equation}
	\mu^*(\phi)\geq0,\ \int_0^{2\pi}d\mu^*(\phi)=1.
	\label{eq:prob_measure}
	\end{equation}
\end{framed}

The \textit{Ergodic Mean Field Game Problem} describes the long time behavior of the oscillators for an individual that remains in time zone angle $p=0$. We emphasize that if we can solve the \textit{Ergodic Mean Field Game Problem} as written for $p=0$, then we have simultaneously solved for the long time behavior in any time zone angle $p$, since, as noted above, we can take $\mu^*(\phi-p)$. We define \textit{entrainment} as the following.

\begin{definition}
An individual is \textit{entrained} to the time zone angle $p$ if the phases of their oscillators are distributed according to $\mu^*(\phi-p)$, where $\mu^*(\phi)$ solves the \textit{Ergodic Mean Field Game Problem}.
\end{definition}
In other words, $\mathcal{L}(\Phi)=\mu^*(\phi-p)$. Note that the optimal control is given by $\alpha^*(\phi-p):=-\partial_\phi V^*(\phi-p)$.

\subsection{Special Case: $K=0$, $\omega_0=\omega_S$}
In this section, we consider the special case $K=0$ and $\omega_0=\omega_S$. It will be useful to define some notation related to Mathieu's differential equation:
\begin{equation}
	\partial^2_{xx}f+[a-2q\cos(2x)]f=0.
	\label{eq:Mathieu}
\end{equation}
Let $\mathbb{M}(a,q,x)$ denote the unique even solution to equation (\ref{eq:Mathieu}) with the normalization constraint $\mathbb{M}(a,q,0)=1$ and let $\mathcal{C}(q)$ be defined such that $\mathbb{M}(\mathcal{C}(q),q,x)$ is periodic with period $\pi$. We have the following result.

\begin{proposition}
	In the special case $K=0$ and $\omega_0=\omega_S$, the solution to the \textit{Ergodic Mean Field Game Problem}, denoted $(\mu^*_{K_0}(\phi),V^*_{K_0}(\phi),\lambda^*_{K_0})$ in this case, is given by:
	\begin{equation*}
		\partial_\phi V^*_{K_0}(\phi)=-\sigma^2\partial_\phi \left[\log\left(\mathbb{M}\left(\mathcal{C}\left(-\frac{2F}{\sigma^4}\right),-\frac{2F}{\sigma^4},\frac{\phi}{2}\right)\right)\right],
	\end{equation*}
	\begin{equation*}
		\lambda^*_{K_0}=\frac{F}{2}+\frac{\sigma^2}{8}\mathcal{C}\left(-\frac{2F}{\sigma^4}\right),
	\end{equation*}
	\begin{equation*}
		\mu^*_{K_0}(\phi)=\frac{1}{c_1}\left[\mathbb{M}\left(\mathcal{C}\left(-\frac{2F}{\sigma^4}\right),-\frac{2F}{\sigma^4},\phi/2\right)\right]^2,
	\end{equation*}
	where $c_1$ is the normalization constant:
	\begin{equation*}
		\int_0^{2\pi}\left[\mathbb{M}\left(\mathcal{C}\left(-\frac{2F}{\sigma^4}\right),-\frac{2F}{\sigma^4},\phi/2\right)\right]^2 d\phi.
	\end{equation*}
\end{proposition}
\begin{proof}
The ergodic HJB can be linearized through the Cole-Hopf transformation:
\begin{equation*}
	W_{K_0}(\phi)=e^{-V^*_{K_0}(\phi)/\sigma^2}.
\end{equation*}
Thus, when $K=0$ and $\omega_0=\omega_S$, equation (\ref{eq:ergodic_HJB}) becomes:
\begin{equation*}
\partial^2_{\phi \phi}W_{K_0}+\frac{2}{\sigma^4} \left[\lambda^*_{K_0}-\frac{F}{2}+\frac{F}{2} \cos\left(\phi\right) \right]W_{K_0}=0.
\end{equation*}
There are even and odd solution to this equation. Since $W_{K_0}(\phi)=e^{-V^*_{K_0}(\phi)/\sigma^2}>0$, we want an even solution. Thus, a periodic solution with period $2\pi$ is given by:
\begin{equation*}
	W_{K_0}(\phi)=c_2\mathbb{M}\left(\mathcal{C}\left(-\frac{2F}{\sigma^4}\right),-\frac{2F}{\sigma^4},\phi/2\right),
\end{equation*}
where $\lambda^*_{K_0}$ solves:
\begin{equation*}
-\frac{4(F-2\lambda^*_{K_0})}{\sigma^4}=\mathcal{C}\left(-\frac{2F}{\sigma^4}\right).
\end{equation*} 
The constant $c_2>0$ can be chosen to satisfy the normalization constraint (\ref{eq:normalization}) and does not affect $\partial_\phi V^*_{K_0}$. Finally, it can be verified that $\mu^*_{K_0}(\phi)=\frac{W^2_{K_0}(\phi)}{\int_0^{2\pi}W^2_{K_0}(\phi)d\phi}$ solves equations (\ref{eq:Poisson}) and (\ref{eq:prob_measure}).
\end{proof}

\subsection{Perturbation of Special Case: Small Interaction}
In this section, we consider the case $\omega_0=\omega_S$ and small $K>0$. Borrowing from the ideas of Chan and Sircar \cite{chan_sircar}, we expand our solution $(\mu^*(\phi),V^*(\phi),\lambda^*)$ to the \textit{Ergodic Mean Field Game Problem} by the parameter $K$:
\begin{equation}
\begin{split}
\mu^*(\phi)&=\mu^*_{K_0}(\phi)+K\mu^*_{K_1}(\phi)+K^2\mu^*_{K_2}(\phi)+ \dots \\
V^*(\phi)&=V^*_{K_0}(\phi)+KV^*_{K_1}(\phi)+K^2V^*_{K_2}(\phi)+ \dots \\
\lambda^*&=\lambda^*_{K_0}+K\lambda^*_{K_1}+K^2\lambda^*_{K_2}+ \dots
\end{split}
\label{eq:expansion}
\end{equation}
We derived the 0th order solution, $(\mu^*_{K_0}(\phi),V^*_{K_0}(\phi),\lambda^*_{K_0})$, in the previous section. Now we consider the 1st order correction, $(\mu^*_{K_1}(\phi),V^*_{K_1}(\phi),\lambda^*_{K_1})$. We have the following result.

\begin{proposition}
	If $\omega_0=\omega_S$, the first order correction for small $K>0$ is given by:
	\begin{equation}
		\partial_\phi V^*_{K_1}(\phi)=\frac{2}{\sigma^2\mu^*_{K_0}(\phi)} \left[c_{K_1}+\int_0^\phi  \mu^*_{K_0}(\phi')\left(\lambda^*_{K_1}-\left(\frac{1}{2}\sin^2\left(\frac{\cdot}{2}\right)* \mu^*_{K_0}\right)(\phi')\right)d\phi'\right],
	\label{eq:V_1}
	\end{equation}
	where
	\begin{equation}
		c_{K_1}=-\frac{\int_0^{2\pi}\int_0^\phi \frac{\mu^*_{K_0}(\phi')}{ \mu^*_{K_0}(\phi)}\left(\lambda^*_{K_1}-\left(\frac{1}{2}\sin^2\left(\frac{\cdot}{2}\right)* \mu^*_{K_0}\right)(\phi')\right)d\phi'd\phi}{\int_0^{2\pi}\frac{1}{\mu^*_{K_0}(\phi)}d\phi},
	\label{eq:c_K_1}
	\end{equation}
	and
	\begin{equation}
		\lambda^*_{K_1}=\int_0^{2\pi}\left(\frac{1}{2}\sin^2\left(\frac{\cdot}{2}\right)*\mu^*_{K_0}\right)(\phi)d\phi,
	\label{eq:lambda_1}
	\end{equation}
	where $*$ denotes convolution. Finally, the 1st order correction to the density is given by $\mu^*_{K_1}(\phi)=\Gamma(\phi) W_{K_0}(\phi)$ where $\Gamma(\phi)$ solves:
	\begin{equation}
	\frac{\sigma^2}{2}\left(\Gamma \partial^2_{\phi \phi}W_{K_0}-W_{K_0}\partial^2_{\phi \phi}\Gamma \right)=\frac{2}{\sigma^2}\mu^*_{K_0}\left(\lambda^*_{K_1}-\left(\frac{1}{2}\sin^2\left(\frac{\cdot}{2}\right)*\mu^*_{K_0}\right)(\phi)\right).
	\label{eq:gamma}
	\end{equation}
\end{proposition}
\begin{proof}
	After substituting the expansion from equation (\ref{eq:expansion}) into the ergodic HJB in equation (\ref{eq:ergodic_HJB}), and collecting terms of order $K$, we arrive at:
	\begin{equation*}
		\frac{\sigma^2}{2}\partial^2_{\phi \phi}V^*_{K_1}-\partial_\phi V^*_{K_0}\partial_\phi V^*_{K_1}=\lambda^*_{k_0}-\left(\frac{1}{2}\sin^2\left(\frac{\cdot}{2} \right) * \mu^*_{K_0}\right)(\phi).
	\end{equation*}
	Using the integration factor method, we arrive at equation (\ref{eq:V_1}) for some constant $c_{K_1}$. Since $V^*_{K_1}$ is periodic, we must have $\int_0^{2\pi}\partial_\phi V^*_{K_1} d\phi=0$, which gives equation (\ref{eq:c_K_1}) for $c_{K_1}$.
	
	Next, we substitute the expansion from equation (\ref{eq:expansion}) into the Poisson equation in equation (\ref{eq:Poisson}). After collecting terms of order $K$, we arrive at:
	\begin{equation*}
		\partial_\phi \left[\mu^*_{K_0}\partial_\phi V^*_{K_1}+\mu^*_{K_1}\partial_\phi V^*_{K_0}) \right]+\frac{\sigma^2}{2}\partial^2_{\phi \phi}\mu^*_{K_1}=0.
	\end{equation*}
	Next, we use the fact that $\partial_\phi V^*_{K_0}=-\sigma^2\frac{\partial_\phi W_{K_0}}{W_{K_0}}$ and we make the substitution $\mu_1(\phi)=\Gamma(\phi) W_0(\phi)$, which leads to:
	\begin{equation*}
		\frac{\sigma^2}{2}\left(\Gamma \partial^2_{\phi \phi}W_{K_0} -W_{K_0}\partial^2_{\phi \phi}\Gamma\right)=\partial_\phi \left[\mu^*_{K_0}\partial_\phi V^*_{K_1} \right].
	\end{equation*}
	Substituting equation (\ref{eq:V_1}), we have equation (\ref{eq:gamma}). Since $\Gamma$, $W_{K_0}$, and their derivatives are periodic,
	\begin{equation*}
		\int_0^{2\pi}\Gamma \partial^2_{\phi \phi}W_{K_0}d\phi= \int_0^{2\pi} W_{K_0}\partial^2_{\phi \phi}\Gamma d\phi,
	\end{equation*}
	and thus with equation (\ref{eq:gamma}), we have:
	\begin{equation*}
	\begin{split}
		0&=\int_0^{2\pi}\frac{2}{\sigma^2}  \mu^*_{K_0}(\phi)\left(\lambda^*_{K_1}-\left(\frac{1}{2}\sin^2\left(\frac{\cdot}{2}\right)* \mu^*_{K_0}\right)(\phi)\right)d\phi, \\
		&=\lambda^*_{K_1}-\int_0^{2\pi}\left(\frac{1}{2}\sin^2\left(\frac{\cdot}{2}\right)* \mu^*_{K_0}\right)(\phi)d\phi,
	\end{split}
	\end{equation*}
	which gives equation (\ref{eq:lambda_1}).
\end{proof}
		
\section{Jet Lag Recovery}\label{sec:Recovery_Problem}
\hspace{5mm} In Section \ref{sec:Stationary_Solutions}, we posed the \textit{Ergodic Mean Field Game Problem}, which models the long time behavior of oscillators when an individual remains in their time zone for a long period of time. In other words, this problem models oscillators which are entrained to their local time zone. The second goal of this study is to model the resynchronization of oscillators after travel to a new time zone (i.e. how the oscillators return to the ergodic solution after switching time zones). We provide two such models: in the first model, which will be called the \textit{Recovery via Ergodic Problem}, the oscillators use the optimal controls they have already learned from the \textit{Ergodic Mean Field Game Problem} to resychronize to a new time zone; in the second model, which will be called the \textit{Recovery Mean Field Game Problem}, the oscillators solve a new mean field game associated with traveling to a new time zone. We now discuss the \textit{Recovery via Ergodic Problem}.

If we revisit our model formulation, provided by equations (\ref{eq:dynamics_phi})-(\ref{eq:cost_csun_phi}), we notice that the long run average form of the cost is not amenable to modeling the transient period of jet lag recovery. To be more specific, if the model is ergodic, the form of the cost in equation (\ref{eq:cost_phi}) will have the same value, regardless of the behavior of the oscillators between time $0$ and time $T$ for arbitrarily large $T>0$. Thus, the long run average form of the cost does not penalize the oscillators for taking an arbitrarily large amount of time to resynchronize, which is undesirable.

Since the oscillators have already learned the optimal control, $\alpha^*(\phi)$, for long time behavior at time zone angle $0$, they have simultaneously learned the optimal control, $\alpha^*(\phi-p)$, for long time behavior at the time zone angle $p$. If the problem is ergodic and they use the control $\alpha^*(\phi-p)$ after traveling $p/\omega_S$ time zones, then the distribution of the oscillators will converge to $\mu^*(\phi-p)$. Thus, the oscillators could use the control they've already learned, $\alpha^*(\phi-p)$, to recover from traveling $p/\omega_S$ time zones. We would like to compute the distribution, denoted $\mu_p(t,\phi)$, of the oscillators as they recover from traveling $p/\omega_S$ time zones. We make the following assumptions:
\begin{enumerate}
	\item Travel is immediate. Without loss of generality,
			\begin{equation*}
			\begin{split}
			\rho(0)&=0, \\
			\rho(t)&=p \in [-\pi,\pi),\ \forall t>0.
			\end{split}
			\end{equation*}
	\item The individual is entrained to their local time zone before travel begins:
	\begin{equation*}
	\mu_p(0,\phi)=\mu^*(\phi),
	\end{equation*}
	where $\mu^*(\phi)$ is a solution to the \textit{Ergodic Mean Field Game Problem}.
	\item The oscillators adopt the control:
	\begin{equation*}
		\alpha_p(t,\phi):=\alpha^*(\phi-p),\ \forall t>0,
	\end{equation*}
	where $\alpha^*(\phi)=-\partial_\phi V^*(\phi)$ is computed from the solution to the \textit{Ergodic Mean Field Game Problem}.
\end{enumerate}
Under these assumptions, the law of the phases of the oscillators is given by solving the Kolmogorov/ Fokker-Planck equation forward in time. We now pose the \textit{Recovery via Ergodic Problem}.

\begin{framed}
\textit{Recovery via Ergodic Problem}: Find $\mu_p(t,\phi)$ solving	
	\begin{equation}
	\begin{split}
	\partial_t \mu_p+(\omega_0-\omega_S) \partial_{\phi} \mu_p+ \partial_{\phi}\left[\mu_p\alpha_p \right]- \frac{\sigma^2}{2} \partial_{\phi \phi}^2 \mu_p&=0, \\
	\mu_p(0,\phi)&=\mu^*(\phi).
	\end{split}
	\label{eq:Kolm_recovery}
	\end{equation}
\end{framed}
Note that since the initial condition, $\mu^*(\phi)$, is a probability measure, a solution $\mu_p(t,\phi)$ to equation (\ref{eq:Kolm_recovery}) is a probability measure for all $t\geq0$.

Since we model the oscillators as selfish players when formulating the \textit{Ergodic Mean Field Game Problem}, it is feasible that the oscillators would also wish to find a mean field game equilibrium while recovering from jet lag. As mentioned earlier, an infinite time horizon is not amenable to modeling the period of time during which jet lag recovery occurs. Since it is realistic to expect jet lag recovery in a finite amount of time, we reformulate our second model for the time-dependent problem of jet lag recovery, called the \textit{Recovery Mean Field Game Problem}, by replacing the long run average form of the cost with a finite horizon. For the time-dependent problem with a finite horizon $T$, we need to specify an initial condition for the oscillators and a terminal cost. We would like to compute the distribution, denoted $\tilde{\mu}_p(t,\phi)$, of the oscillators as they recover from traveling $p/\omega_S$ time zones, while using the control computed from the value function $\tilde{V}_p(t,\phi)$. We make the following assumptions:
\begin{enumerate}
	\item Travel is immediate. Without loss of generality,
	\begin{equation*}
	\begin{split}
	\rho(0)&=0, \\
	\rho(t)&=p, \forall t>0.
	\end{split}
	\end{equation*}
	\item The individual is entrained to their local time zone before travel begins:
	\begin{equation*}
	\tilde{\mu}_p(0,\phi)=\mu^*(\phi),
	\end{equation*}
	where $\mu^*(\phi)$ is the solution to the \textit{Ergodic Mean Field Game Problem}.
	\item $T$ is chosen sufficiently large that the terminal condition $\tilde{V}_p(T,\phi)$ will not play a significant role in the behavior of the oscillators during resynchronization. Thus, we let
	\begin{equation*}
	\tilde{V}_p(T,\phi)=0.
	\end{equation*}
\end{enumerate}
Using the analytic approach to mean field games, we formulate the solution using coupled forward backward HJB and Kolmogorov/Fokker-Planck equations and pose the \textit{Recovery Mean Field Game Problem}.
\begin{framed}
\textit{Recovery Mean Field Game Problem}: Find $(\tilde{\mu}_p(t,\phi),\tilde{V}_p(t,\phi))$ solving
\begin{equation*}
\partial_t \tilde{V}_p+(\omega_0-\omega_S) \partial_\phi \tilde{V}_p-\frac{1}{2}(\partial_\phi \tilde{V}_p)^2+\frac{\sigma^2}{2}\partial_{\phi \phi}^2 \tilde{V}_p=-K\bar{c}(\phi,\tilde{\mu}_p(t,\cdot))-Fc_{sun}(\phi,p)
\end{equation*}
\begin{equation*}
\tilde{V}_p(T,\phi)=0
\end{equation*}
\begin{equation*}
\partial_t \tilde{\mu}_p+(\omega_0-\omega_S) \partial_{\phi} \tilde{\mu}_p- \partial_{\phi}\left[\tilde{\mu}_p(\partial_\phi \tilde{V}_p) \right]- \frac{\sigma^2}{2} \partial_{\phi \phi}^2 \tilde{\mu}_p=0
\end{equation*}
\begin{equation*}
\bar{c}(\phi,\tilde{\mu}_p(t,\cdot))=\frac{1}{2}\int_0^{2\pi}\sin^2\left(\frac{\phi'-\phi}{2}\right) d\tilde{\mu}_p(t,\phi')
\end{equation*}
\begin{equation*}
c_{sun}(\phi,p)=\frac{1}{2}\sin^2\left(\frac{p-\phi}{2}\right)
\end{equation*}
\begin{equation*}
\tilde{\mu}_p(t,\phi)\geq0,\int_0^{2\pi}d\tilde{\mu}_p(t,\phi)=1
\end{equation*}
\begin{equation*}
\tilde{\mu}_p(0,\phi)=\mu^*(\phi)
\end{equation*}
\end{framed}
Note that the optimal control is given by $\tilde{\alpha}_p(t,\phi):=-\partial_\phi \tilde{V}_p(t,\phi)$. To summarize the proposed models, in Section \ref{sec:Stationary_Solutions}, we posed the \textit{Ergodic Mean Field Game Problem} to describe the long time behavior of the oscillators for an individual who remains in one time zone. To model jet lag recovery, we assume that travel is immediate, the oscillators are synchronized with their time zone before travel, and the oscillators attempt to resynchronize to the new time zone by either adapting the control from the \textit{Ergodic Mean Field Game Problem} for the new time zone, as in the \textit{Recovery via Ergodic Problem}, or the oscillators find a new mean field game equilibrium throughout their recovery for a finite time horizon, as in the \textit{Recovery Mean Field Game Problem}. In the next sections, we discuss existence and uniqueness for these models and provide some definitions for quantifying jet lag recovery.

\section{Existence \& Uniqueness}\label{sec:Existence_Uniqueness}
\hspace{5mm} For the \textit{Ergodic Mean Field Game Problem}, the existence of a classical solution follows from the early papers of Lasry and Lions, relying on compactness arguments and a use of Schauder's fixed point theorem \cite{lasry2007mean}. A detailed proof is given for Theorem 2.1 in the paper of Bardi and Feleqi \cite{bardi2016nonlinear}. Existence for the \textit{Recovery Mean Field Game Problem} also follows \cite{lasry2007mean}. The notes of Cardaliaguet based on the lectures of Lions provide a detailed proof which can be readily extended to our \textit{Recovery Mean Field Game Problem} \cite{cardaliaguet2010notes}.

Uniqueness, on the other hand, is hard to prove in general for mean field games. Typically, uniqueness arguments are based on a monotonicity condition for the cost functions in the measure argument. This Lasry Lions monotonicity condition \cite{lasry2007mean} is not satisfied for the running cost function $\bar{c}(\phi,\nu)$. In fact, we can use the results in \cite{yin2012synchronization} to show that there is not uniqueness for the \textit{Ergodic Mean Field Game Problem} in at least some cases. For example, if $F=0$, $\omega_0=\omega_S$, and $K=1$, we recover the model in \cite{yin2012synchronization}, where their parameter $\omega$ is zero. Their model is slightly different in that they look for solutions on an infinite time horizon which may depend on time. In this case, there are at least two solutions: the uniform distribution, which they call the \textit{incoherence solution}, and a `traveling wave' solution that they prove in Theorem 4.3. Since the speed of the `traveling wave' is $\omega=0$, the solution does not depend on time, and we recover a second solution to our \textit{Ergodic Mean Field Game Problem}. Thus, uniqueness does not hold in general if $F=0$. We will require $F>0$ and we will assume uniqueness in order to have clear definitions for jet lag recovery. In fact, assuming uniqueness allows us to uniquely define the initial condition for both recovery problems and the control to be used for the \textit{Recovery via Ergodic Problem}.

Given a solution $\mu^*$ and $\alpha^*$ to the \textit{Ergodic Mean Field Game Problem}, existence and uniqueness for the \textit{Recovery via Ergodic Problem} follows from the standard PDE literature.

We also note that in the case $K=0$, the monotonicity condition of Lasry and Lions is satisfied, and we have uniqueness for both the \textit{Ergodic Mean Field Game Problem} and the \textit{Recovery Mean Field Game Problem} \cite{lasry2007mean}. In fact, if $K=0$, there is no mean field interaction so uniqueness comes from standard optimal control theory.

\section{Quantifying Jet Lag Recovery}\label{sec:Jet_Lag_Recovery}
\hspace{5mm} To quantify jet lag recovery, we measure the time it takes to recover from jet lag, and the cost the oscillators accrued while recovering from jet lag. First, we need to define what it means to have recovered from jet lag. To simplify the presentation, we write the following definitions in terms of the solution $\mu_p(t,\phi)$ of the \textit{Recovery via Ergodic Problem}, where analogous definitions are clear for the solution $(\tilde{\mu}_p(t,\phi),\tilde{V}_p(t,\phi)$ to the \textit{Recovery Mean Field Game Problem}. Since we are assuming $(\mu^*(\phi),V^*(\phi),\lambda^*)$ solving the \textit{Ergodic Mean Field Game Problem} is unique, we define jet lag recovery as the following.

\begin{definition}
For a given $p \in[-\pi,\pi)$ and $\epsilon^W>0$, we say that the oscillators have \textit{$\epsilon^W$-recovered} from jet lag for traveling $p/\omega_S$ time zones away if $\exists$ $t\geq 0$ s.t.:
\begin{equation*}
\mathcal{W}_2(\mu_p(t,\phi),\mu^*(\phi-p))<\epsilon^W,
\end{equation*}
where $\mathcal{W}_2$ denotes the 2-Wasserstein distance, $\mu^*(\phi)$ denotes the solution to the \textit{Ergodic Mean Field Game Problem} and $\mu_p(t,\phi)$ denotes the solution to the \textit{Recovery via Ergodic Problem} for the given value of $p$.
\end{definition}
In other words, the oscillators are within $\epsilon^W$ of entrainment after traveling from time zone angle $0$ to time zone angle $p$. We define the recovery time as:
\begin{equation*}
\tau^W_p:=\inf\{t>0:\mathcal{W}_2(\mu_p(t,\phi),\mu^*(\phi-p)) < \epsilon^W\}.
\end{equation*}

A second notion of jet lag recovery time is based on the analog of the order parameter introduced in Lu et al. \cite{lu2016resynchronization}. Let $z^*=\int_0^{2\pi}e^{i\phi}d\mu^*(\phi)$. Similarly, let $z_p(t)=\int_0^{2\pi}e^{i\phi}d\mu_p(t,\phi)$. We define a second notion of jet lag recovery time:
\begin{equation*}
\tau^z_p:=\inf\{t>0:|z_p(t)-e^{ip}z^*| < \epsilon^z\}.
\end{equation*}

We also consider the cost accrued while recovering from jet lag using the values of the cost functions that the oscillators wish to minimize. We define the instantaneous contributions to the overall cost by:
\begin{equation*}
f_{p,\alpha}(t):=\int_0^{2\pi} \frac{1}{2}\alpha_p(t,\phi)^2d\mu_p(t,\phi)=\int_0^{2\pi} \frac{1}{2}\alpha^*(\phi-p)^2d\mu_p(t,\phi),
\end{equation*}
\begin{equation*}
f_{p,osc}(t):=\int_0^{2\pi} \bar{c}(\phi,\mu_p(t,\cdot))d\mu_p(t,\phi),
\end{equation*}
\begin{equation*}
f_{p,sun}(t):=\int_0^{2\pi} c_{sun}(\phi,p)d\mu_p(t,\phi),
\end{equation*}
\begin{equation*}
f_p(t):=\frac{f_{\alpha}(t)+K\cdot f_{osc}(t)+F\cdot f_{sun}(t)}{1+K+F},
\end{equation*}
To simplify the presentations of the numerical results, we compare the cost accrued over the first $10$ days after a time zone change, since we find that the oscillators recover within $10$ days for almost all of the cases we consider. With an abuse of notation, we compute $f_{p,\alpha}:=\int_0^{240} f_{p,\alpha}(t) dt$, $f_{p,osc}:=\int_0^{240} f_{p,osc}(t) dt$, $f_{p,sun}:=\int_0^{240} f_{p,sun}(t) dt$, and $f_{p}:=\int_0^{240} f_{p}(t) dt$.

As mentioned earlier, we have analogous definitions, denoted $\tilde{\tau^W_p}$, $\tilde{\tau^z_p}$, $\tilde{z}_p(t)$, $\tilde{f}_{p,\alpha}$, $\tilde{f}_{p,osc}$, $\tilde{f}_{p,sun}$, and $\tilde{f}_{p}$ computed from the \textit{Recovery Mean Field Game Problem} solution $(\tilde{\mu}_p(t,\phi),\tilde{V}_p(t,\phi))$.

\section{Numerical Methods}\label{sec:Numerical_Methods}
\hspace{5mm} We have three problems to solve numerically: the \textit{Ergodic Mean Field Game Problem}, the \textit{Recovery via Ergodic Problem}, and the \textit{Recovery Mean Field Game Problem}. The numerical complications are mainly due to the nonlinearity in the HJB equation, as well as the coupling between the ergodic HJB and the Poisson equation. For all three problems, our numerical approach is to use finite differences. The ideas for our numerical methods were drawn from the papers on numerically solving mean field games by Achdou and collaborators \cite{achdou2013finite}\cite{achdou2013mean}\cite{achdou2010mean}\cite{achdou2012iterative}. Our methods are described in more detail below. Note that the \textit{Ergodic Mean Field Game Problem} needs to be solved first, since its solution is used as the initial condition for the two recovery problems.

\subsection{Ergodic Mean Field Game Problem: Finite Difference Equations}\label{sec:Ergodic Mean Field Game Problem: Finite Difference Equations}
\hspace{5mm} We approximate the solution $(\mu^*(\phi),V^*(\phi),\lambda^*)$ to the \textit{Ergodic Mean Field Game Problem} by $(M=\{M_j\}_{0 \leq j < n},U=\{U_j\}_{0 \leq j < n},\Lambda)$ where $M_j \approx \mu^*(j \Delta \phi)$, $U_j \approx V^*(j \Delta \phi),$ and $\Lambda \approx \lambda^*$. The grid step size is $\Delta \phi:= \frac{2 \pi}{n}$, where $n$ is the number of grid points.

To linearize the HJB equation, we use the following approximation of the control, $\beta_j \approx \alpha^*(j \Delta \phi)=-\partial_\phi V^*(j\Delta \phi)$. Using a monotone scheme to define the direction of the first derivative, the discretized HJB becomes:
	\begin{equation}
	\begin{split}
	(\omega_0-\omega_S+\beta_j)^+\cdot \frac{U_{j+1}-U_{j}}{\Delta \phi}+(\omega_0-\omega_S+\beta_j)^-\cdot\frac{U_j-U_{j-1}}{\Delta \phi}& \\
	+\frac{\sigma^2}{2}\frac{U_{j+1}-2U_{j}+U_{j-1}}{\Delta \phi^2}=
	\Lambda-\frac{1}{2}\beta_j^2-K\bar{c}(j \Delta \phi,M)-Fc_{sun}(j \Delta \phi),\ \forall j,&
	\end{split}
	\label{eq:discretized_HJB}
	\end{equation}
	\begin{equation*}
	\bar{c}(j\Delta \phi,M)=\frac{1}{2}\sum_{k=0}^{n-1}\sin^2\left(\frac{(k-j)\Delta \phi}{2}\right)M_k \Delta \phi,\ \forall j,
	\end{equation*}
	\begin{equation*}
	c_{sun}(j \Delta \phi)=\frac{1}{2}\sin^2\left(\frac{j \Delta \phi}{2}\right),\ \forall j,
	\end{equation*}
	\begin{equation}
	\sum_{j=0}^{n-1}U_j=0.
	\label{eq:sum_V}
	\end{equation}
Note that the domain is periodic, so $U_{-1}=U_{n-1}$, and $U_n=U_0$, and similarly for $M$. We define the matrix $L_{\beta}$ to be the linear operator on the left hand side of equation (\ref{eq:discretized_HJB}). Equation (\ref{eq:discretized_HJB}) becomes:
	\begin{equation}
	\left(L_{\beta}U\right)_j=\Lambda-\frac{1}{2}\beta_j^2-K\bar{c}(j \Delta \phi,M)-Fc_{sun}(j \Delta \phi),\ \forall j.
	\label{eq:discretized_HJB_L}
	\end{equation}
Using this operator, we write the corresponding discretized Poisson equation as the following:
	\begin{equation}
	\left(L_{\beta}^TM\right)_j=0,\ \forall j,
	\label{eq:discretized_Poisson}
	\end{equation}
	\begin{equation*}
	M_j\geq0,\ \forall j,
	\end{equation*}
	\begin{equation}
	\sum_{j=0}^{n-1}M_j \Delta \phi=1,
	\label{eq:sum_mu}
	\end{equation}
where $L_{\beta}^T$ denotes the transpose. The last thing we need is to make sure $\beta$ is a good approximation for the control, while preserving a monotone scheme. Thus, for each $j$, the approximation of $\beta_j$ is calculated from $U$ as follows:

Let $l_j=\omega_0-\omega_S-\frac{U_{j}-U_{j-1}}{\Delta \phi}$, and $r_j=\omega_0-\omega_S-\frac{U_{j+1}-U_{j}}{\Delta \phi}$.

\begin{itemize}

\item If $l_j<0$ and $r_j<0$:
	$$\beta_j=-\frac{U_{j}-U_{j-1}}{\Delta \phi}.$$
\item Else if $l_j>0$ and $r_j>0$:
	$$\beta_j=-\frac{U_{j+1}-U_{j}}{\Delta \phi}.$$
\item Else:
	$$\beta_j=0.$$
\end{itemize}

\subsection{Ergodic Mean Field Game Problem: Numerical Method 1}
\hspace{5mm} We implemented two numerical methods to solve the finite difference scheme for the \textit{Ergodic Mean Field Game Problem} proposed in the previous section. The first numerical method we implemented is the following:
\begin{enumerate}
	\item Initialize $M^0$ and $\beta^0$. In practice, we set $M_j^0=\frac{1}{2\pi}$, $\beta_j^0=0$, $j=0, \dots , n-1$.
	\item Given $M^k$ and $\beta^k$, solve the linear system of equations (\ref{eq:sum_V}) and (\ref{eq:discretized_HJB_L}) for $U^k$ and $\Lambda^k$.
	\item Given $U^k$, calculate new approximation for the control, $\beta^{k+1}$ as described at the end of Section \ref{sec:Ergodic Mean Field Game Problem: Finite Difference Equations}.
	\item Given $\beta^{k+1}$, solve the linear system of equations (\ref{eq:discretized_Poisson}) and (\ref{eq:sum_mu}) for $M^{k+1}$. Since there are more equations than unknowns, this system is overdetermined, so we use the least squares solution. Let $\epsilon^{k+1}$ denote the error in solving this system:
	\begin{equation*}
		\epsilon^{k+1}:=\sqrt{\sum_{j=0}^{n-1}\left(\left(L_{\beta}^TM\right)_j\right)^2+\left(1-\sum_{j=0}^{n-1}M_j \Delta \phi\right)^2}.
	\end{equation*}
	\item Repeat steps 2 through 4 until:
	\begin{equation*}
		\begin{split}
		\mathcal{W}_2(M^{k}\Delta \phi,M^{k+1}\Delta \phi)<&\epsilon, \\
		\sqrt{\sum_{j=0}^{n-1}\left(\beta_j^k-\beta_j^{k+1}\right)^2}<&\epsilon, \\
		|\Lambda^k-\Lambda^{k+1}|<&\epsilon, \\
		\epsilon^{k+1} < \epsilon,
		\end{split}
	\end{equation*}
	where $\mathcal{W}_2$ denotes the 2-Wasserstein distance.
\end{enumerate}
Note that convergence is not guaranteed. In fact, for some values of the parameters, this method does not converge. This is the motivation for trying a second method for solving the \textit{Ergodic Mean Field Game Problem}.

\subsection{Ergodic Mean Field Game Problem: Numerical Method 2}
\hspace{5mm} Since the previous method does not always converge, we implemented a second method to solve the posed finite difference scheme for the \textit{Ergodic Mean Field Game Problem}. In this method, we introduce an artificial time derivative. Thus, given $M^k$, $U^k$, and $\beta^k$, we explicitly calculate $U^{k+1}$ by:
	\begin{equation}
	U_j^{k+1}=U_j^k+\Delta t \left[\left(L_{\beta^k}U^k\right)_j+\frac{1}{2}(\beta^k_j)^2+K\bar{c}(j \Delta \phi,M^k)+Fc_{sun}(j \Delta \phi)\right].
	\label{eq:discretized_HJB_L_method_2}
	\end{equation}
Similarly, given $\beta^{k+1}$, we explicitly calculate $M^{k+1}$ by:
	\begin{equation}
	M_j^{k+1}=M_j^k+\Delta t \left(L_{\beta^{k+1}}^TM^k\right)_j.
	\label{eq:discretized_Poisson_method_2}
	\end{equation}
Bringing this together, the second numerical method is the following:
\begin{enumerate}
	\item Initialize $M^0$, $\beta^0$, and $U^0$. In practice, we set $M_j^0=\frac{1}{2\pi}$, $\beta_j^0=0$, $U_j^0=0$, $j=0, \dots , n-1$.
	\item Given $M^k$, $\beta^k$, and $U^k$, compute $U^{k+1}$ from equation (\ref{eq:discretized_HJB_L_method_2}).
	\item Given $U^{k+1}$, calculate new approximation for the control, $\beta^{k+1}$ as described at the end of Section \ref{sec:Ergodic Mean Field Game Problem: Finite Difference Equations}.
	\item Given $M^k$ and $\beta^{k+1}$, compute $M^{k+1}$ from equation (\ref{eq:discretized_Poisson_method_2}).
	\item Repeat steps 2 through 4 until
	\begin{equation*}
	\begin{split}
	\mathcal{W}_2(M^{k}\Delta \phi,M^{k+1}\Delta \phi)<&\epsilon, \\
	\sqrt{\sum_{j=0}^{n-1}\left(\beta_j^k-\beta_j^{k+1}\right)^2}<&\epsilon.
	\end{split}
	\end{equation*}
\end{enumerate}
As for the previous method, convergence is not guaranteed. If the method converges, then $M^k \approx M^{k+1}$ and $\beta^k \approx \beta^{k+1}$, and the solution approximately solves the original discretized equations before adding the artificial time dependence. Note that since this is an explicit method, we need a CFL stability condition on the relationship between $\Delta t$ and $\Delta \phi$:
\begin{equation}
\Delta t \leq \frac{1}{2(\sigma^2/(\Delta \phi)^2+(|\omega_S-\omega_0|+A)/\Delta \phi)},
\label{eq:CFL}
\end{equation}
where $A$ is a bound on $|\beta|$. Since $A$ is not known a priori, we guess a value for $A$, apply the algorithm, and check if $|\beta^k_{j}| \leq A,\ \forall j,k$. If not, a larger value of $A$ is used, and the process is repeated until a sufficiently large $A$ is chosen. The numerical results are presented in Section \ref{sec:Numerical_Results}. Next, we describe the numerical methods for the \textit{Recovery via Ergodic Problem}.
	
\subsection{Recovery via Ergodic Problem: Numerical Method}
\hspace{5mm} We approximate the solution $\mu_p(t,\phi)$ to the \textit{Recovery via Ergodic Problem} by $M^p=\{M_{i,j}\}_{0 \leq i < m, 0 \leq j < n}$ where $M^p_{i,j} \approx \mu_p(i \Delta t, j \Delta \phi)$. The time step size, $\Delta t$, is chosen to satisfy the CFL condition as given in equation (\ref{eq:CFL}), where $A=\max_{j}|\beta_j|$.

Since we assume the oscillators adopt the control $\alpha^*(\phi-p)$ and we have the approximation $\beta_j \approx \alpha^*(j \Delta \phi)$, we need to rotate the control by $n \cdot p/(2\pi)$ grid points. We conveniently take $n=120$ so that $n \cdot p/(2\pi)$ is an integer for $|p|/\omega_S=1,\cdots,12$ (i.e. $n$ is a multiple of $24$.). Let $\beta^p_j:=\beta_{j-r}\approx \alpha^*(j \Delta \phi-p)$ where $r:=n \cdot p/(2\pi)$.

We use an explicit scheme by using a forward difference for the time derivative in the Kolmogorov/ Fokker-Planck equation (since we have an initial condition). Thus, the Kolmogorov/ Fokker-Planck equation becomes:
	\begin{equation}
	M^p_{i+1,j}=M^p_{i,j}+\Delta t \left(L_{\beta^p}^TM^p_i\right)_j,\ \forall i,j
	\label{eq:discretized_Kolm},
	\end{equation}
with initial condition $M^p_{0,j}=M_j$, $j=0, \dots, n-1$. Note that because of the choice of using a monotone scheme to define the operator $L$, if $M^p_{0,j} \geq0,\ \forall j$, then $M^p_{i,j} \geq0,\ \forall i,j$. Also note that since the rows of $L$ sum to $0$, we also have that $\sum_{j=0}^{n-1}M^p_{0,j}\Delta \phi=1$ implies $\sum_{j=0}^{n-1}M^p_{i,j}\Delta \phi=1,\ \forall i$. The numerical method for solving the \textit{Recovery via Ergodic Problem} is straightforward: compute equation (\ref{eq:discretized_Kolm}) forward in time. Next, we describe the numerical methods for the \textit{Recovery Mean Field Game Problem}.

\subsection{Recovery Mean Field Game Problem: Finite Difference Equations}
\hspace{5mm} We approximate the solution $(\tilde{\mu}_p(t,\phi),\tilde{V}_p(t,\phi))$ to the \textit{Recovery Mean Field Game Problem} by

$(\tilde{M}^p=\{\tilde{M}^p_{i,j}\}_{0 \leq i < m, 0 \leq j < n}$, $\tilde{U}^p=\{\tilde{U}^p_{i,j}\}_{0 \leq i < m, 0 \leq j < n})$ where $\tilde{M}^p_{i,j} \approx \tilde{\mu}_p(i\Delta t,j \Delta \phi)$, and $\tilde{U}^p_{i,j} \approx \tilde{V}_p(i\Delta t,j \Delta \phi)$. The time step is $\Delta t = \frac{T }{m}$, where $m$ is the number of grid points in time. Again, $\Delta t$, is chosen to satisfy the CFL condition as given in equation (\ref{eq:CFL}). To set notation, whenever only one subscript appears, we are referring to the time index. (E.g. $\tilde{U}^p_{i}=\{ \tilde{U}^p_{i,j}\}_{0 \leq j < n}$.) We again make use of the following approximation of the control: $\tilde{\beta}^p_{i,j} \approx \tilde{\alpha}_p(i \Delta t,j \Delta \phi)=-\partial_\phi \tilde{V}_p(i \Delta t,j\Delta \phi)$.

We use an explicit scheme. Since we have a terminal condition for the HJB, we use a backwards difference for the time derivative. Thus, the HJB equation becomes:
\begin{equation}
\tilde{U}^p_{i,j}=\tilde{U}^p_{i+1,j}+\Delta t \left[(L_{\tilde{\beta}^p_{i+1}}\tilde{U}^p_{i+1})_j+\frac{1}{2}(\tilde{\beta}^p_{i+1,j})^2+K\bar{c}(j \Delta \phi,\tilde{M}^p_{i+1}(\cdot))+Fc_{sun}(j \Delta \phi)\right],
\label{eq:discretized_HJB_recovery_MFG}
\end{equation}
with terminal condition $\tilde{U}^p_{m,j}=0$, $j=0, \dots, n-1$. Similarly, since we have an initial condition for the Kolmogorov/Fokker-Planck equation, we use a forward difference for the time derivative. The Kolmogorov/Fokker-Planck equation becomes:
\begin{equation}
\tilde{M}^p_{i+1,j}=\tilde{M}^p_{i,j}+\Delta t (L_{\tilde{\beta}^p_{i}}^T\tilde{M}^p_i)_j,
\label{eq:discretized_Poisson_recovery_MFG}
\end{equation}
with initial condition $\tilde{M}^p_{0,j}=M_j$, $j=0, \dots, n-1$. As for the previous model, the initial condition guarantees that $\tilde{M}^p_{i,j} \geq0, \forall i,j$ and $\sum_{j=0}^{n-1}\tilde{M}^p_{i,j}\Delta \phi=1, \forall i$. For each time step $i$, $\tilde{\beta}^p_{i}$ is computed from $\tilde{U}^p_{i}$ in the same way as described at the end of Section \ref{sec:Ergodic Mean Field Game Problem: Finite Difference Equations}.

\subsection{Recovery Mean Field Game Problem: Numerical Method}
\hspace{5mm} The numerical method for solving the \textit{Finite Horizon Problem} is the following:
\begin{enumerate}
	\item Initialize $\tilde{M}^{p,0}$ and $\tilde{\beta}^{p,0}$. In practice, we set $\tilde{M}^{p,0}_{i,j}=M_j$, $\tilde{\beta}^{p,0}_{i,j}=0$, $i=0, \dots, m-1$, $j=0, \dots, n-1$.
	\item Given $\tilde{M}^{p,k}$ and $\tilde{\beta}^{p,k}$, compute $\tilde{U}^{p,k}$ backward in time from equation (\ref{eq:discretized_HJB_recovery_MFG}).
	\item Given $\tilde{U}^{p,k}$, calculate new approximation for the control, $\tilde{\beta}^{p,k+1}$.
	\item Given $\tilde{\beta}^{p,k+1}$, compute $\tilde{M}^{p,k+1}$ forward in time from equation (\ref{eq:discretized_Poisson_recovery_MFG}).
	\item Repeat steps 2 through 4 until
	\begin{equation*}
	\begin{split}
	\max_i \sqrt{\sum_{j=0}^{n-1}\left(\tilde{M}^{p,k}_{i,j}\Delta \phi-\tilde{M}^{p,k+1}_{i,j}\Delta \phi\right)^2}<&\epsilon, \\
	\max_i \sqrt{\sum_{j=0}^{n-1}\left(\tilde{\beta}^{p,k}_{i,j}-\tilde{\beta}^{p,k+1}_{i,j}\right)^2}<&\epsilon.
	\end{split}
	\end{equation*}
\end{enumerate}
Note that the $L^2$ distance is used between $\tilde{M}^{p,k}_{i}$ and $\tilde{M}^{p,k+1}_{i}$, since the 2-Wasserstein distance is expensive and $m$, the number of time steps, is large. As in \textit{Method 2} for the \textit{Ergodic Mean Field Game Problem}, this is an explicit method and we need to use the stability condition in equation (\ref{eq:CFL}).

\section{Numerical Results}\label{sec:Numerical_Results}
\hspace{5mm} For our numerical results, we fix $\omega_S=\frac{2\pi}{24}$ radians per hour. The number of grid points in $\phi$ is also fixed at $n=120$. The parameters of interest are thus $p$, $\omega_0$, $\sigma$, $K$, and $F$. Results are presented for a reference set of values for these parameters as well as changing each parameter individually while holding the other parameters at the reference values. For our reference set, we let $p=\pm 9 \omega_S$ (travel east or west by 9 time zones), $\omega_0=\frac{2\pi}{24.5}$ radians per hour, $\sigma=0.1$, $K=0.01$, and $F=0.01$.

\subsection{Ergodic Mean Field Game Problem: Convergence of Numerical Methods}
\hspace{5mm} As previously mentioned, the proposed iterative schemes are not guaranteed to converge. In this section, we explore the convergence of the two proposed methods for the \textit{Ergodic Mean Field Game Problem}, both when using the monotone scheme proposed above as well as using a centered scheme (i.e. replacing the one sided differences with centered differences in the definition of $L_{\beta}$ from equation (\ref{eq:discretized_HJB})), for different values of the weight parameters, $K$ and $F$.

Note that if $F=0$, we have the model introduced in \cite{yin2012synchronization}, in which it is shown that there is not uniqueness. For values of $K \geq 0$ and $F >0$, we report if \textit{Method 1} converges within one thousand iterations and if \textit{Method 2} converges within one million iterations. We let $\epsilon=10^{-5}$ in the convergence definitions.

For some values of the parameters, the solution `converges' to a $\mu$ which is concentrated on a value far from the value $p=0$. This `convergence' is up to $\epsilon=10^{-5}$ as defined in the numerical methods descriptions. Since this is clearly not a solution to the problem when $F>0$, we check if the `average' phase is near 0. More precisely, we compute the order parameter $z$ defined in Section \ref{sec:Jet_Lag_Recovery}. Then we compute the angle $\psi \in [-\pi,\pi]$ that $z$ makes with the real line in the complex plane. We require $|\psi|<0.1$ to consider the solution valid. When we use the centered scheme instead of the monotone scheme for defining first derivatives, some values of $\mu_j$ are negative. Since small negative values are permissible, we require $\mu_j>-\epsilon$ for $\epsilon=10^{-5}$ to consider the solution valid. Thus there are three possible outcomes: converges (as defined in the descriptions of the algorithms in Section \ref{sec:Numerical_Methods}), converges but not to a valid solution, and does not converge.

Figure \ref{fig:convergence} shows the convergence results for both methods and schemes, when the other parameters remain at the reference values stated above. From these plots, we cannot conclude that any single method is the best. In fact, there are only two points, $(K,F)=(0.001, 0.01)$ and $(K,F)=(0.01, 0.01)$, for which all four methods/schemes converge. For this reason, we set $(K,F)=(0.01, 0.01)$ for our reference values of the parameters. Where both methods converge, \textit{Method 1} and \textit{Method 2} produce the same results when using the same scheme (monotone or centered). The main difference between the solutions for the monotone and centered schemes is the monotone scheme solution is not smooth near $\phi=0$. Figure \ref{fig:smoothness} shows the solutions for the monotone and centered schemes. Because the solution is smooth, we chose to use the centered scheme for the remainder of the numerical results.

\begin{figure}
	\begin{subfigure}{.5\textwidth}
		\centering
		\includegraphics[scale=0.5]{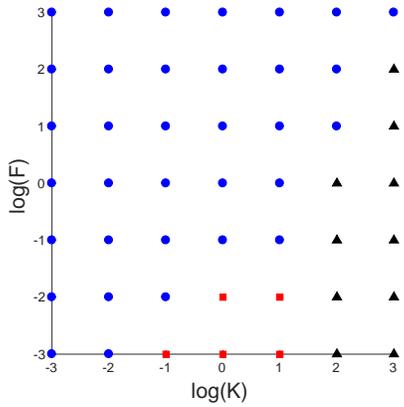}
		\caption{Method 1, Monotone Scheme}
	\end{subfigure}
	\begin{subfigure}{.5\textwidth}
		\centering
		\includegraphics[scale=0.5]{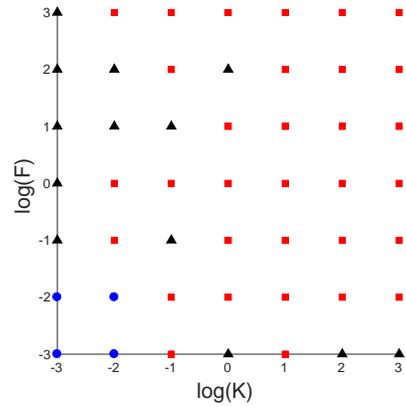}
		\caption{Method 1, Centered Scheme}
	\end{subfigure}
	\begin{subfigure}{.5\textwidth}
		\centering
		\includegraphics[scale=0.5]{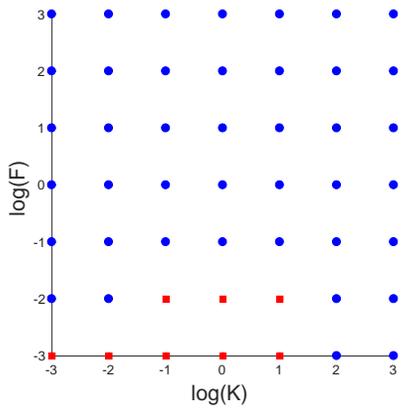}
		\caption{Method 2, Monotone Scheme}
	\end{subfigure}
	\begin{subfigure}{.5\textwidth}
		\centering
		\includegraphics[scale=0.5]{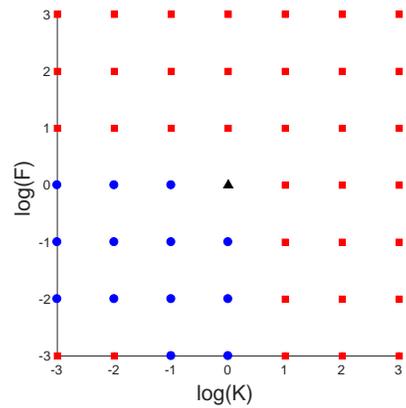}
		\caption{Method 2, Centered Scheme}
	\end{subfigure}
	\caption{\textit{Ergodic Mean Field Game Problem}: Convergence results for different values of $(K,F)$. Blue circle means convergence. Black triangle means convergence but not to a valid solution. Red square means did not converge. Note that the axes are $log_{10}$.}
	\label{fig:convergence}
\end{figure}

\begin{figure}
	\begin{subfigure}{.5\textwidth}
		\centering
		\includegraphics[scale=0.5]{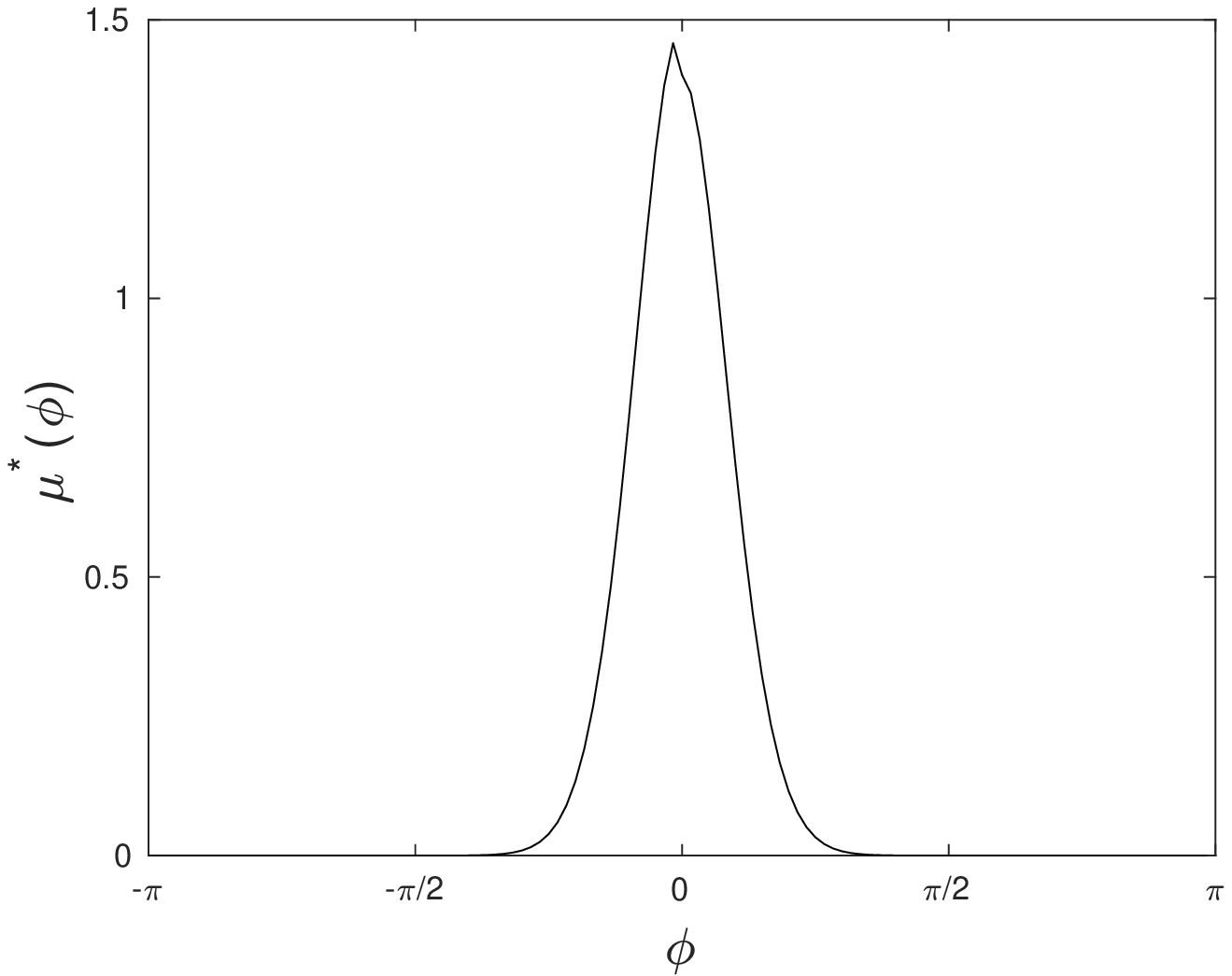}
		\caption{$\mu^*(\phi)$ for monotone scheme}
	\end{subfigure}
	\begin{subfigure}{.5\textwidth}
		\centering
		\includegraphics[scale=0.5]{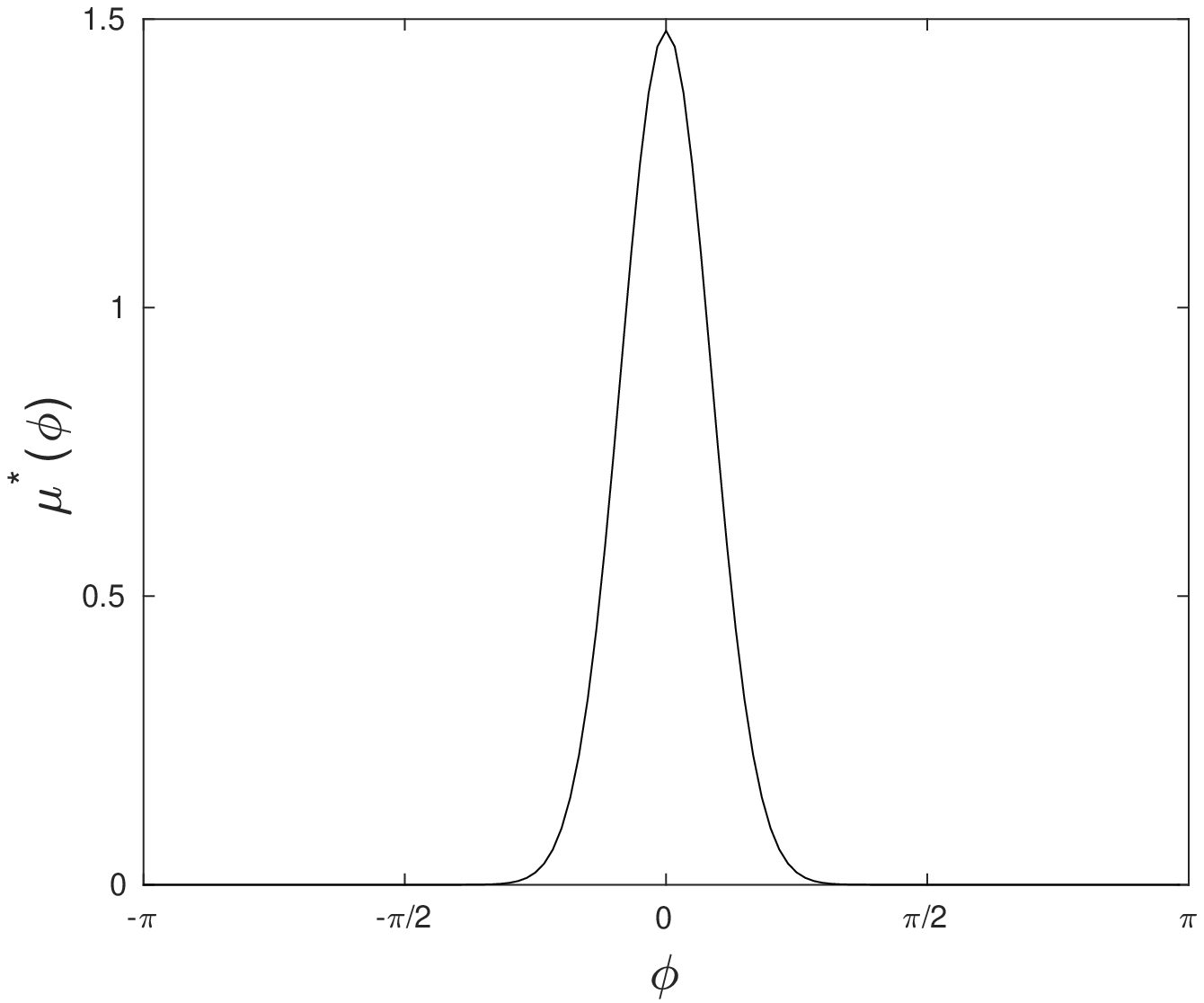}
		\caption{$\mu^*(\phi)$ for centered scheme}
	\end{subfigure}
	\caption{Comparison of solutions to \textit{Ergodic Mean Field Game Problem} with different schemes for the reference set of parameters. Note that the monotone scheme solution is not smooth near $\phi=0$.}
	\label{fig:smoothness}
\end{figure}

\subsection{Ergodic Mean Field Game Problem Numerical Solution}
\hspace{5mm} For our reference set of parameters, the solution to the \textit{Ergodic Mean Field Game Problem} when using \textit{Method 1} with a centered scheme is shown in Figure \ref{fig:Stationary_Reference}. Note that the measure is concentrated near $p=0$ (recall that the domain is periodic). The stationary solution is used to calculate the optimal control $\alpha^*(\phi)$. The drift of an oscillator using the optimal control is a function of $\phi$, shown in Figure \ref{fig:Stationary_Reference_Drift}. Note that near $\phi=0$, if $\phi<0$, meaning the oscillator is lagging behind the phase of the natural 24 hour cycle, then the drift is positive. This allows the oscillator to advance it's phase forward to try to `catch up.' (Since we are on a periodic domain, we state this more precisely as the drift is positive if $\phi \in (\pi+\epsilon,2\pi)$ for a small $\epsilon>0$. The intuition is not as clear for this statement, however.) Similarly, if $\phi \in (0,\pi]$, the drift is negative, meaning that the oscillator delays it's phase.

		\begin{figure}
			\begin{subfigure}{.5\textwidth}
				\centering
				\includegraphics[scale=0.5]{mu_star_reference_set_centered.eps}
				\caption{$\mu^*(\phi)$}
			\end{subfigure}
			\begin{subfigure}{.5\textwidth}
				\centering
				\includegraphics[scale=0.5]{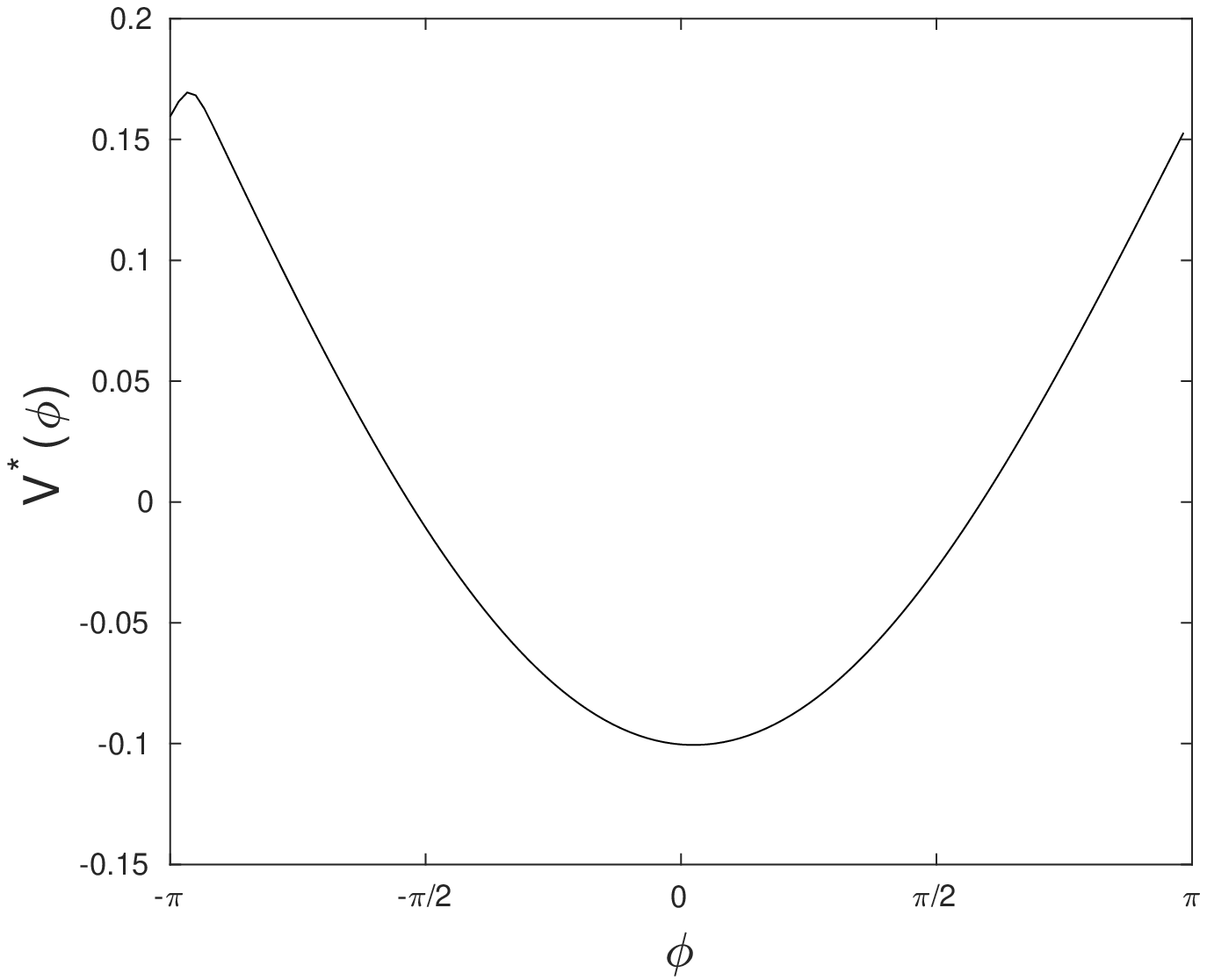}
				\caption{$V^*(\phi)$}
			\end{subfigure}
			\caption{Solution to \textit{Ergodic Mean Field Game Problem} for the reference set of parameters.}
			\label{fig:Stationary_Reference}
		\end{figure}

		\begin{figure}
				\centering
				\includegraphics[scale=0.5]{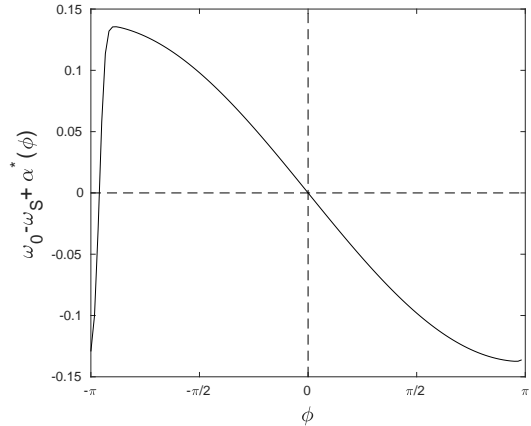}
				\caption{Drift for the solution to the \textit{Ergodic Mean Field Game Problem} ($\omega_0-\omega_S+\alpha^*(\phi))$ for the reference set of parameters.}
			\label{fig:Stationary_Reference_Drift}
		\end{figure}

\subsection{Recovery via Ergodic Problem Numerical Solution}
\hspace{5mm} Using the numerical solution to the \textit{Ergodic Mean Field Game Problem} as the initial condition and control, the solution to the \textit{Recovery via Ergodic Problem} is shown in Figures \ref{fig:Reference_East} and \ref{fig:Reference_West} for traveling $9$ time zones east and west, respectively. The solution is shown at time $0$, after $1$ day, after $2$ days, and after $3$ days to illustrate the gradual adjustment to the new time zone angle $p$, which is plotted as a vertical line.

	\begin{figure}
		\begin{subfigure}{.5\textwidth}
			\centering
			\includegraphics[scale=0.5]{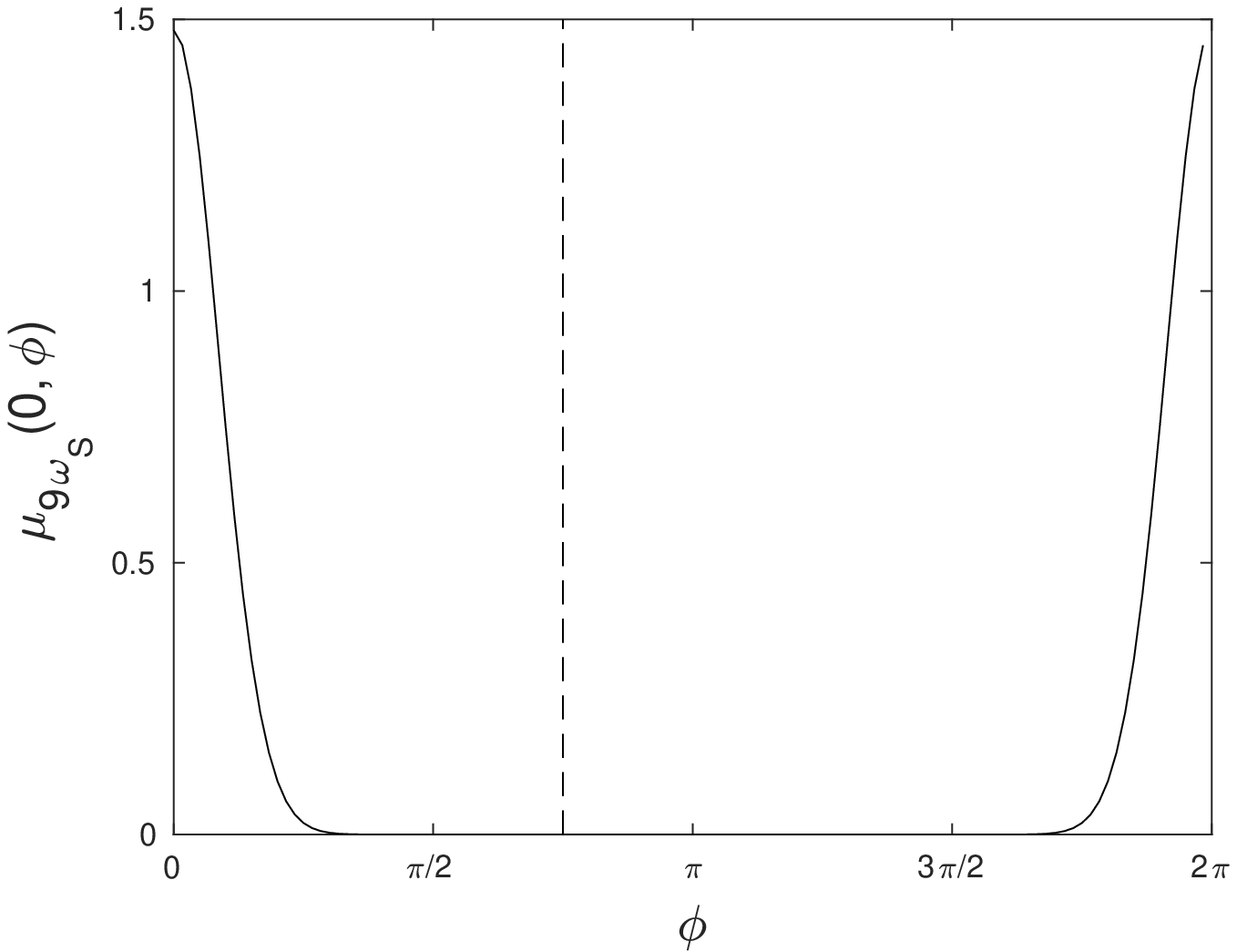}
			\caption{$t=0$}
		\end{subfigure}
		\begin{subfigure}{.5\textwidth}
			\centering
			\includegraphics[scale=0.5]{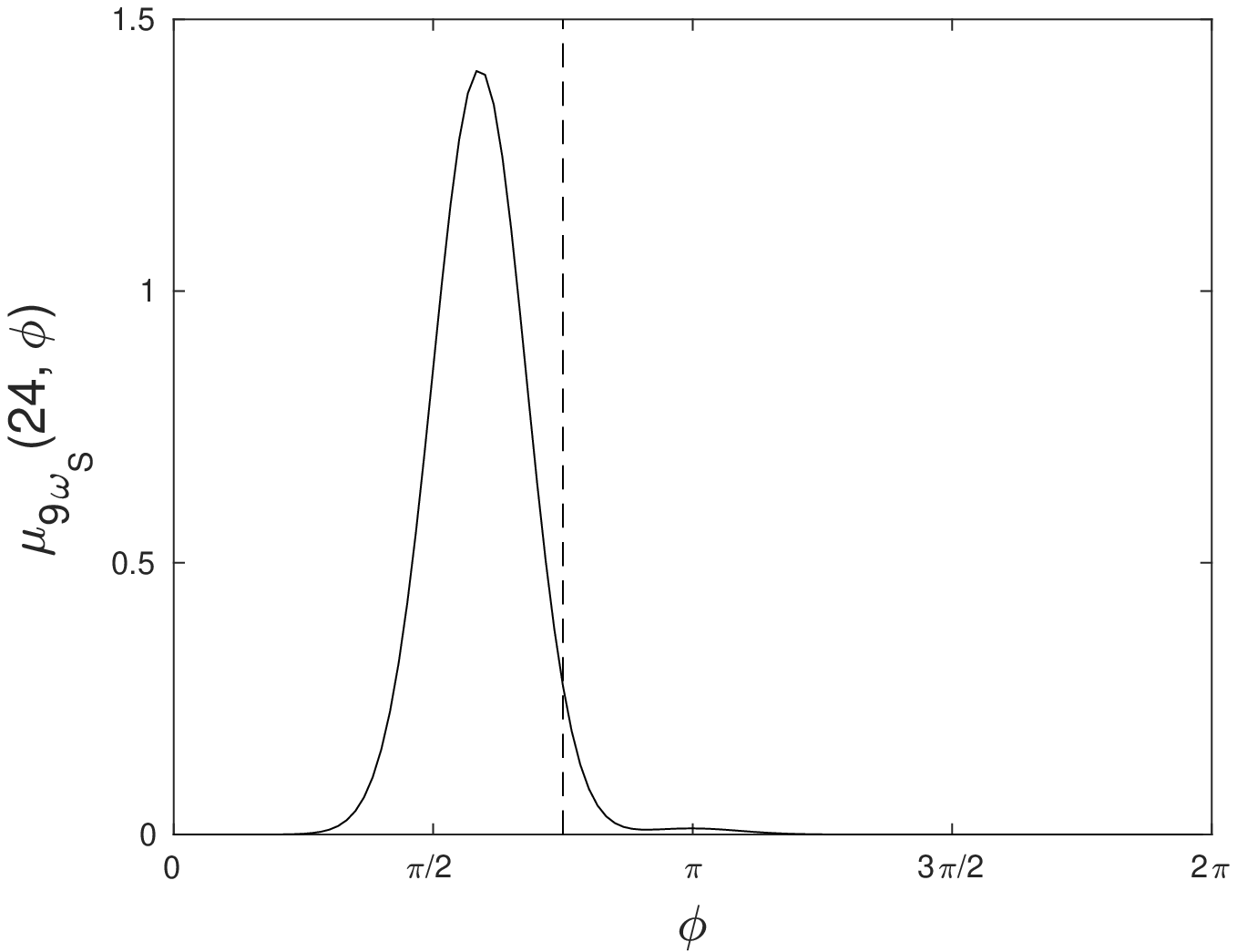}
			\caption{$t=1$ day}
		\end{subfigure}
		\\
		\begin{subfigure}{.5\textwidth}
			\centering
			\includegraphics[scale=0.5]{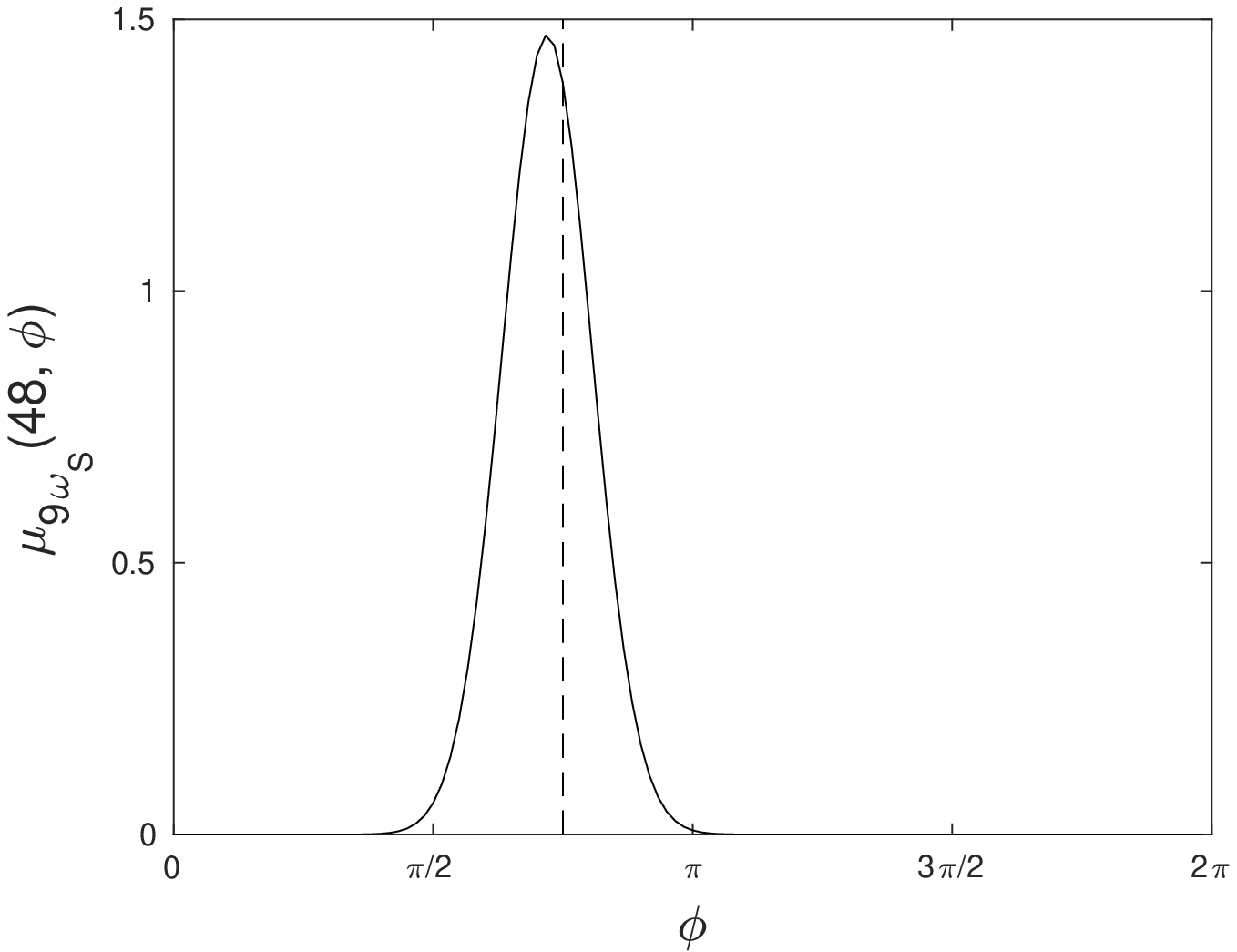}
			\caption{$t=2$ days}
		\end{subfigure}
		\begin{subfigure}{.5\textwidth}
			\centering
			\includegraphics[scale=0.5]{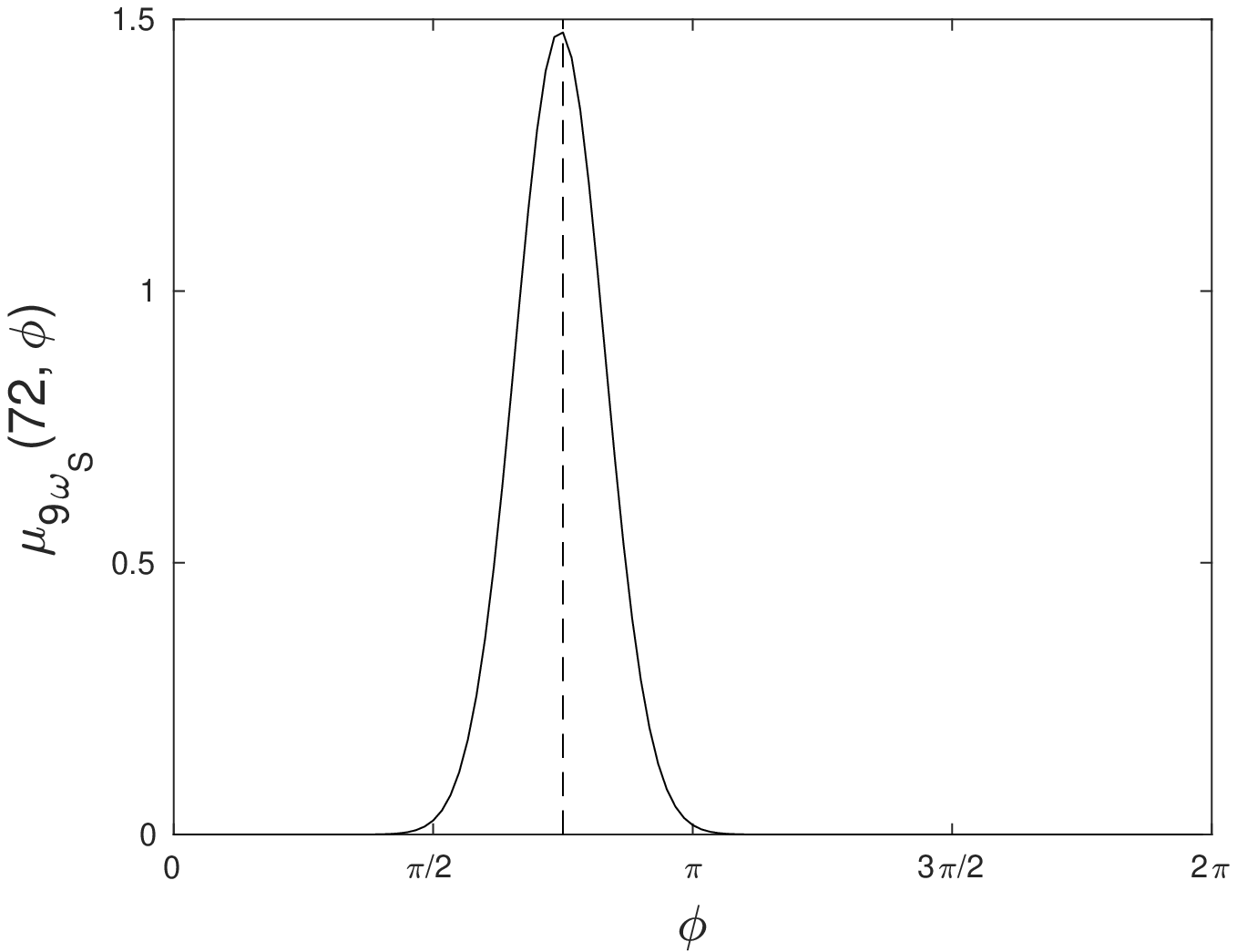}
			\caption{$t=3$ days}
		\end{subfigure}
		\caption{\textit{Recovery via Ergodic Problem}: Jet lag recovery after eastern travel by 9 time zones for the reference set of parameters.}
		\label{fig:Reference_East}
	\end{figure}

	\begin{figure}
		\begin{subfigure}{.5\textwidth}
			\centering
			\includegraphics[scale=0.5]{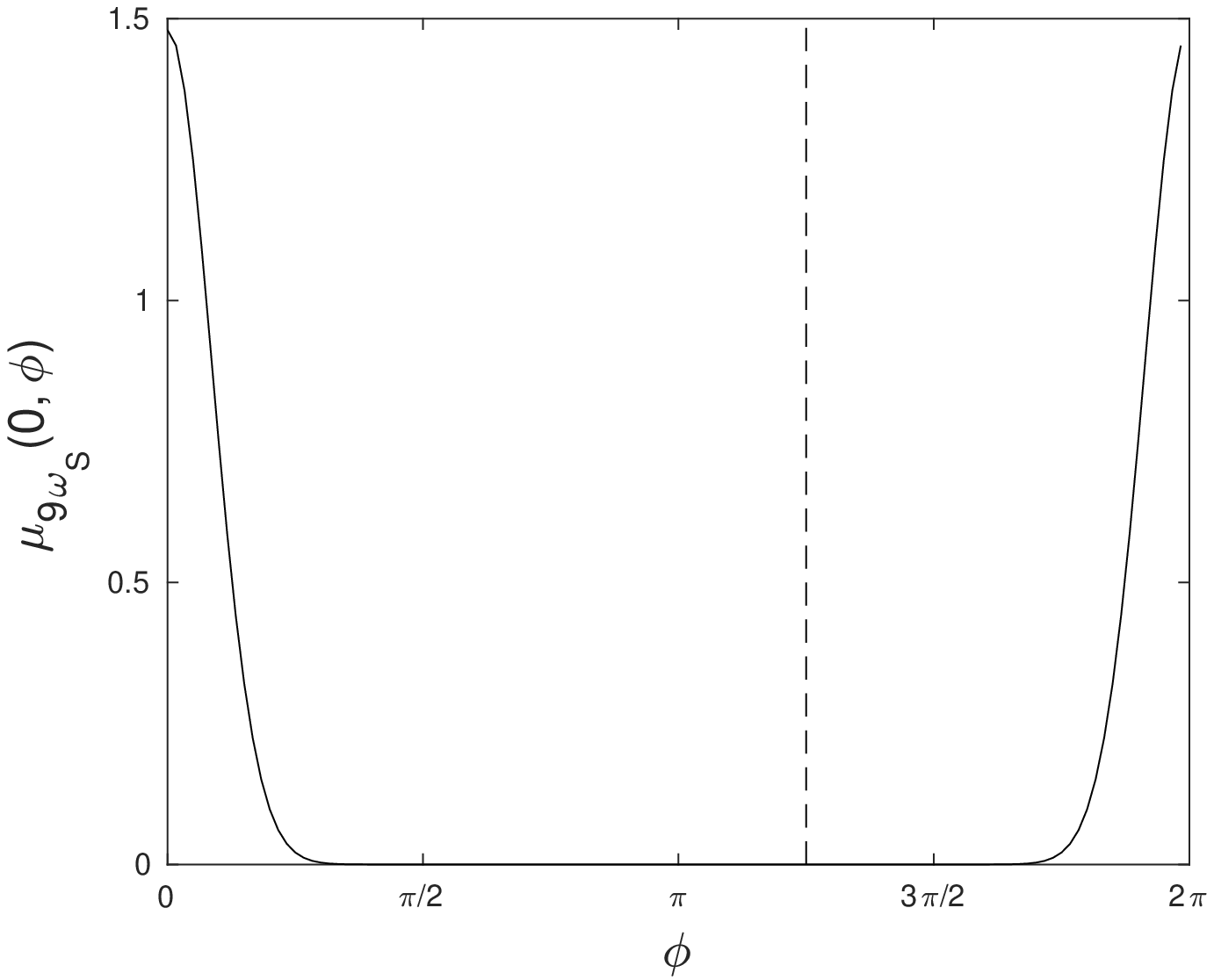}
			\caption{$t=0$}
		\end{subfigure}
		\begin{subfigure}{.5\textwidth}
			\centering
			\includegraphics[scale=0.5]{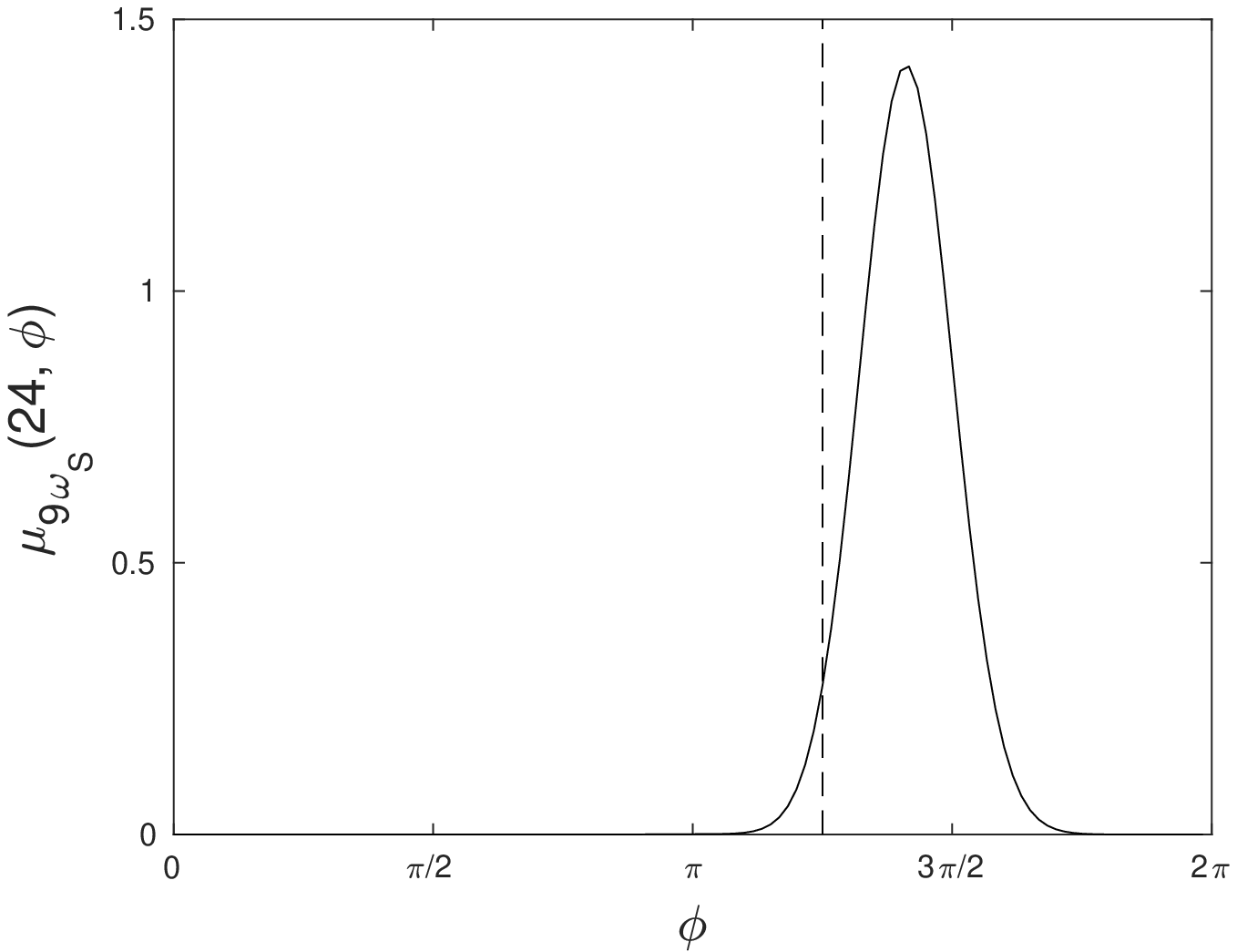}
			\caption{$t=1$ day}
		\end{subfigure}
		\\
		\begin{subfigure}{.5\textwidth}
			\centering
			\includegraphics[scale=0.5]{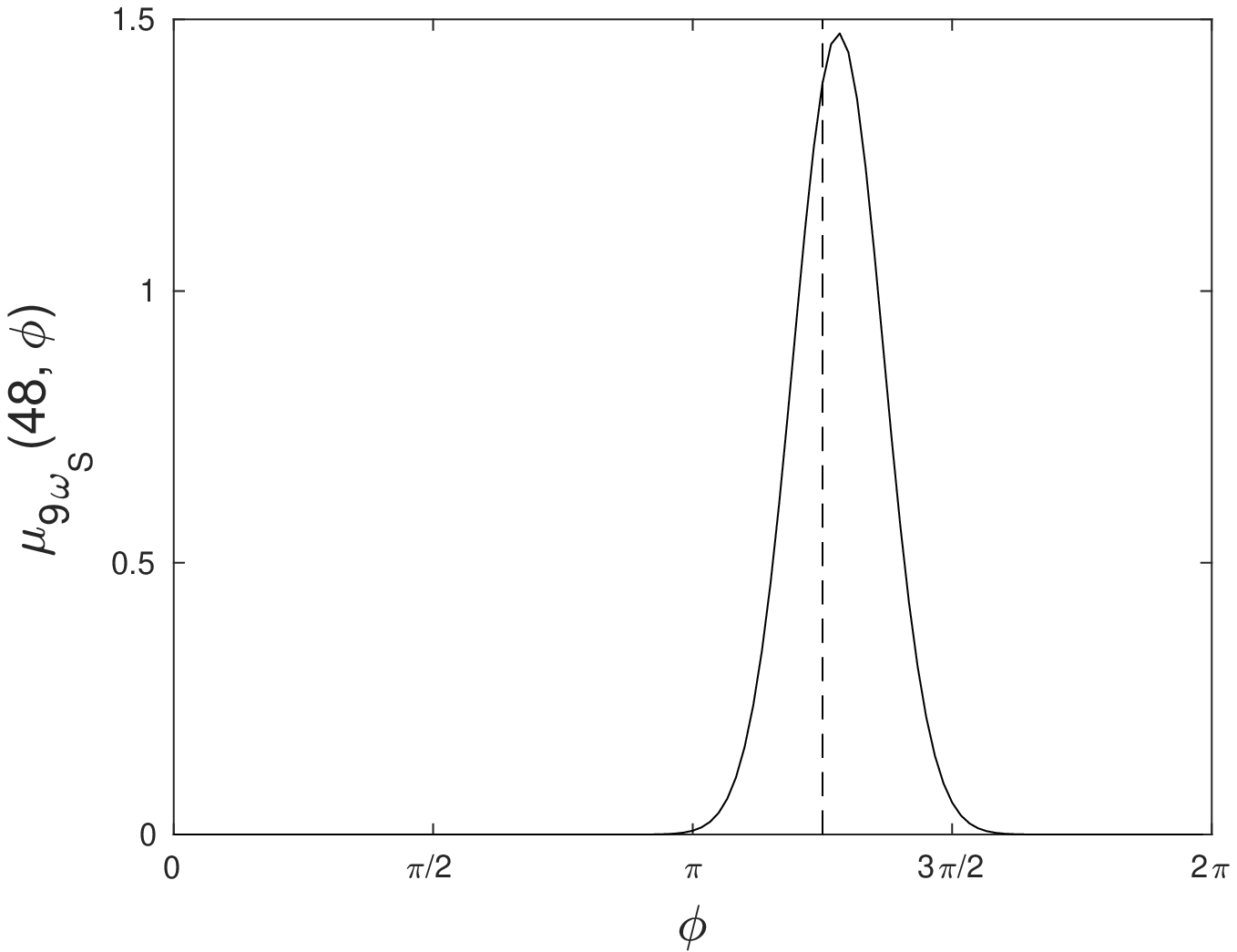}
			\caption{$t=2$ days}
		\end{subfigure}
		\begin{subfigure}{.5\textwidth}
			\centering
			\includegraphics[scale=0.5]{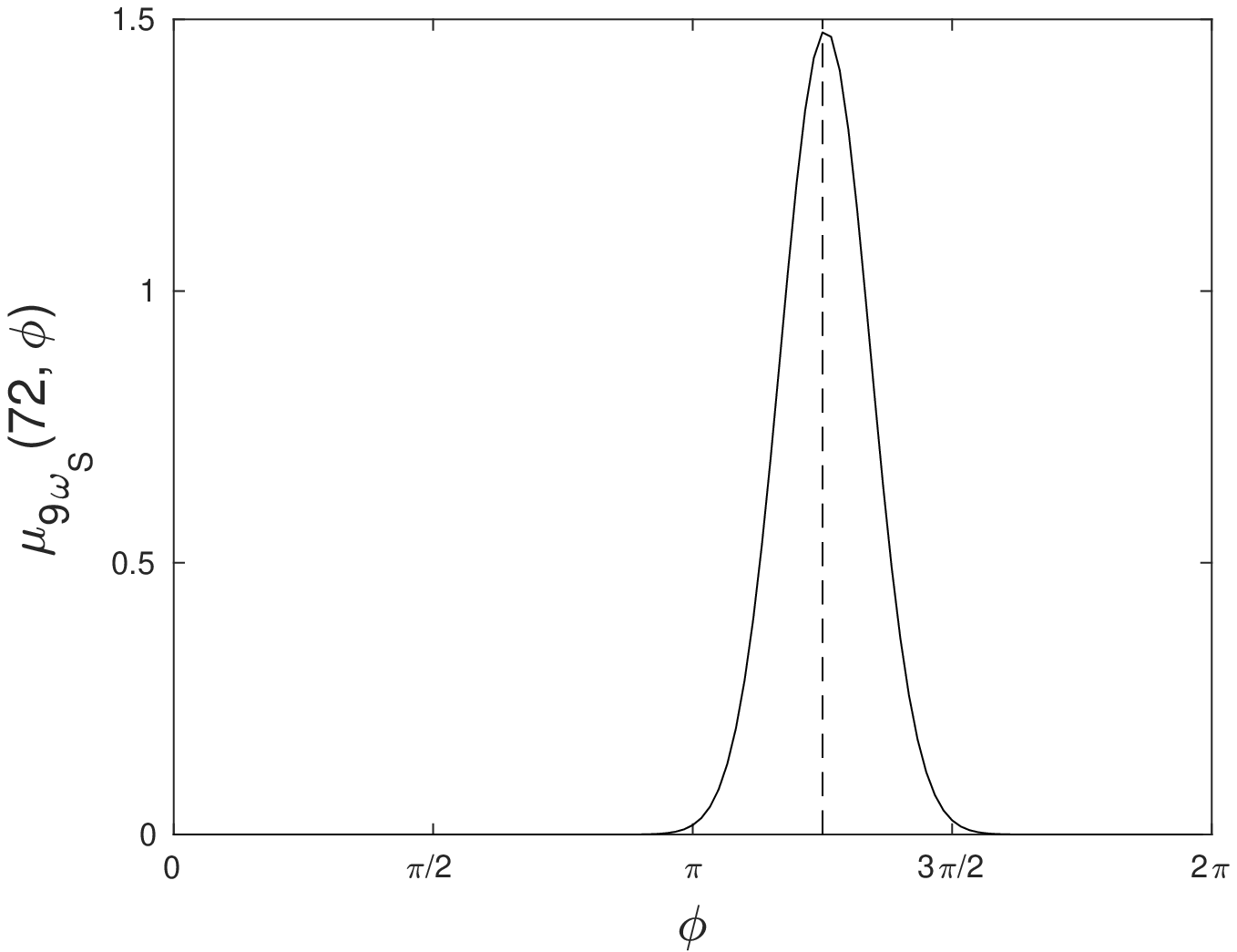}
			\caption{$t=3$ days}
		\end{subfigure}
		\caption{\textit{Recovery via Ergodic Problem}: Jet lag recovery after western travel by 9 time zones for the reference set of parameters.}
		\label{fig:Reference_West}
	\end{figure}

As shown in the plots, after a few days, the distribution of the oscillators has adjusted itself to align with the new value of $p$. Now we compare our measures of jet lag recovery between east and west travel to test the claim that jet lag is worse when traveling east.

The jet lag recovery times based on the 2-Wasserstein distance are $\tau^W_{9\omega_S}=4.38$ days and $\tau^W_{-9\omega_S}=4.42$ days, where we take $\epsilon^W=0.01$. The jet lag recovery times based on the order parameter $z$ are $\tau^z_{9\omega_S}=1.54$ days and $\tau^z_{-9\omega_S}=1.58$ days, where we take the same value of $\epsilon^z=0.2$ as in \cite{lu2016resynchronization}. Thus, jet lag recovery time is about the same for east and west travel for the reference set of parameters. The jet lag recovery costs are shown in Figure \ref{fig:Reference_Costs}. The cost from the controls, $f_{\alpha}(t)$, and the cost of synchronizing with the other oscillators, $f_{osc}(t)$, are larger for the eastward trip than the westward trip, while the cost for the 24 hour cycle, $f_{sun}(t)$, is the same for east and west. Thus, the total recovery cost is larger when recovering from traveling east. To summarize, recovery time is about the same for east and west travels, but there is a larger recovery cost associated with traveling east.

	\begin{figure}
		\begin{subfigure}{.5\textwidth}
			\centering
			\includegraphics[scale=0.5]{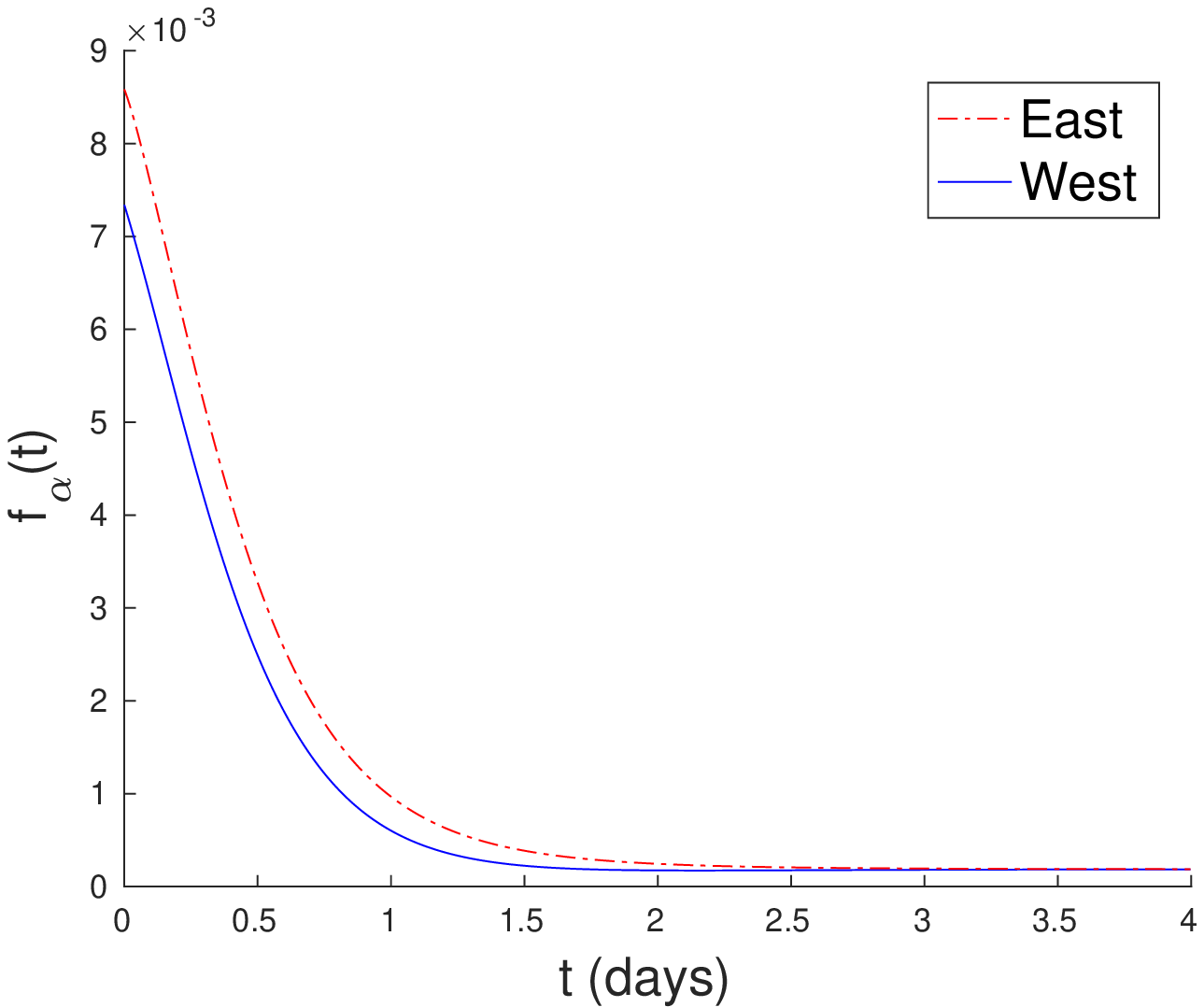}
			\caption{$f_{\alpha}(t)$}
		\end{subfigure}
		\begin{subfigure}{.5\textwidth}
			\centering
			\includegraphics[scale=0.5]{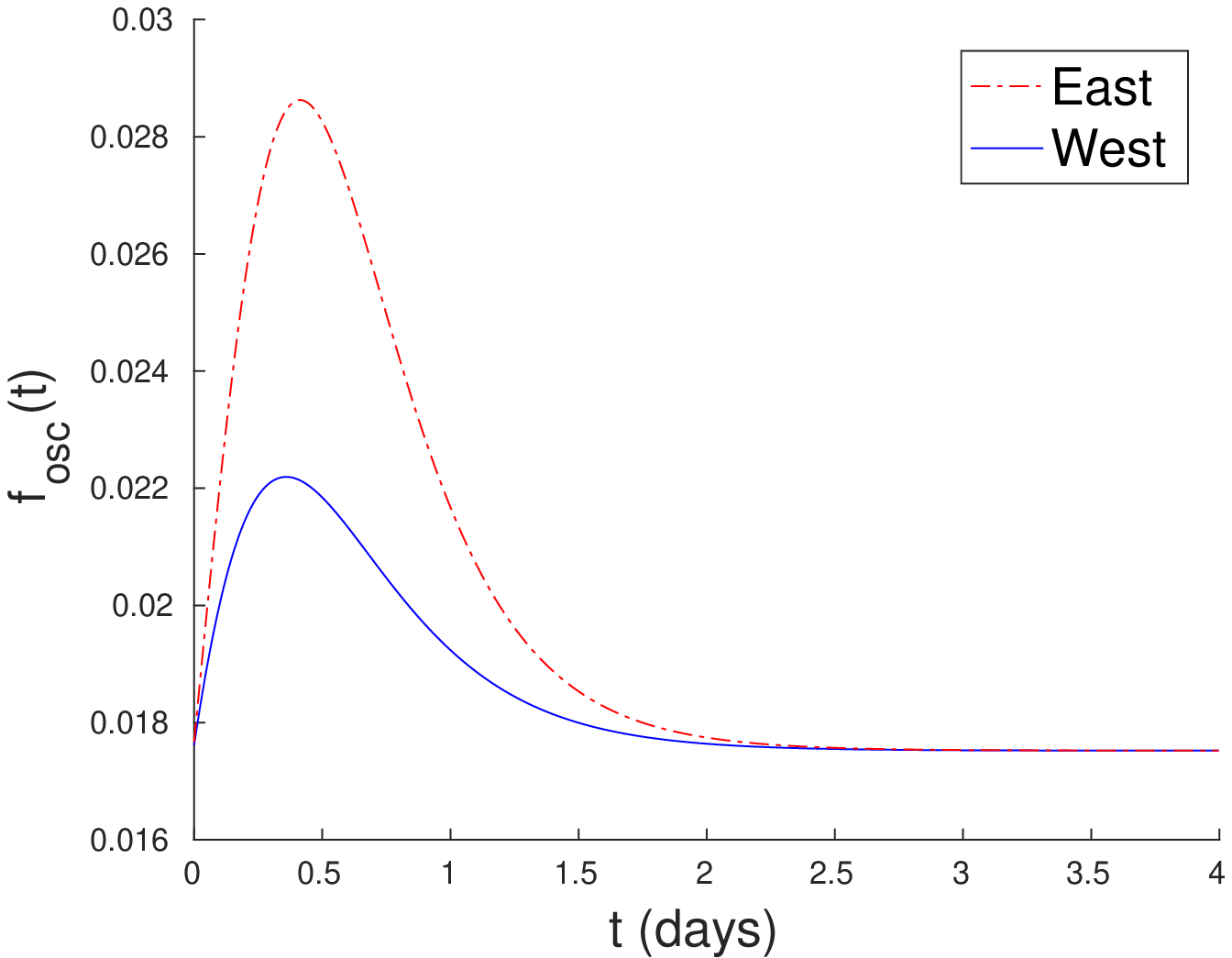}
			\caption{$f_{osc}(t)$}
		\end{subfigure}
		\\
		\begin{subfigure}{.5\textwidth}
			\centering
			\includegraphics[scale=0.5]{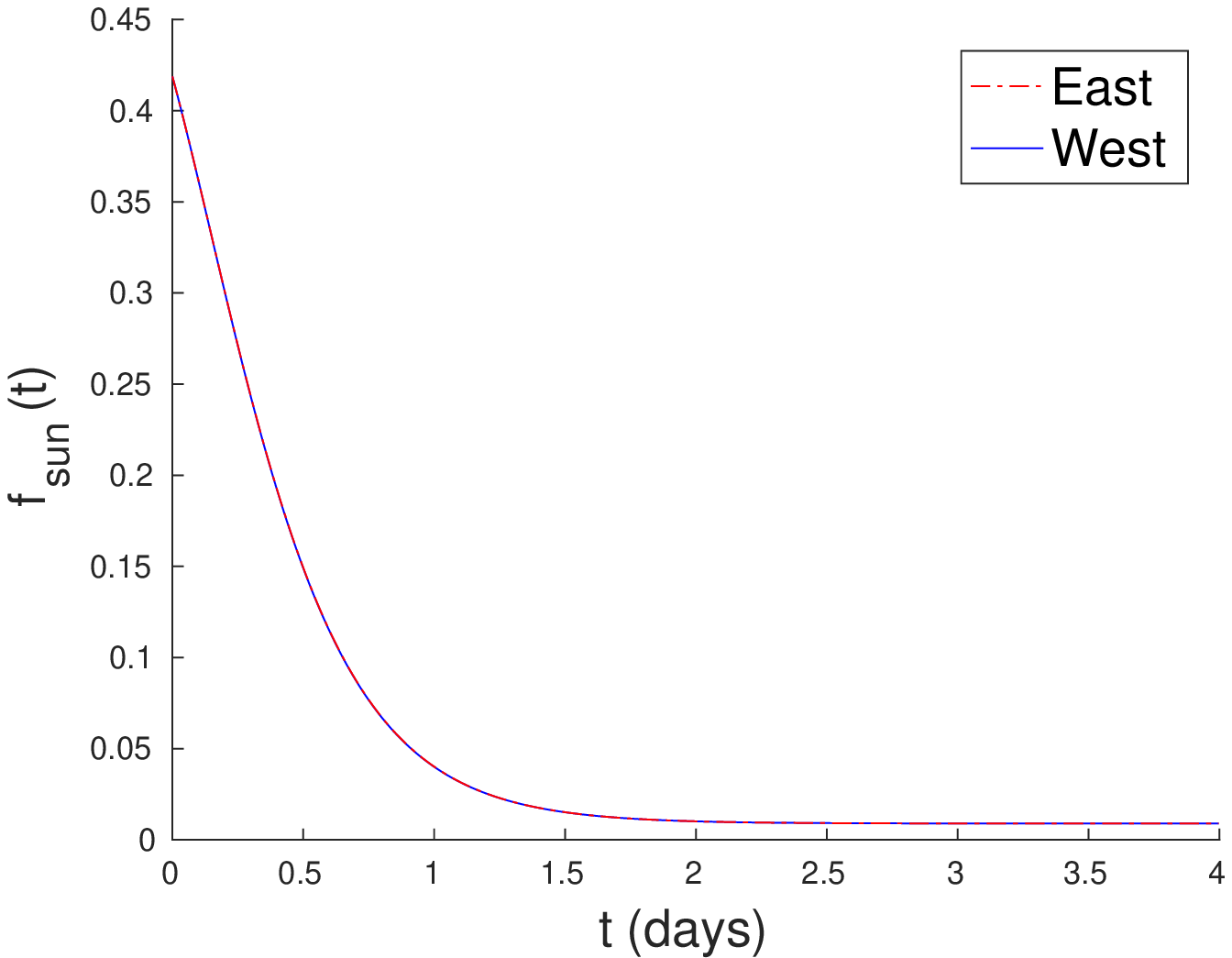}
			\caption{$f_{sun}(t)$}
		\end{subfigure}
		\begin{subfigure}{.5\textwidth}
			\centering
			\includegraphics[scale=0.5]{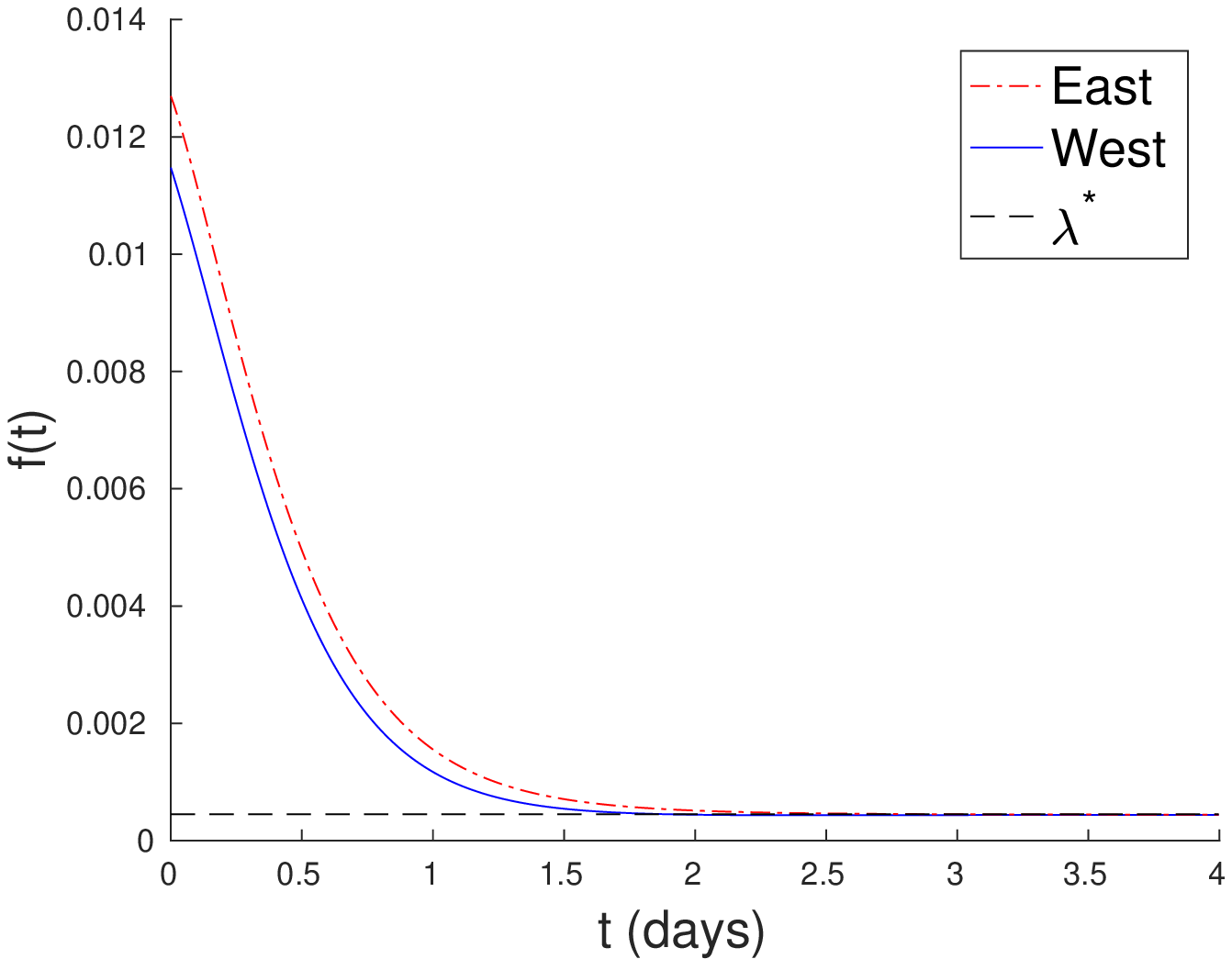}
			\caption{$f(t)$}
		\end{subfigure}
		\caption{\textit{Recovery via Ergodic Problem}: Jet lag recovery costs for travel by 9 time zones for the reference set of parameters.}
		\label{fig:Reference_Costs}
	\end{figure}

For comparison with Lu et al., Figure \ref{fig:Reference_z_t} shows the paths $z_p(t)$ for $p= \pm 9 \omega_S$. Recall that a larger $|z_p(t)|$ means that the oscillators are more synchronized with each other. The results are similar to Figure 2(a) in \cite{lu2016resynchronization}, except the stationary solution $z^*$ for our approach is closer to $(1,0)$ on the complex plane. Note that the presentation is slightly different in that we take $\rho(0)=0$ and $\rho(t)=p$ for $t>0$ whereas in \cite{lu2016resynchronization}, they take $\rho(0)=p$ and $\rho(t)=0$ for $t>0$.

		\begin{figure}
			\centering
			\includegraphics[scale=0.5]{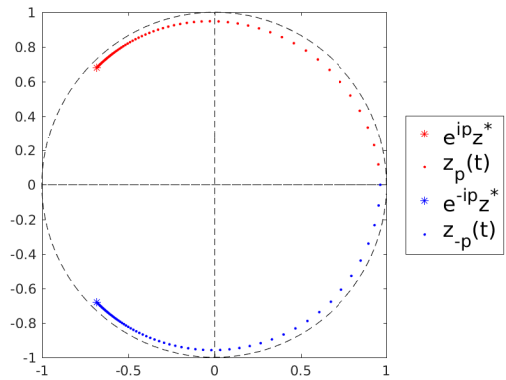}
			\caption{\textit{Recovery via Ergodic Problem}: Path $z_p(t)$ while recovering from jet lag after traveling $9$ time zones east (red) and west (blue) for the reference set of parameters. A point is plotted every hour.}
			\label{fig:Reference_z_t}
		\end{figure}

\subsection{Parameter Sensitivity Analysis for the Recovery via Ergodic Problem}
\hspace{5mm} Now we explore how the results change with the five parameters of interest: $p$, $\omega_0$, $\sigma$, $K$, and $F$. Unless otherwise specified, parameters remain at their reference values of $p=\pm 9 \omega_S$, $\omega_0=2\pi/24.5$, $\sigma=0.1$, $K=0.01$, and $F=0.01$. We are somewhat restricted because we can only obtain results for the \textit{Recovery via Ergodic Problem} when our numerical algorithm for solving the \textit{Ergodic Mean Field Game Problem} converges (otherwise, we do not have an initial condition and a control to feed into the \textit{Recovery via Ergodic Problem}). In particular, when using the centered scheme with \textit{Method 1}, the values of $K$ and $F$ need to be sufficiently small for the algorithm to converge (hence, our choice of $K=0.01$, and $F=0.01$ in the reference set). In addition, we should not consider values of $F$ which are too small, since when $F=0$, it is known that uniqueness does not hold \cite{yin2012synchronization}. Despite these limitations, we still get a picture of the behavior of the solutions as we vary each parameter.

Figure \ref{fig:Changing_p} shows the two jet lag recovery times, $\tau_p^W$ and $\tau_p^z$, as we vary $p$. The east and west recovery times almost completely coincide, except for deviations at $|p| \geq 9 \omega_S$. The costs accrued over the first 10 days are summarized in Figure \ref{fig:Changing_p_costs}. The most interesting result is Figure \ref{fig:Changing_p_costs}(a) which shows a larger cost for eastward trips with an increasing disparity between east and west as $|p|$ increases from $1\omega_S$ through $11\omega_S$. Note that in our model, a trip of $12$ time zones is the same if viewed as an eastward or a westward trip, and the results for $p= \pm 12 \omega_S$ coincide.

		\begin{figure}
			\centering
			\includegraphics[scale=0.5]{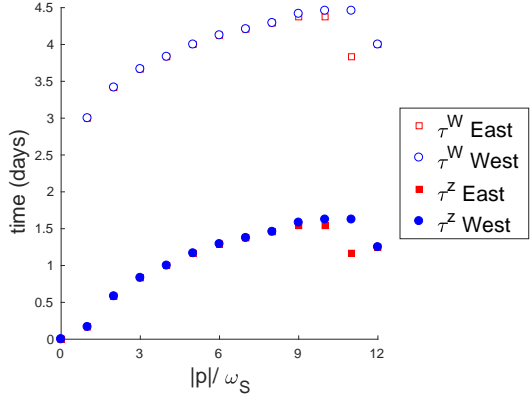}
			\caption{\textit{Recovery via Ergodic Problem}: Jet lag recovery times as a function of $p$.}
			\label{fig:Changing_p}
		\end{figure}
		
	\begin{figure}
		\begin{subfigure}{.5\textwidth}
			\centering
			\includegraphics[scale=0.5]{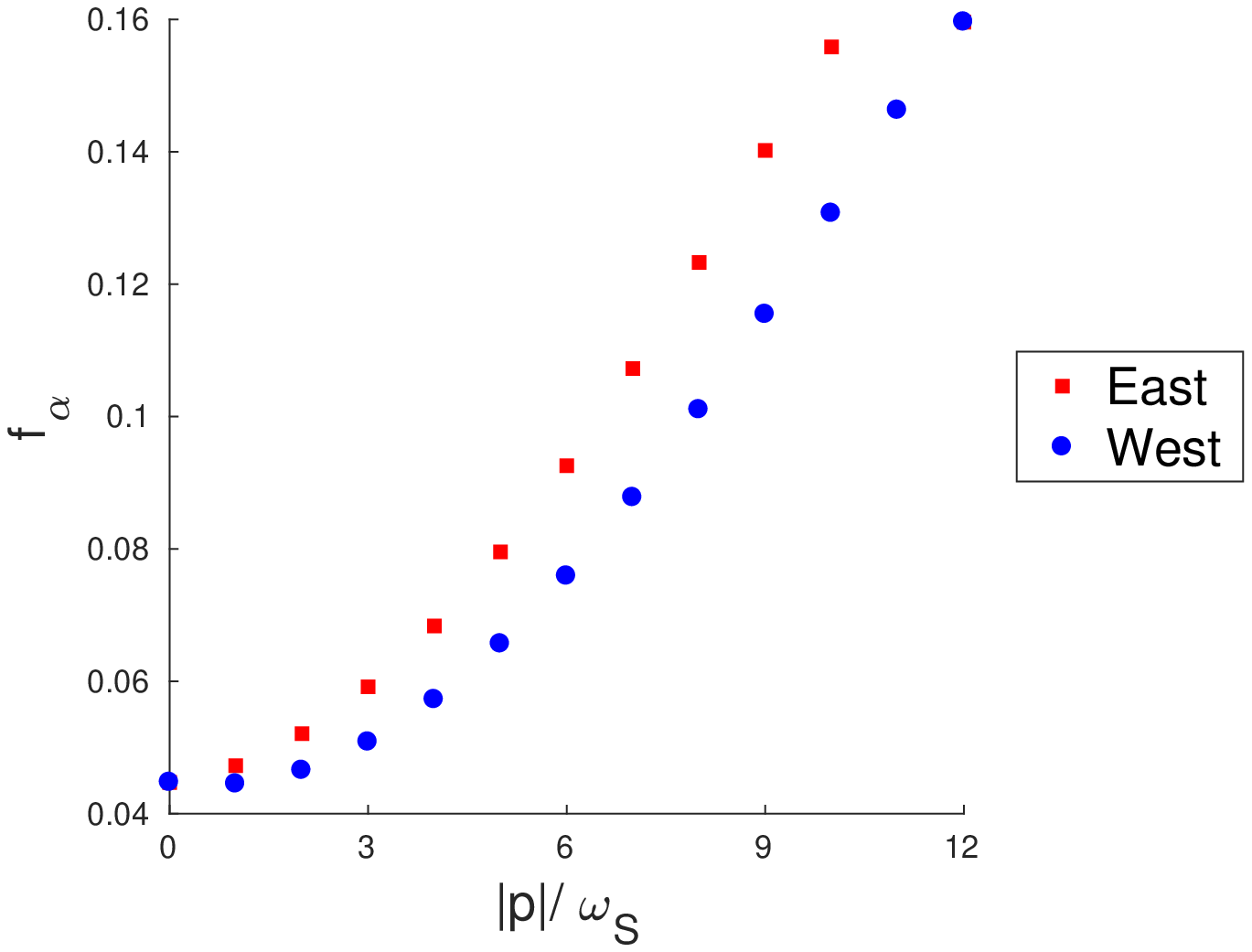}
			\caption{$f_{\alpha}(t)$}
		\end{subfigure}
		\begin{subfigure}{.5\textwidth}
			\centering
			\includegraphics[scale=0.5]{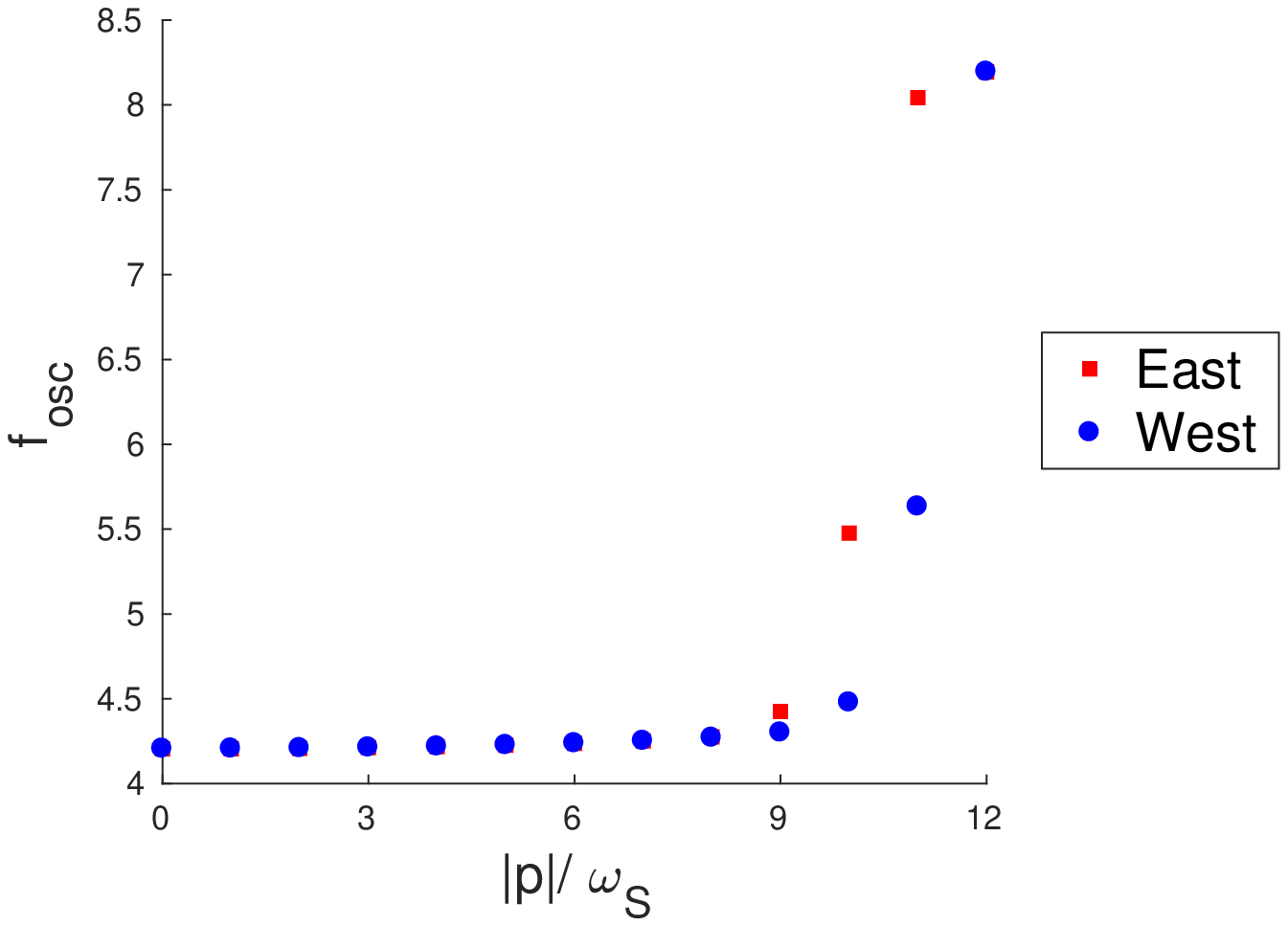}
			\caption{$f_{osc}(t)$}
		\end{subfigure}
		\\
		\begin{subfigure}{.5\textwidth}
			\centering
			\includegraphics[scale=0.5]{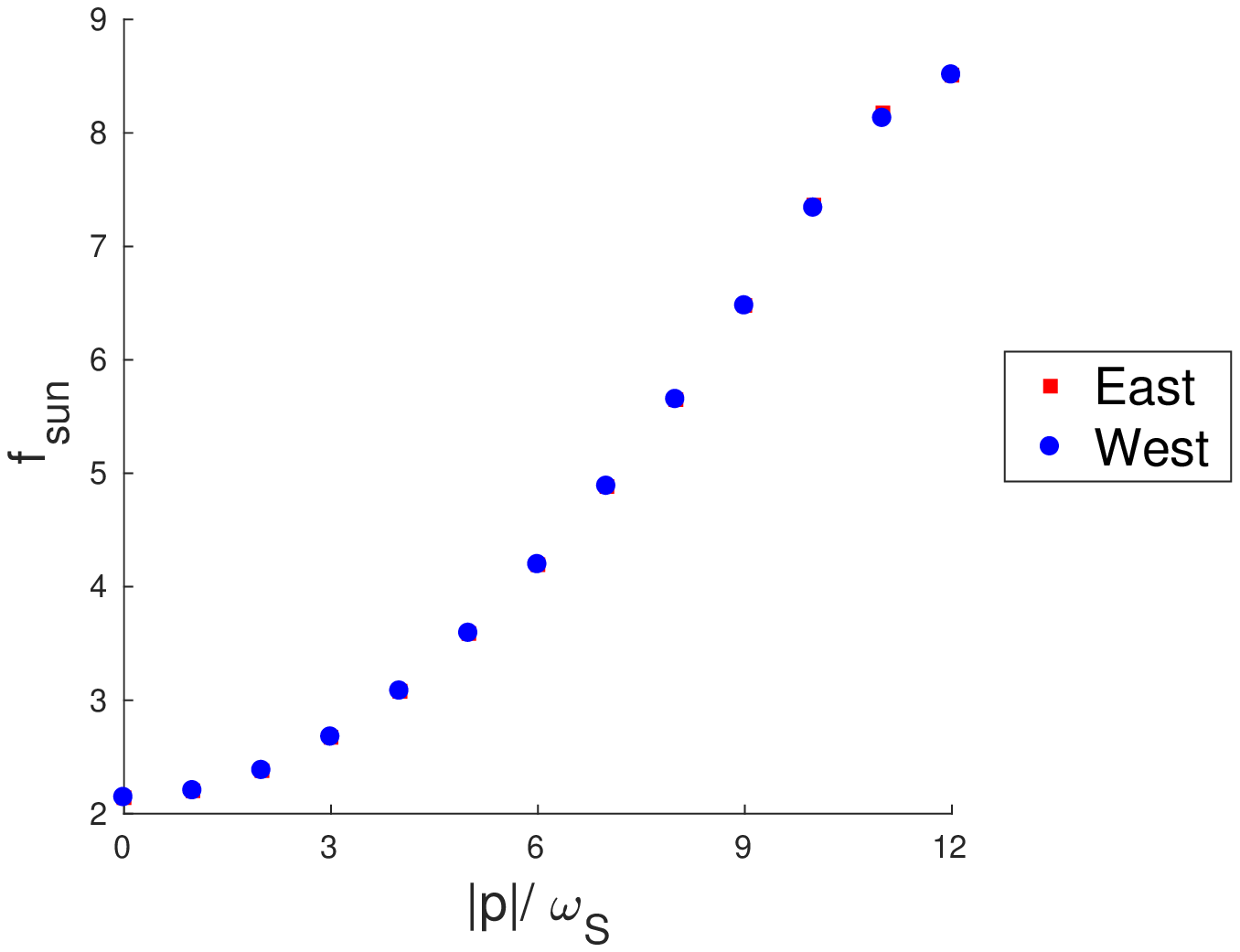}
			\caption{$f_{sun}(t)$}
		\end{subfigure}
		\begin{subfigure}{.5\textwidth}
			\centering
			\includegraphics[scale=0.5]{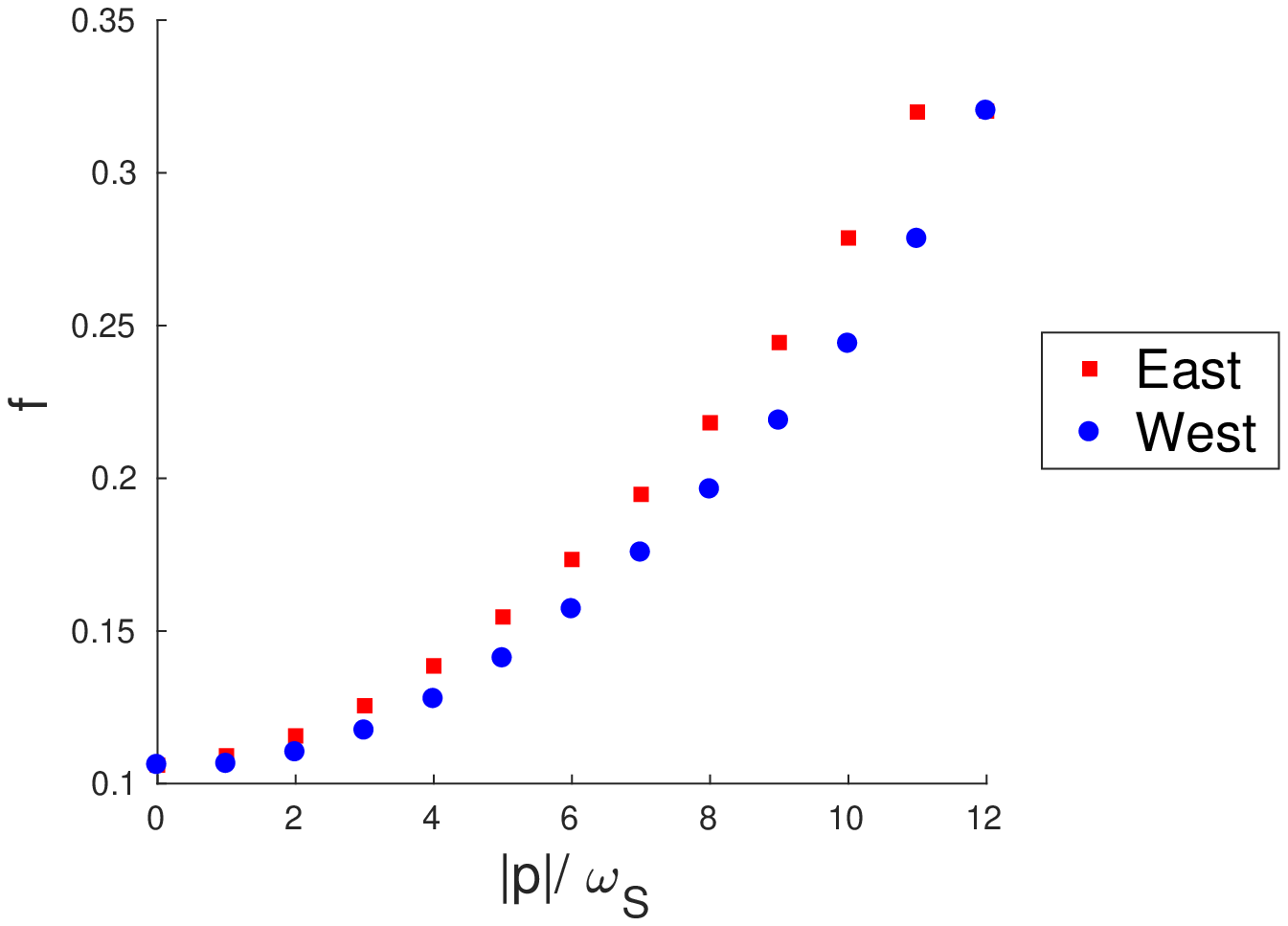}
			\caption{$f(t)$}
		\end{subfigure}
		\caption{\textit{Recovery via Ergodic Problem}: Jet lag recovery costs as a function of $p$. Note that a trip of $12$ time zones is the same when viewed as an eastward or westward trip.}
		\label{fig:Changing_p_costs}
	\end{figure}
	
The results of changing $\omega_0$ are shown in Figures \ref{fig:Changing_omega} and \ref{fig:Changing_omega_costs}. As for the previous results, we see that the recovery time is about the same for east and west travels, but the cost is larger for east travels when $\omega_0<\omega_S$. By symmetry, we expect that if $\omega_0-\omega_S=\Omega$ and $p=+j\omega_S$, the results should be the same as when $\omega_0-\omega_S=-\Omega$ and $p=-j\omega_S$. This is confirmed in Figures \ref{fig:Changing_omega} and \ref{fig:Changing_omega_costs}.

Returning to our analysis of Figure \ref{fig:Changing_omega}, we note that the recovery time is fairly stable as we change $\omega_0$ to different values near $\omega_S$, and recovery times increase for values of $\omega_0$ far from $\omega_S$. The general trend in Figure \ref{fig:Changing_omega_costs} is that the jet lag recovery costs increase as $\omega_0$ goes further away from $\omega_S$. A surprising result, however, is that for the extremal values of $\omega_0=2\pi/18$ or $\omega_0=2\pi/36$, the cost from the control, $f_{\alpha}$, decreases.
	
		\begin{figure}
			\centering
			\includegraphics[scale=0.5]{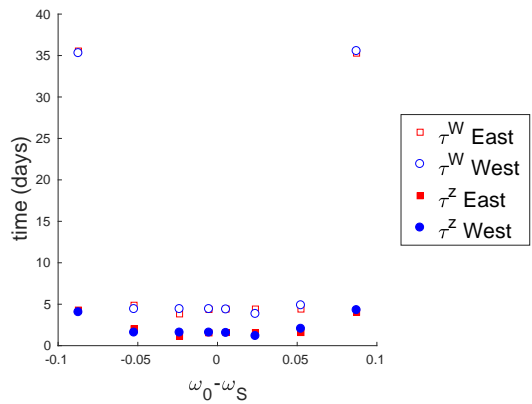}
			\caption{\textit{Recovery via Ergodic Problem}: Jet lag recovery times as a function of $\omega_0$.}
			\label{fig:Changing_omega}
		\end{figure}
		
		\begin{figure}
			\begin{subfigure}{.5\textwidth}
				\centering
				\includegraphics[scale=0.5]{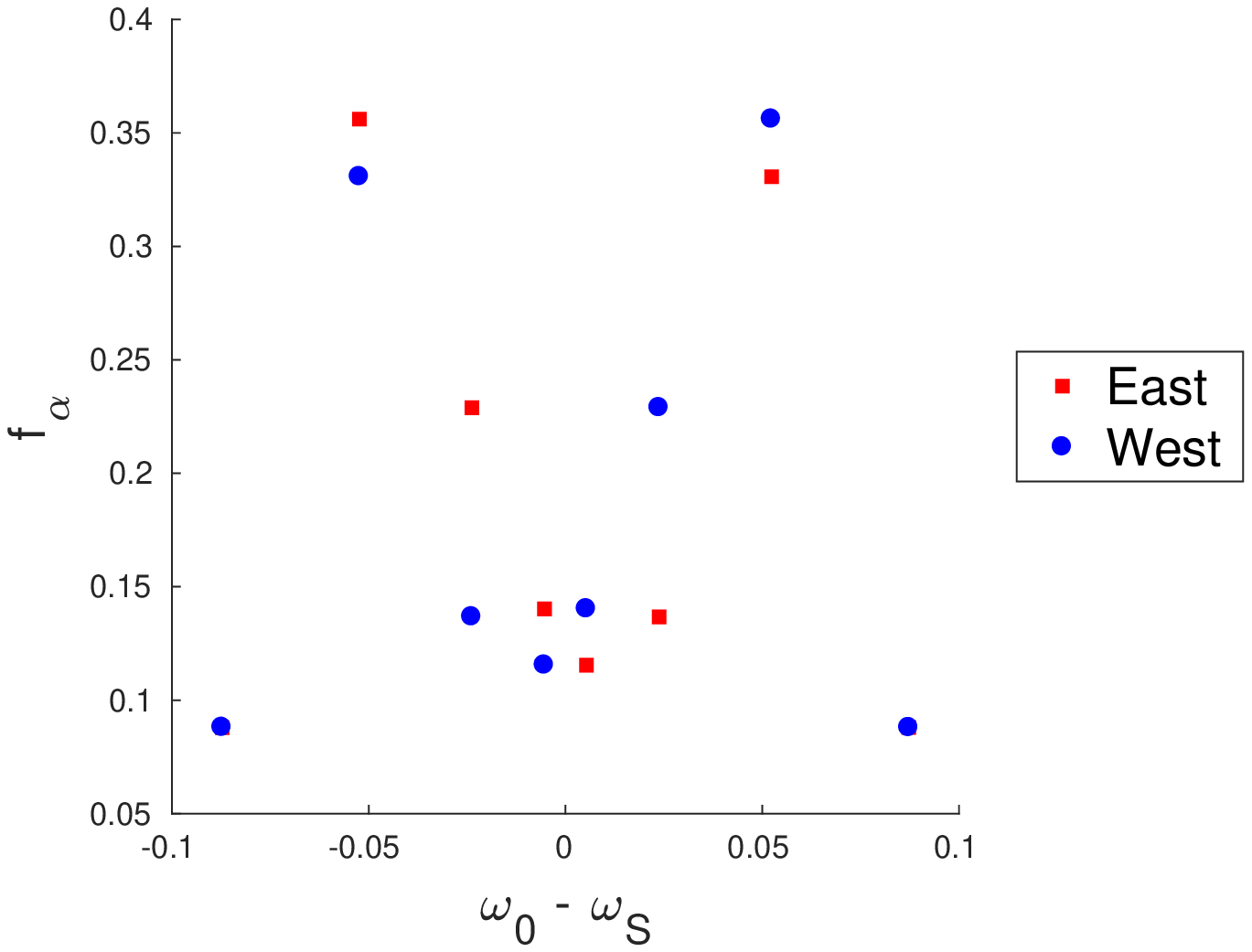}
				\caption{$f_{\alpha}(t)$}
			\end{subfigure}
			\begin{subfigure}{.5\textwidth}
				\centering
				\includegraphics[scale=0.5]{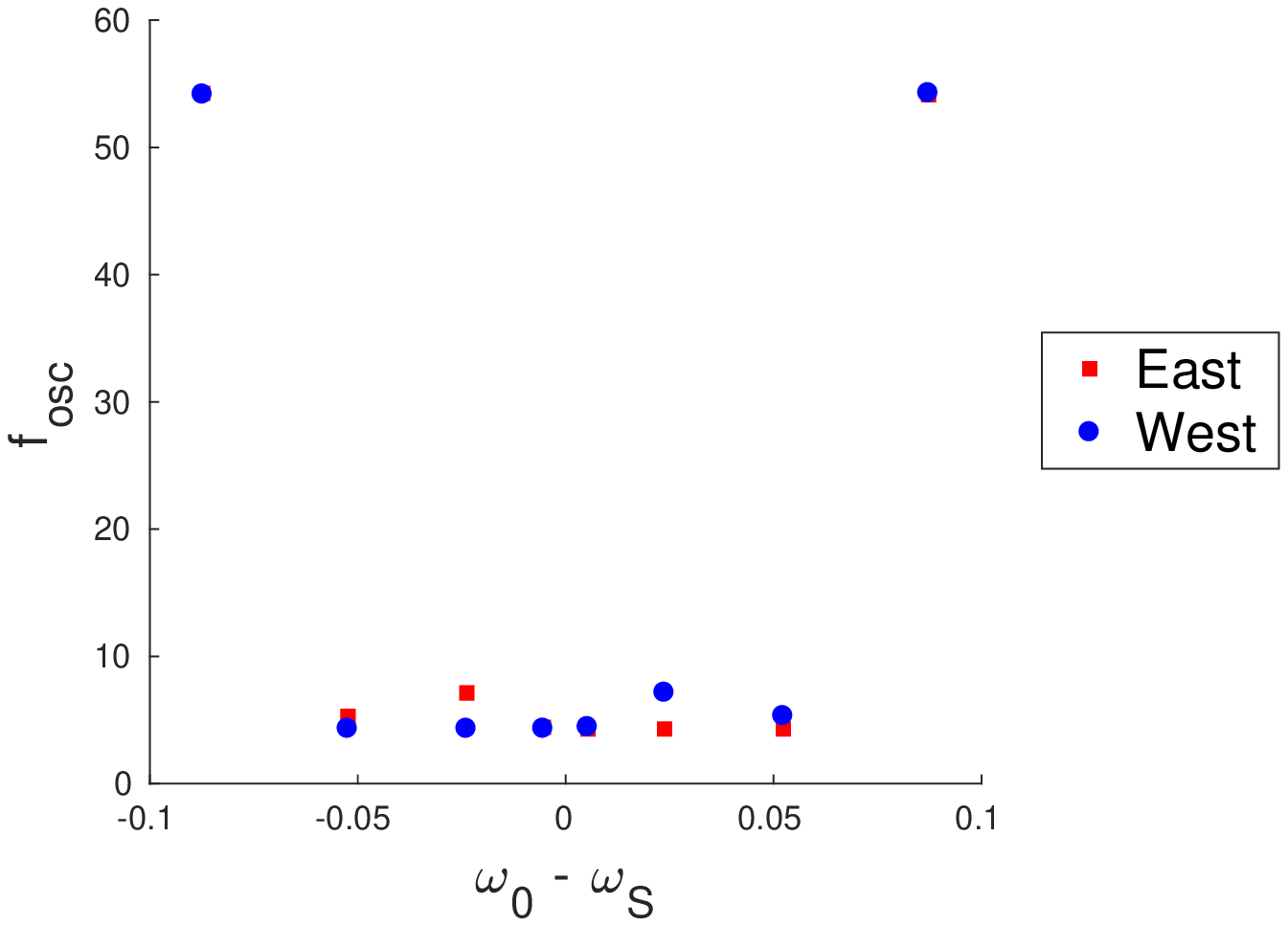}
				\caption{$f_{osc}(t)$}
			\end{subfigure}
			\\
			\begin{subfigure}{.5\textwidth}
				\centering
				\includegraphics[scale=0.5]{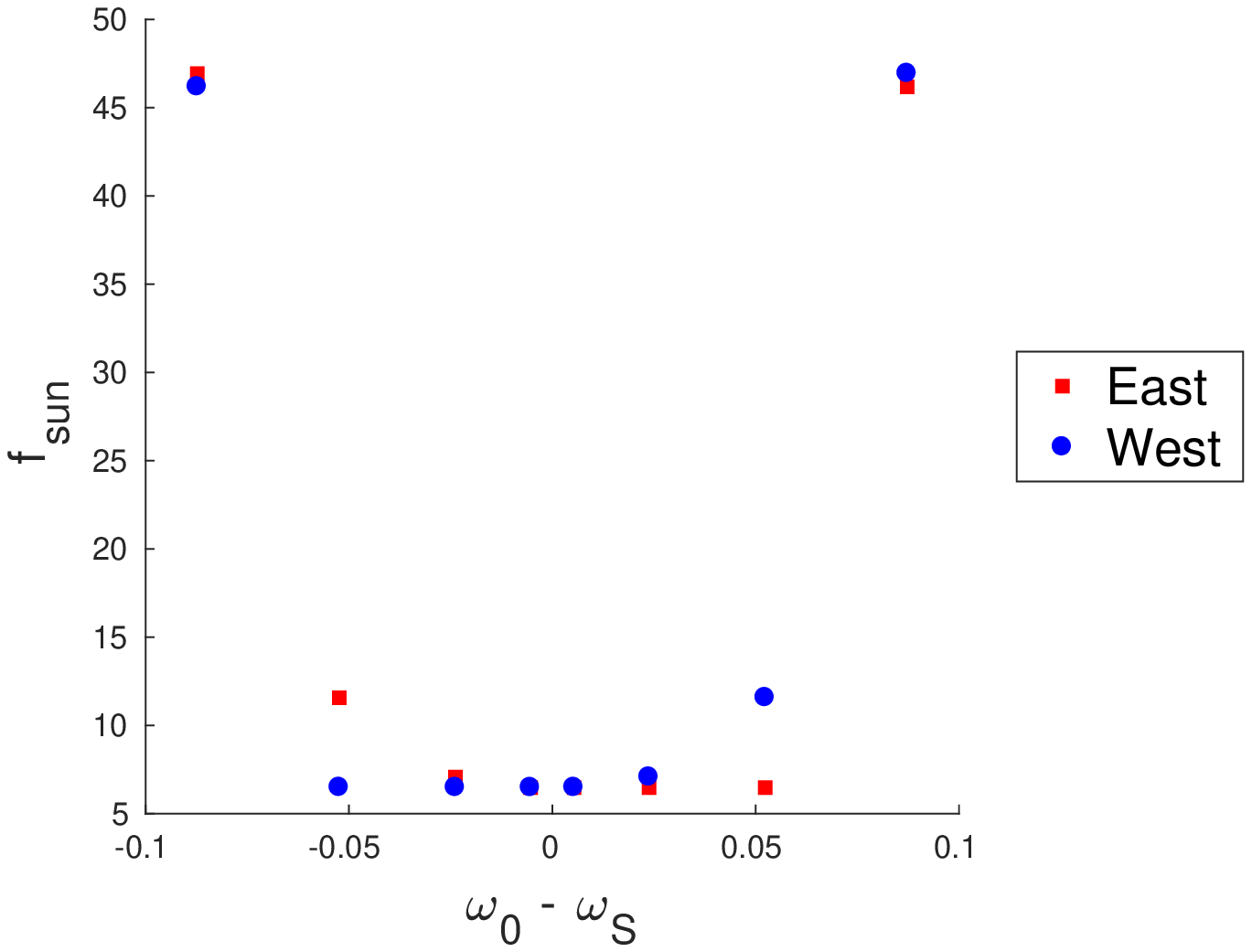}
				\caption{$f_{sun}(t)$}
			\end{subfigure}
			\begin{subfigure}{.5\textwidth}
				\centering
				\includegraphics[scale=0.5]{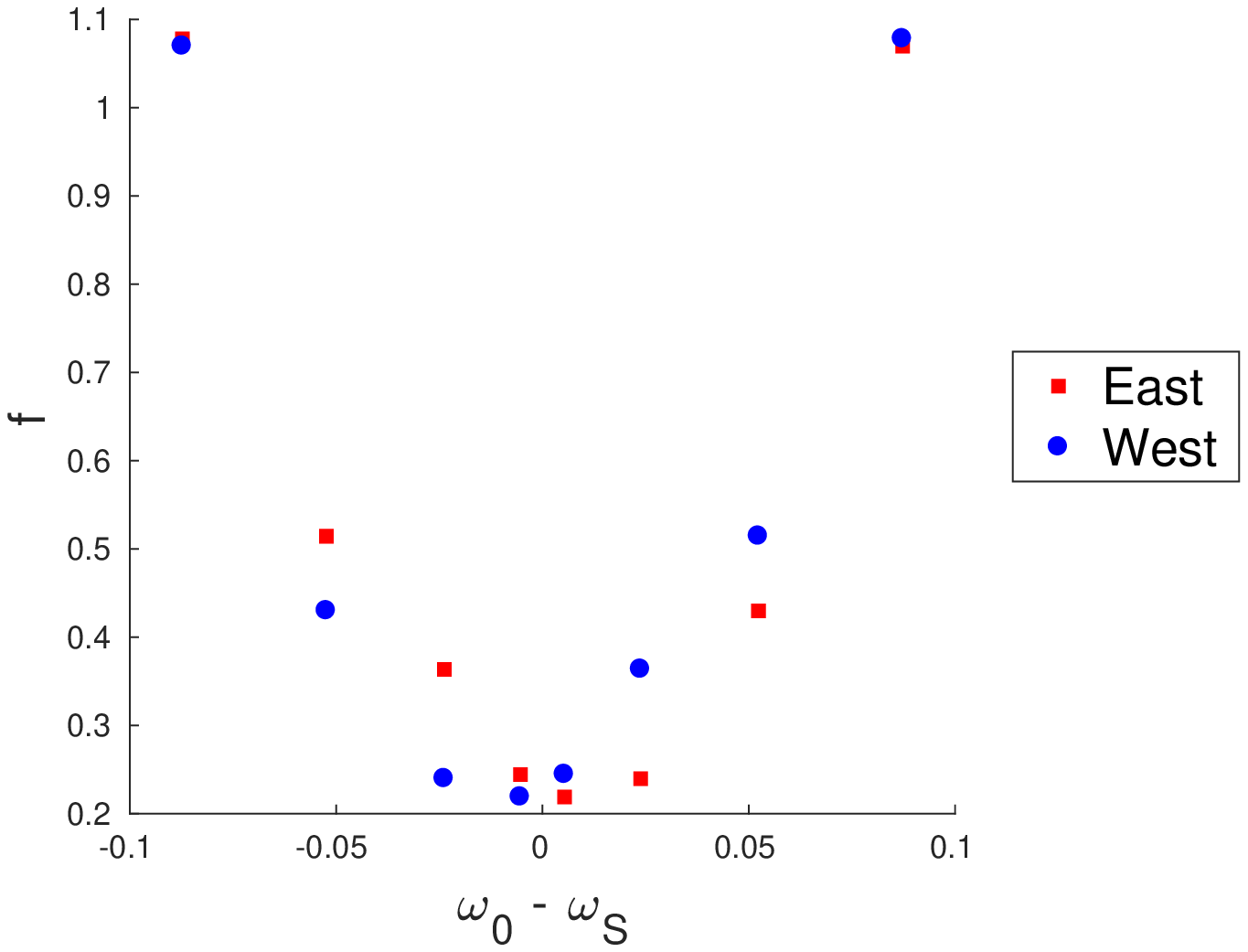}
				\caption{$f(t)$}
			\end{subfigure}
			\caption{\textit{Recovery via Ergodic Problem}: Jet lag recovery costs as a function of $\omega_0$.}
			\label{fig:Changing_omega_costs}
		\end{figure}
		
To understand why it takes so much longer to recover for the extremal values of $\omega_0-\omega_S$, Figure \ref{fig:z_t_omega} shows the paths $z_p(t)$ for $p= \pm 9 \omega_S$ for different values of $\omega_0$. For values of $\omega_0$ closer to $\omega_S$, as in Figure \ref{fig:z_t_omega}(a) and Figure \ref{fig:z_t_omega}(b), the oscillators phase advance after traveling east and phase delay after traveling west. If $\omega_0$ is much smaller than $\omega_S$, as in Figure \ref{fig:z_t_omega}(c) and Figure \ref{fig:z_t_omega}(d), the oscillators recover by phase delaying for both east and west travels. Clearly by symmetry, if $\omega_0$ is much larger than $\omega_S$, then the oscillators recover by phase advancing for both east and west travels. It is interesting to note that in Figure \ref{fig:z_t_omega}(d) the path for recovery is not direct as in the other cases, which is why it takes so much longer to recover.
		
		\begin{figure}
			\begin{subfigure}{.5\textwidth}
				\centering
				\includegraphics[scale=0.5]{Reference_z_t.eps}
				\caption{$\omega_0=2\pi/24.5$}
			\end{subfigure}
			\begin{subfigure}{.5\textwidth}
				\centering
				\includegraphics[scale=0.5]{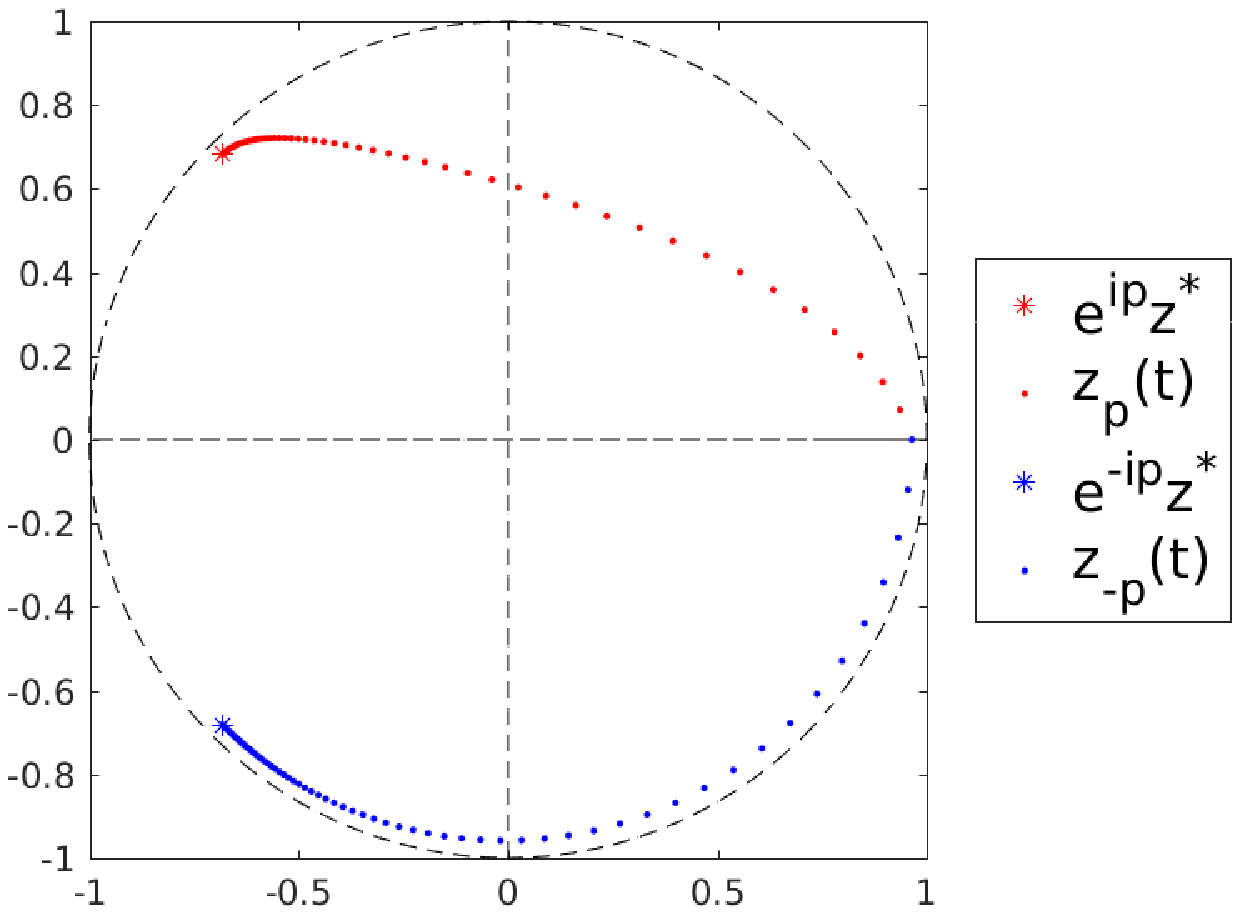}
				\caption{$\omega_0=2\pi/26.4$}
			\end{subfigure}
			\\
			\begin{subfigure}{.5\textwidth}
				\centering
				\includegraphics[scale=0.5]{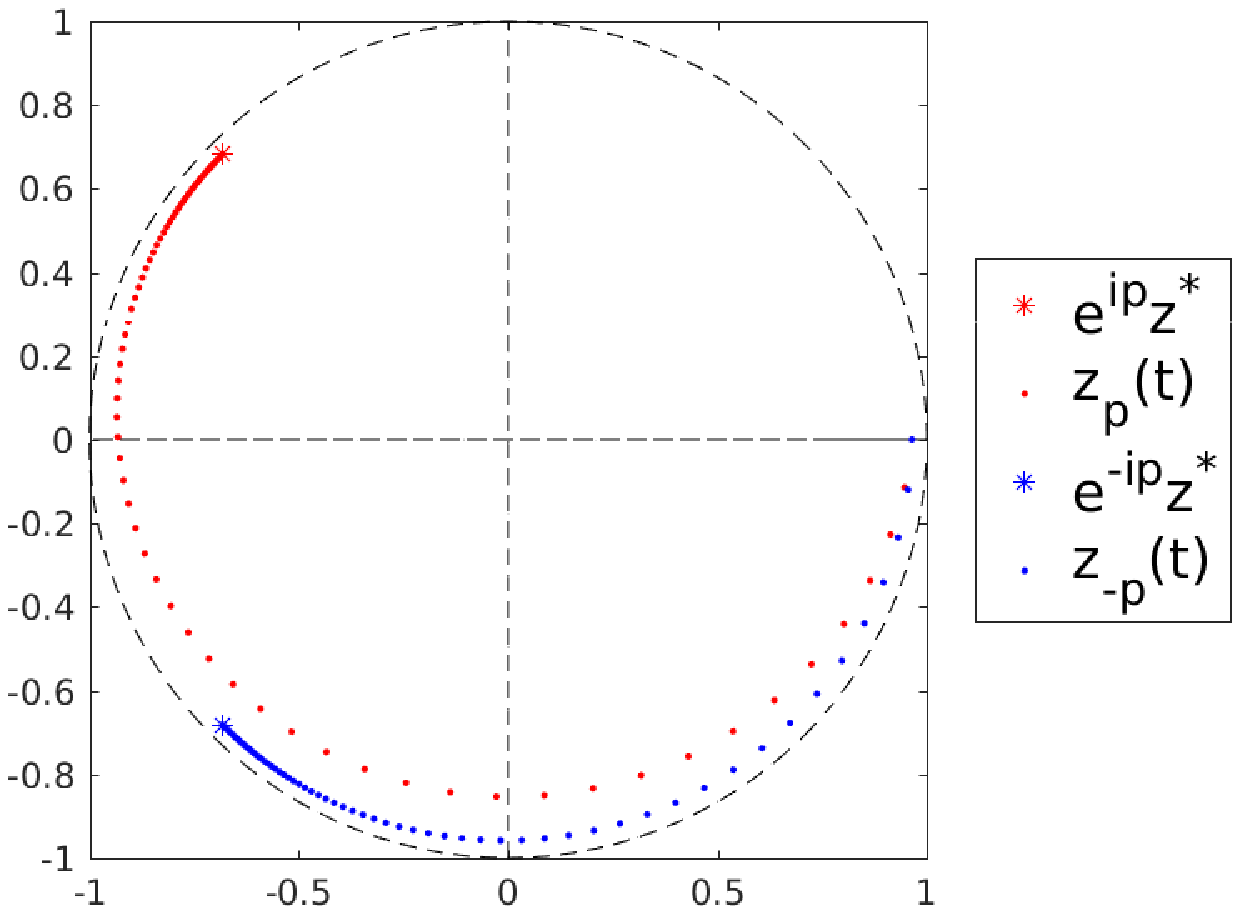}
				\caption{$\omega_0=2\pi/30$}
			\end{subfigure}
			\begin{subfigure}{.5\textwidth}
				\centering
				\includegraphics[scale=0.5]{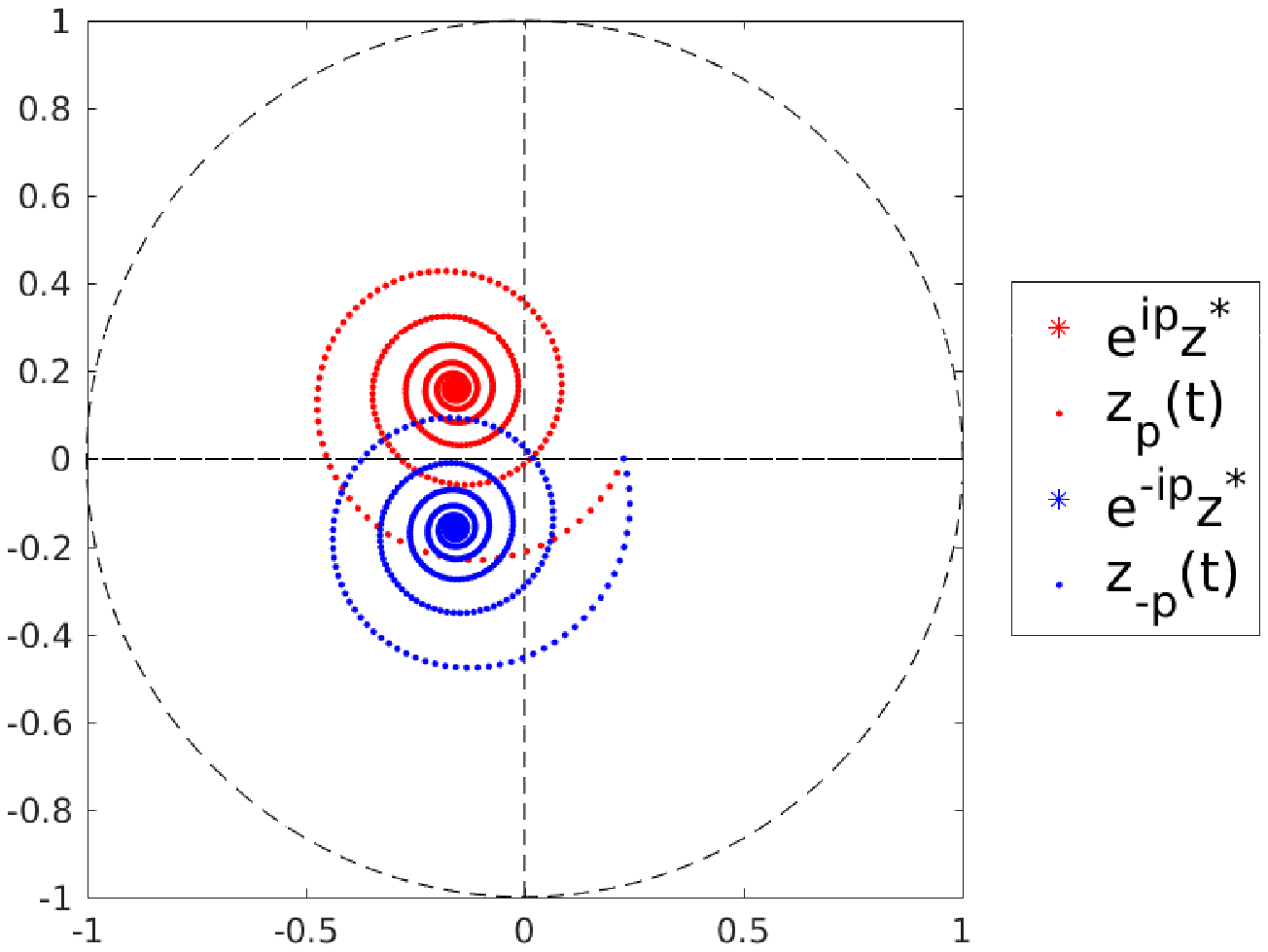}
				\caption{$\omega_0=2\pi/36$}
			\end{subfigure}
			\caption{\textit{Recovery via Ergodic Problem}: Path $z_p(t)$ while recovering from jet lag after traveling $9$ time zones east (red) and west (blue) for various values of $\omega_0$. A point is plotted every hour.}
			\label{fig:z_t_omega}
		\end{figure}

The results of changing $\sigma$ are shown in Figures \ref{fig:Changing_sigma} and \ref{fig:Changing_sigma_costs}. At first glance, it is surprising that the recovery times decrease as $\sigma$ increases, as one might think that it would be harder to recover from jet lag when there is more noise. However, the solution to the \textit{Ergodic Mean Field Game Problem} as $\sigma$ increases becomes closer and closer towards a uniform distribution, and thus $\mu^*(\phi)$ and $\mu^*(\phi-p)$ become closer to each other. In fact, when $|z^*|$ is close enough to $0$, then $|z_p(0)-e^{ip}z^*|=|z^*-e^{ip}z^*| <\epsilon^z$ and $\tau_p^z=0$. This is the case when $\sigma=0.5$ and $\sigma=1$ as shown in Figure \ref{fig:Changing_sigma}.

Since $\mu^*(\phi)$ becomes closer to a uniform distribution with larger $\sigma$, it also makes sense that the cost for synchronization with the other oscillators, $f_{osc}$, and the cost for synchronization with the natural 24 hour cycle, $f_{sun}$, both increase with $\sigma$. The reason for the increase then decrease of $f_{\alpha}$ is less clear.
		
		\begin{figure}
			\centering
			\includegraphics[scale=0.5]{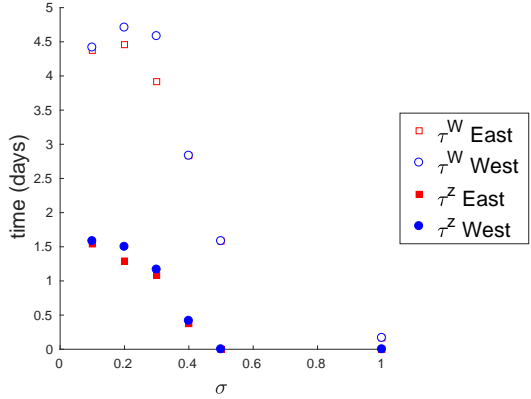}
			\caption{\textit{Recovery via Ergodic Problem}: Jet lag recovery times as a function of $\sigma$.}
			\label{fig:Changing_sigma}
		\end{figure}
		
		\begin{figure}
			\begin{subfigure}{.5\textwidth}
				\centering
				\includegraphics[scale=0.5]{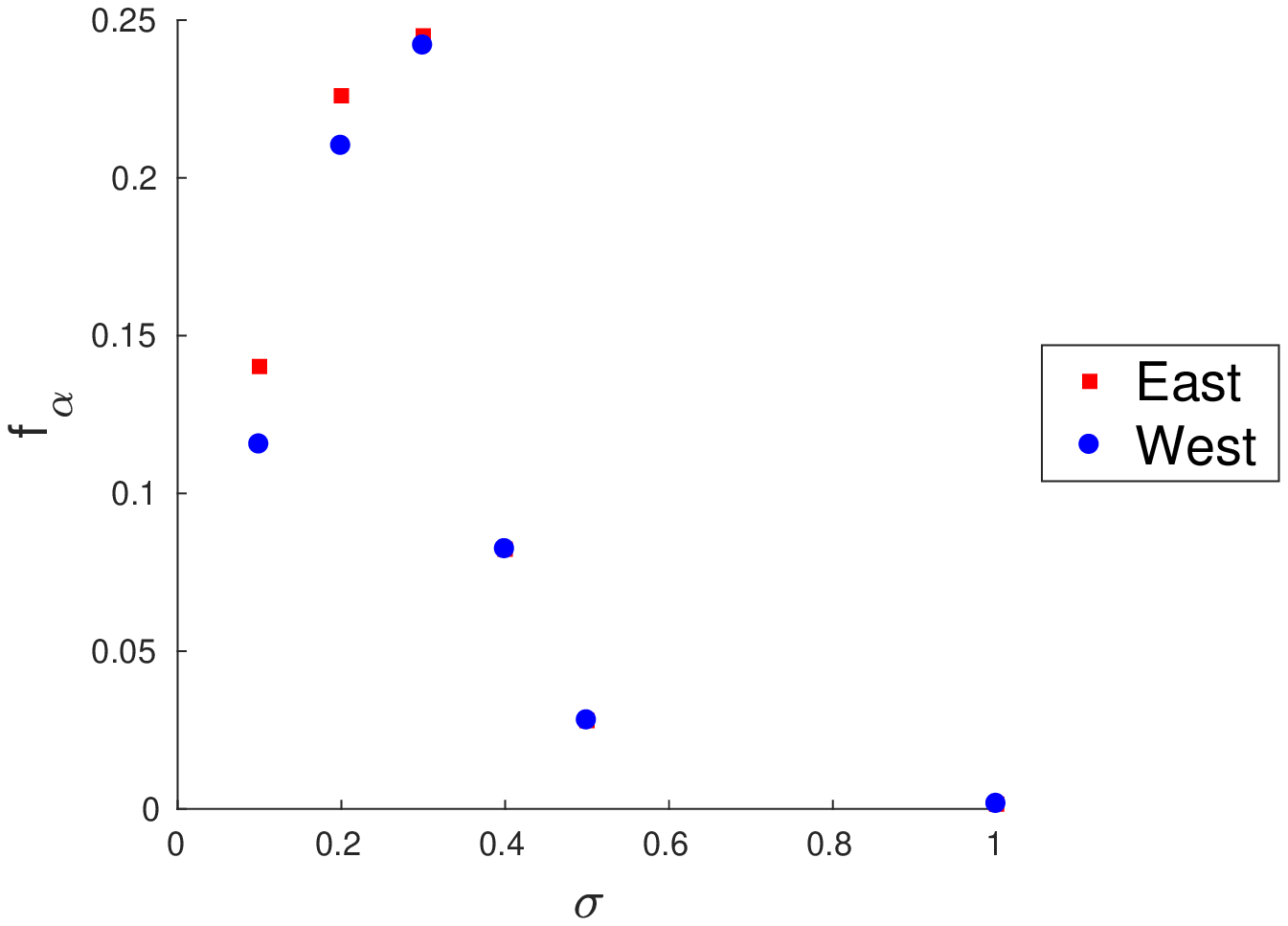}
				\caption{$f_{\alpha}(t)$}
			\end{subfigure}
			\begin{subfigure}{.5\textwidth}
				\centering
				\includegraphics[scale=0.5]{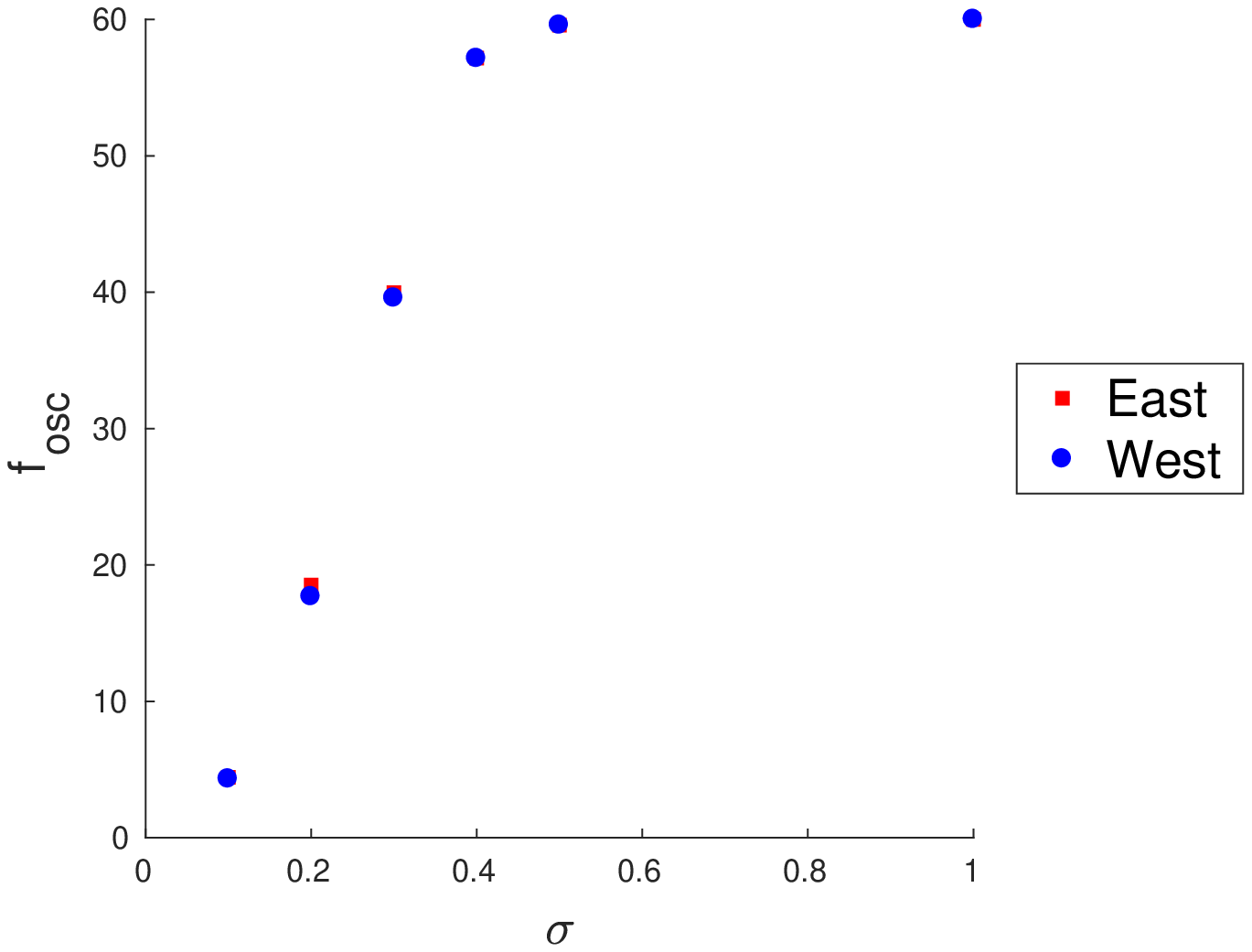}
				\caption{$f_{osc}(t)$}
			\end{subfigure}
			\\
			\begin{subfigure}{.5\textwidth}
				\centering
				\includegraphics[scale=0.5]{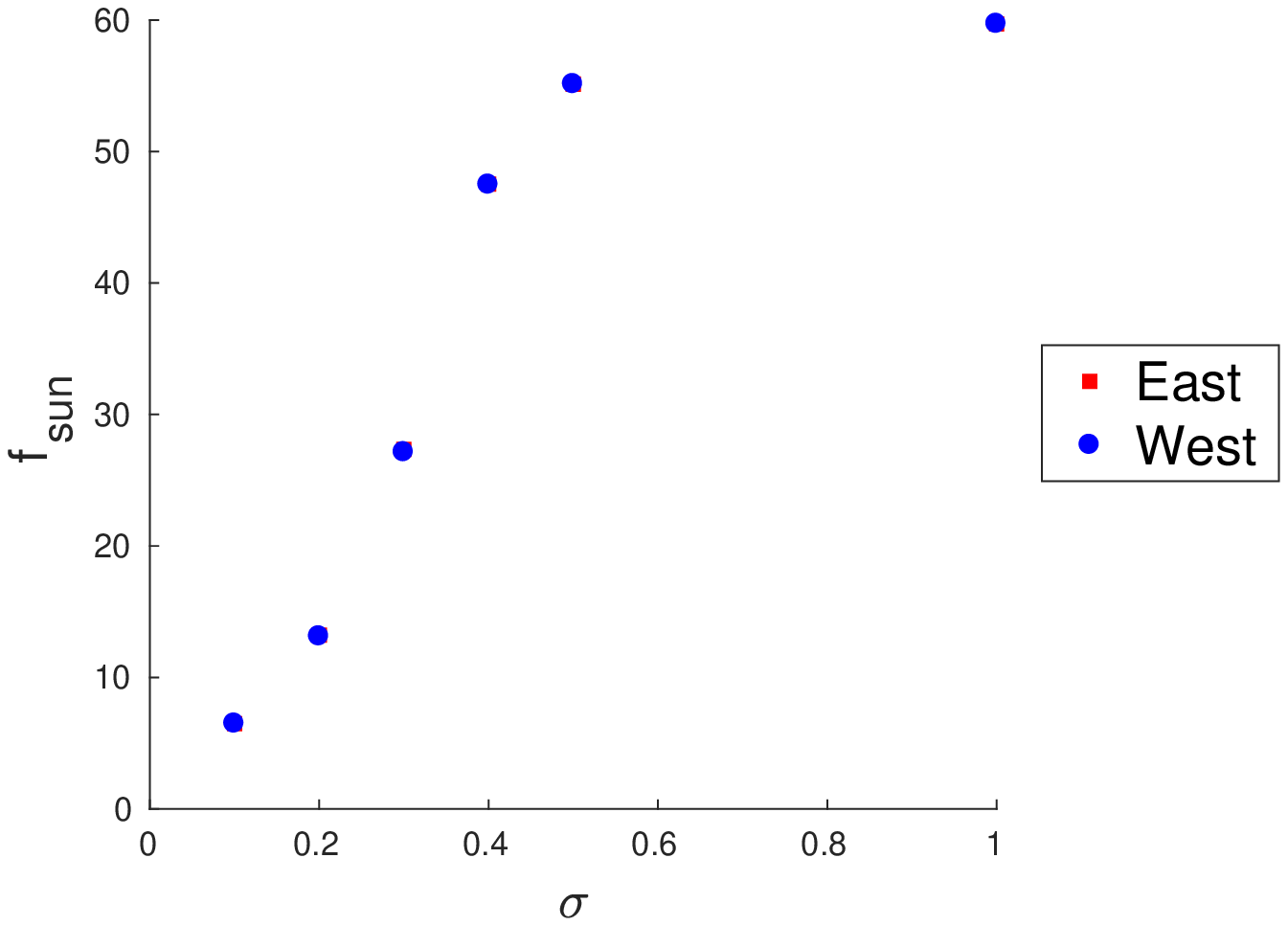}
				\caption{$f_{sun}(t)$}
			\end{subfigure}
			\begin{subfigure}{.5\textwidth}
				\centering
				\includegraphics[scale=0.5]{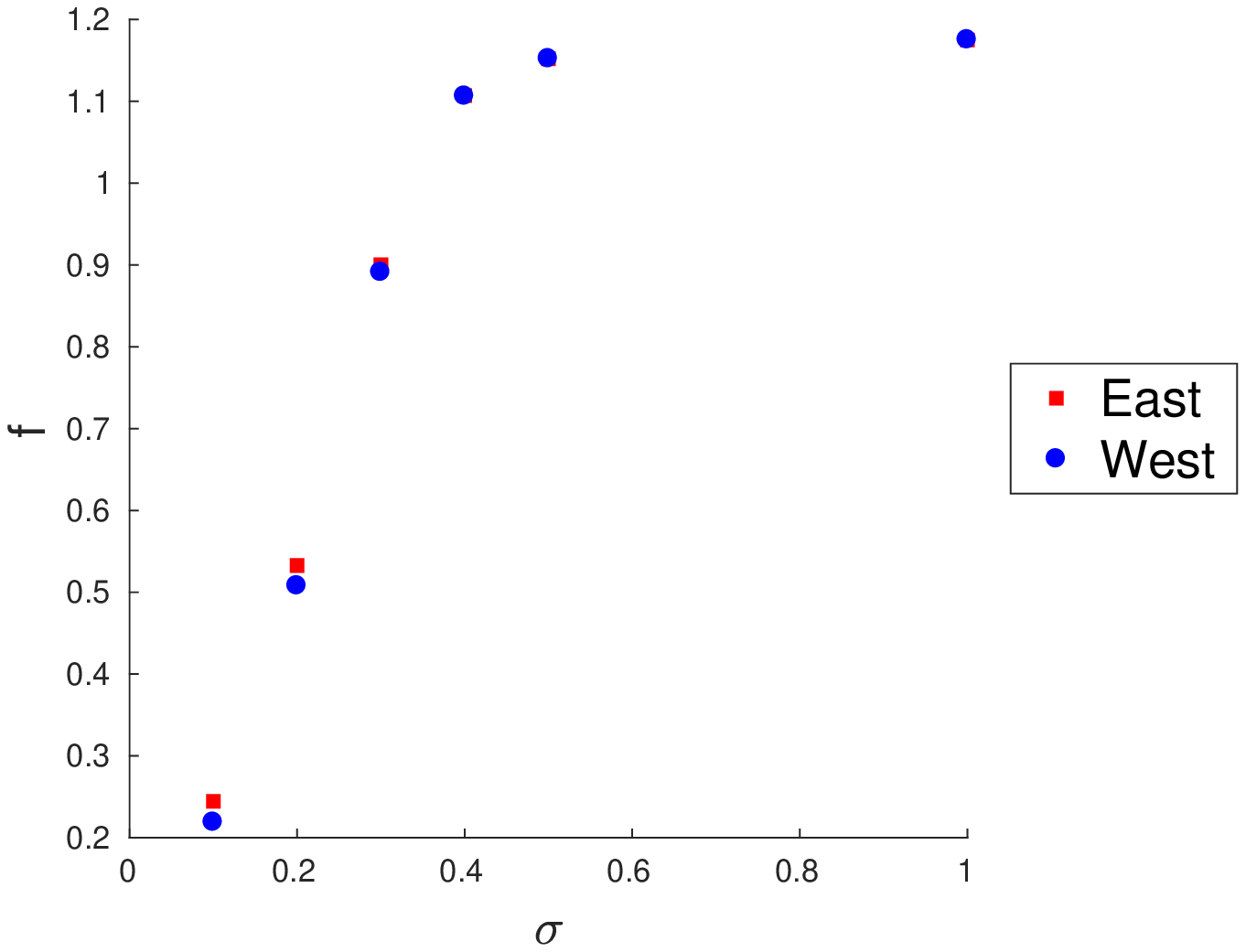}
				\caption{$f(t)$}
			\end{subfigure}
			\caption{\textit{Recovery via Ergodic Problem}: Jet lag recovery costs as a function of $\sigma$.}
			\label{fig:Changing_sigma_costs}
		\end{figure}
		
Figures \ref{fig:Changing_K} and \ref{fig:Changing_K_costs} show the results when changing $K$. The recovery times decrease with $K$. Since a larger $K$ puts more weight on synchronization of the oscillators with each other, it is unsurprising that the $f_{osc}$ and $f_{sun}$ decrease with $K$. Since more synchronization requires more effort, it is also unsurprising that $f_{\alpha}$ increases with $K$.
		
		\begin{figure}
			\centering
			\includegraphics[scale=0.5]{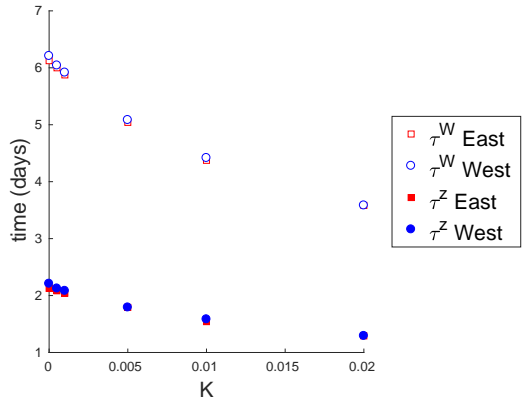}
			\caption{\textit{Recovery via Ergodic Problem}: Jet lag recovery times as a function of $K$.}
			\label{fig:Changing_K}
		\end{figure}
		
		\begin{figure}
			\begin{subfigure}{.5\textwidth}
				\centering
				\includegraphics[scale=0.5]{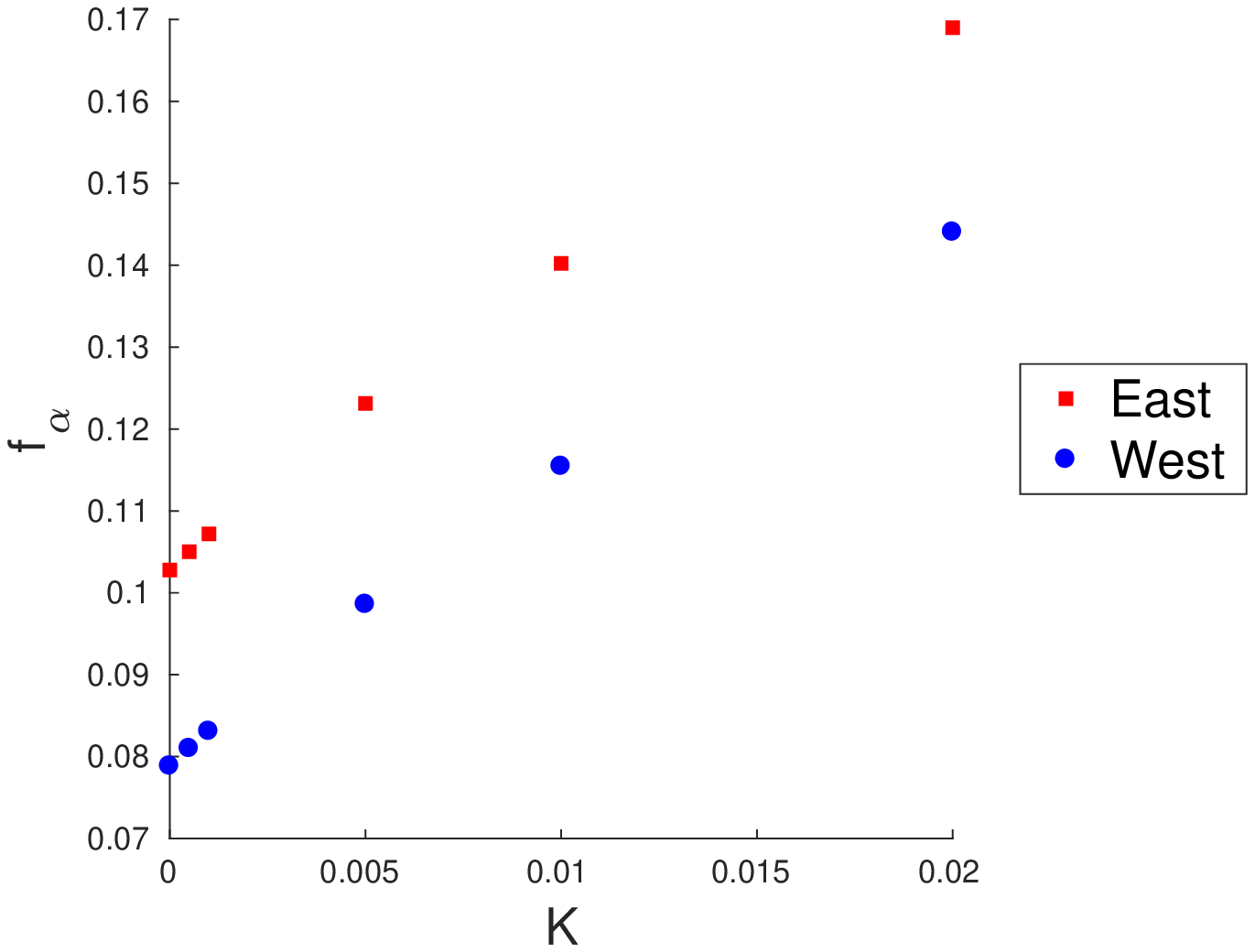}
				\caption{$f_{\alpha}(t)$}
			\end{subfigure}
			\begin{subfigure}{.5\textwidth}
				\centering
				\includegraphics[scale=0.5]{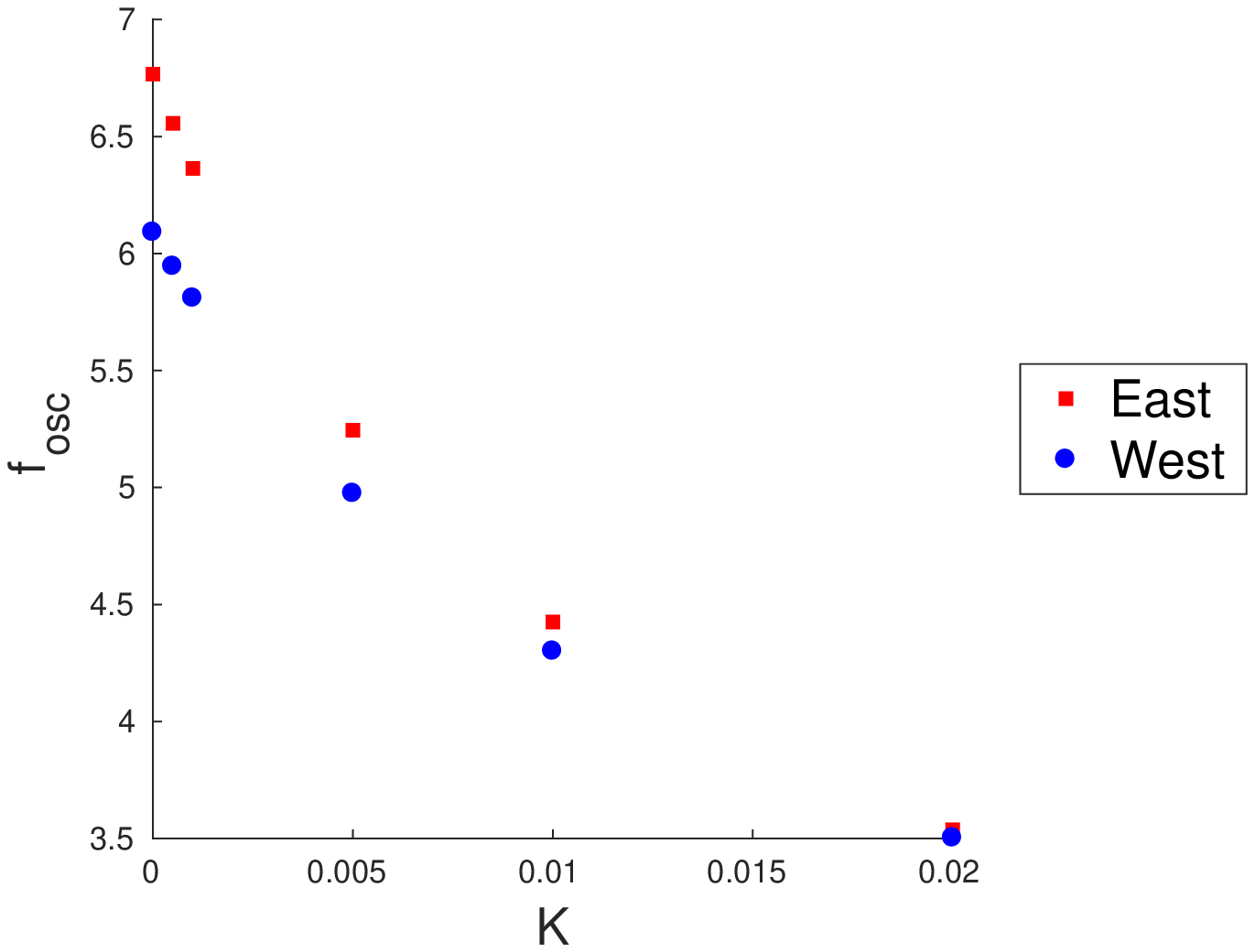}
				\caption{$f_{osc}(t)$}
			\end{subfigure}
			\\
			\begin{subfigure}{.5\textwidth}
				\centering
				\includegraphics[scale=0.5]{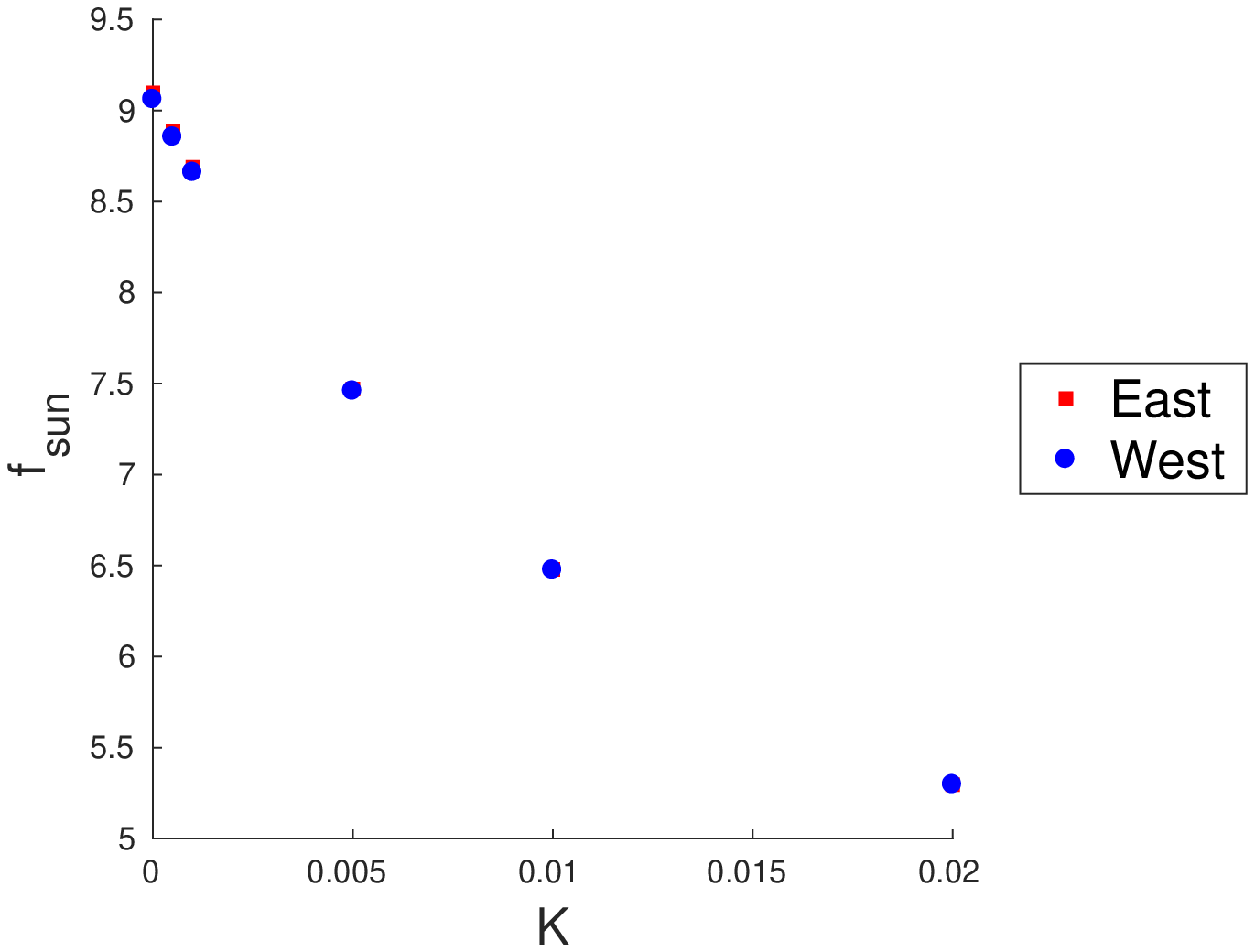}
				\caption{$f_{sun}(t)$}
			\end{subfigure}
			\begin{subfigure}{.5\textwidth}
				\centering
				\includegraphics[scale=0.5]{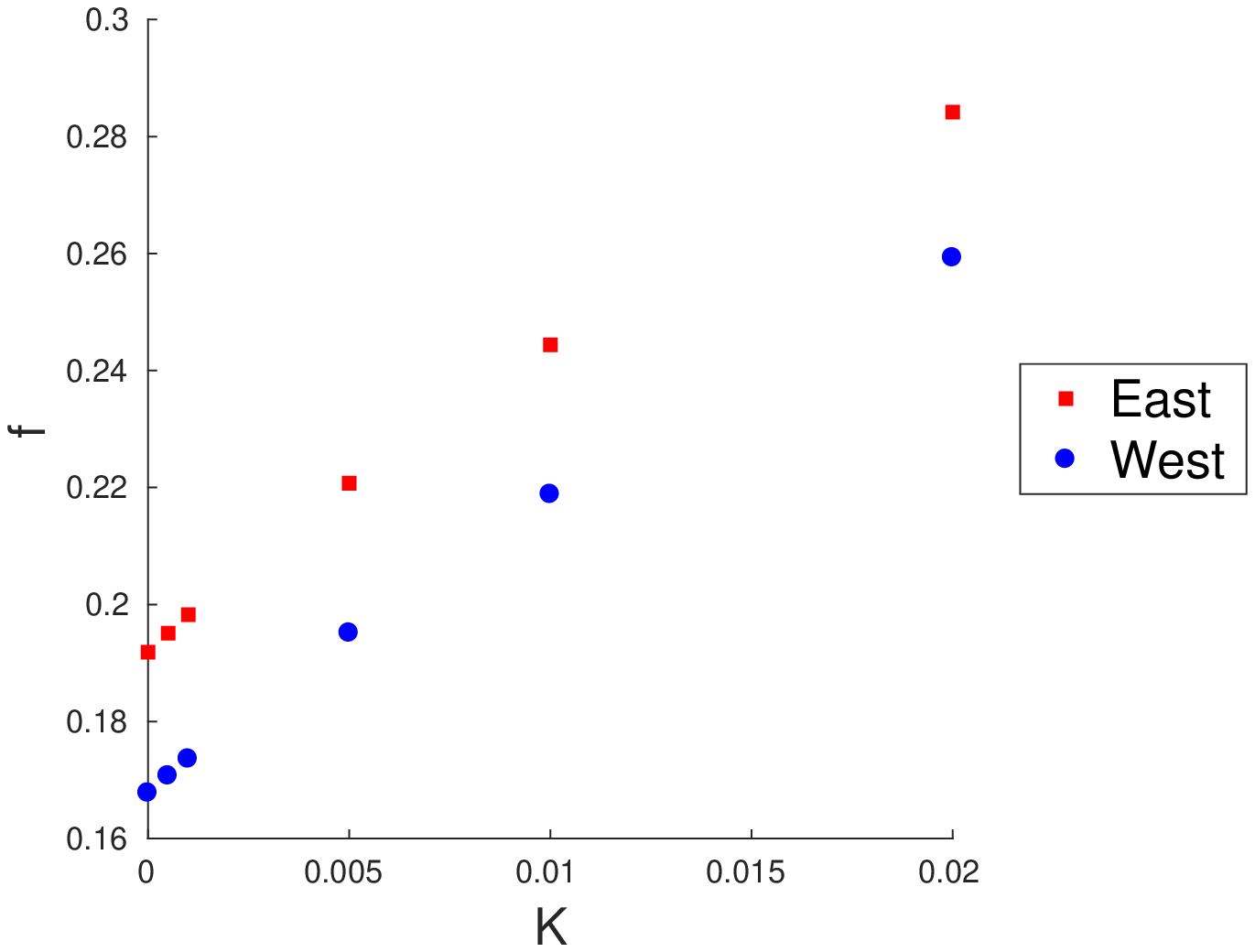}
				\caption{$f(t)$}
			\end{subfigure}
			\caption{\textit{Recovery via Ergodic Problem}: Jet lag recovery costs as a function of $K$.}
			\label{fig:Changing_K_costs}
		\end{figure}
		
The results for changing $F$ are shown in Figures \ref{fig:Changing_F} and \ref{fig:Changing_F_costs}. The recovery times decrease as $F$ increases, which is intuitive. The recovery costs are qualitatively the same as for changing $K$, with similar justifications. As we increase $F$, we put more weight on synchronization with the 24 hour light/dark cycle, which will decrease $f_{sun}$. As a result, the oscillators will be more synchronized with each other as well, which will decrease $f_{osc}$. To achieve a larger degree of synchronization will require a larger control, so $f_{\alpha}$ increases.
		
		\begin{figure}
			\centering
			\includegraphics[scale=0.5]{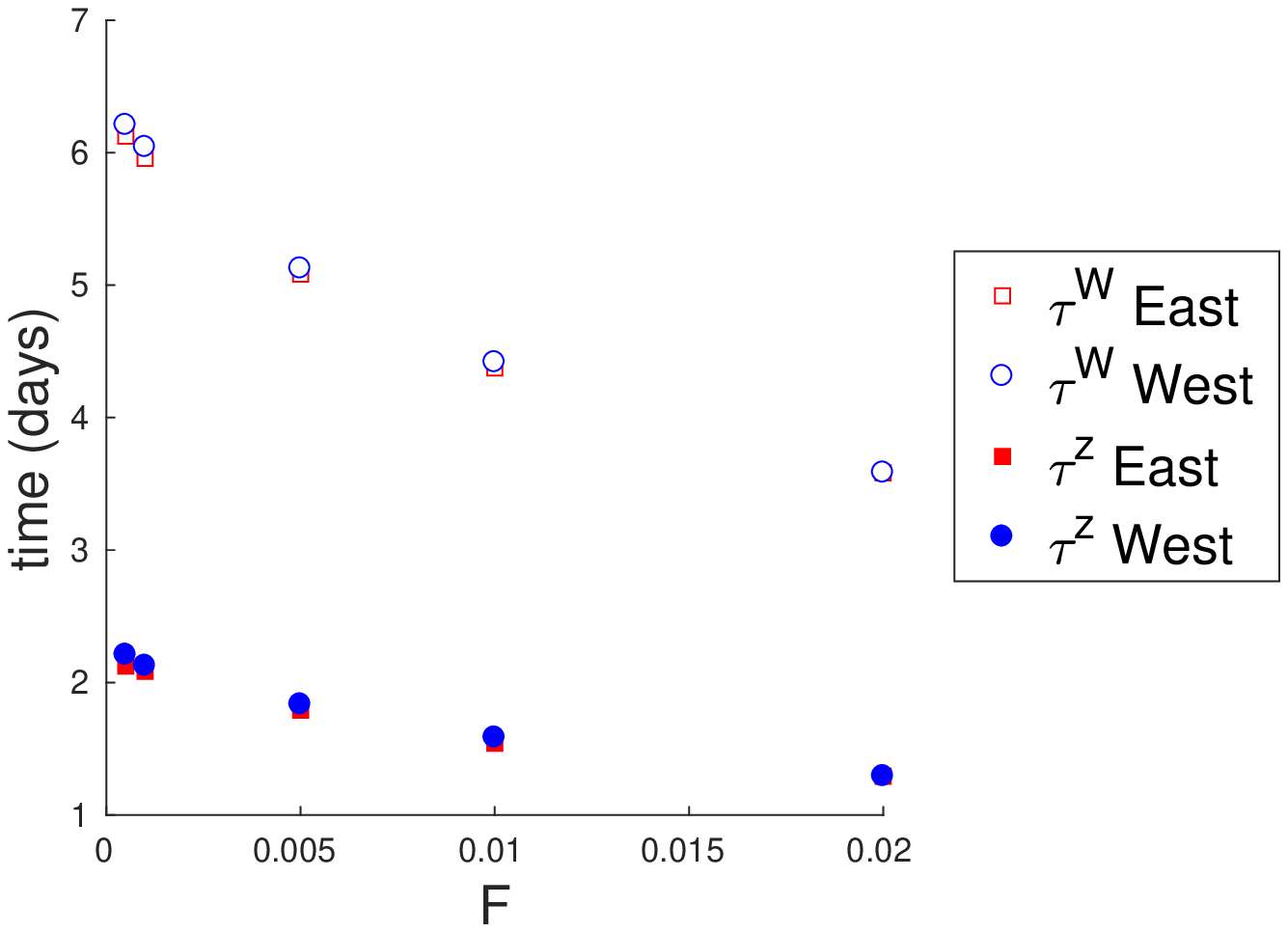}
			\caption{\textit{Recovery via Ergodic Problem}: Jet lag recovery times as a function of $F$.}
			\label{fig:Changing_F}
		\end{figure}
		
		\begin{figure}
			\begin{subfigure}{.5\textwidth}
				\centering
				\includegraphics[scale=0.5]{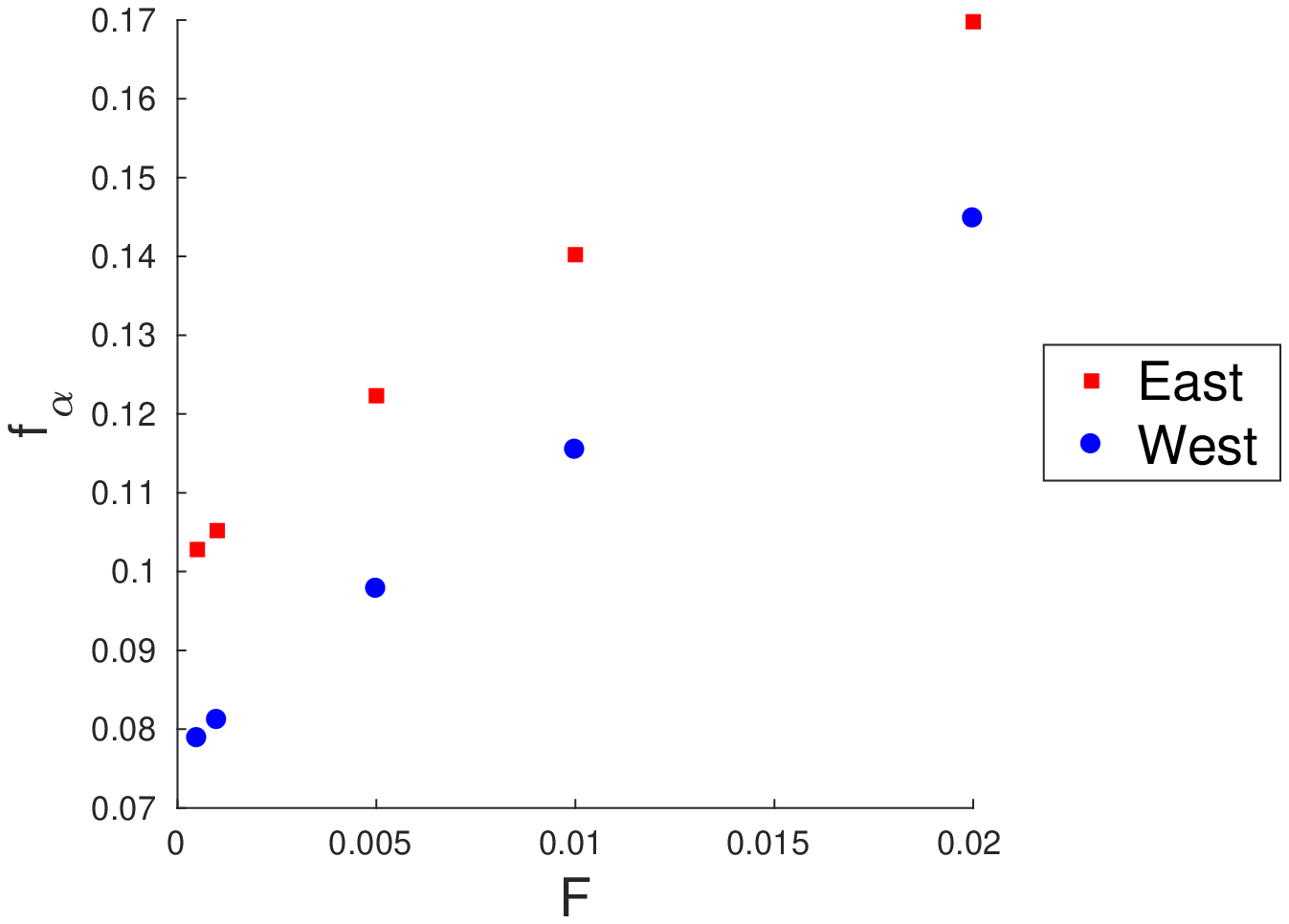}
				\caption{$f_{\alpha}(t)$}
			\end{subfigure}
			\begin{subfigure}{.5\textwidth}
				\centering
				\includegraphics[scale=0.5]{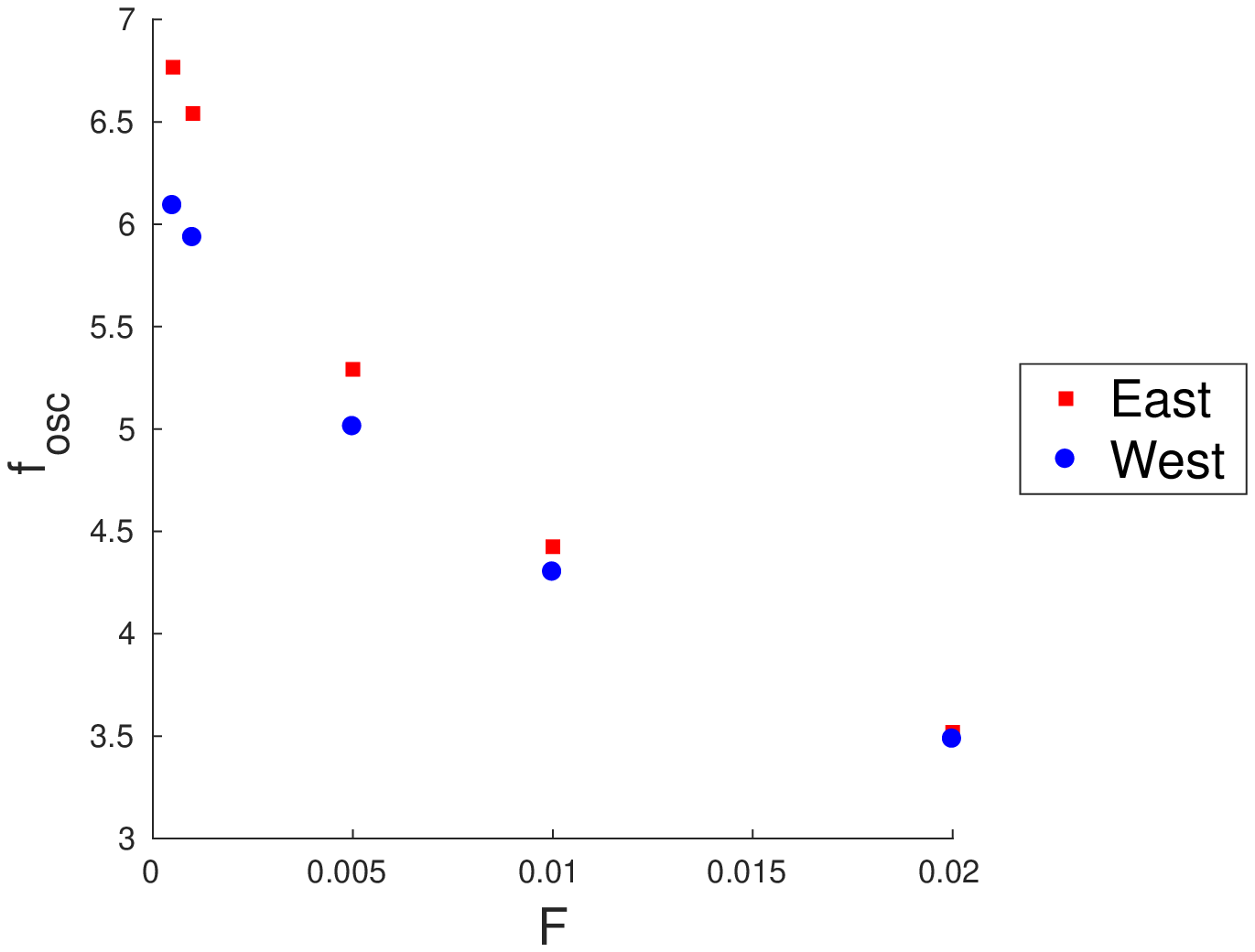}
				\caption{$f_{osc}(t)$}
			\end{subfigure}
			\\
			\begin{subfigure}{.5\textwidth}
				\centering
				\includegraphics[scale=0.5]{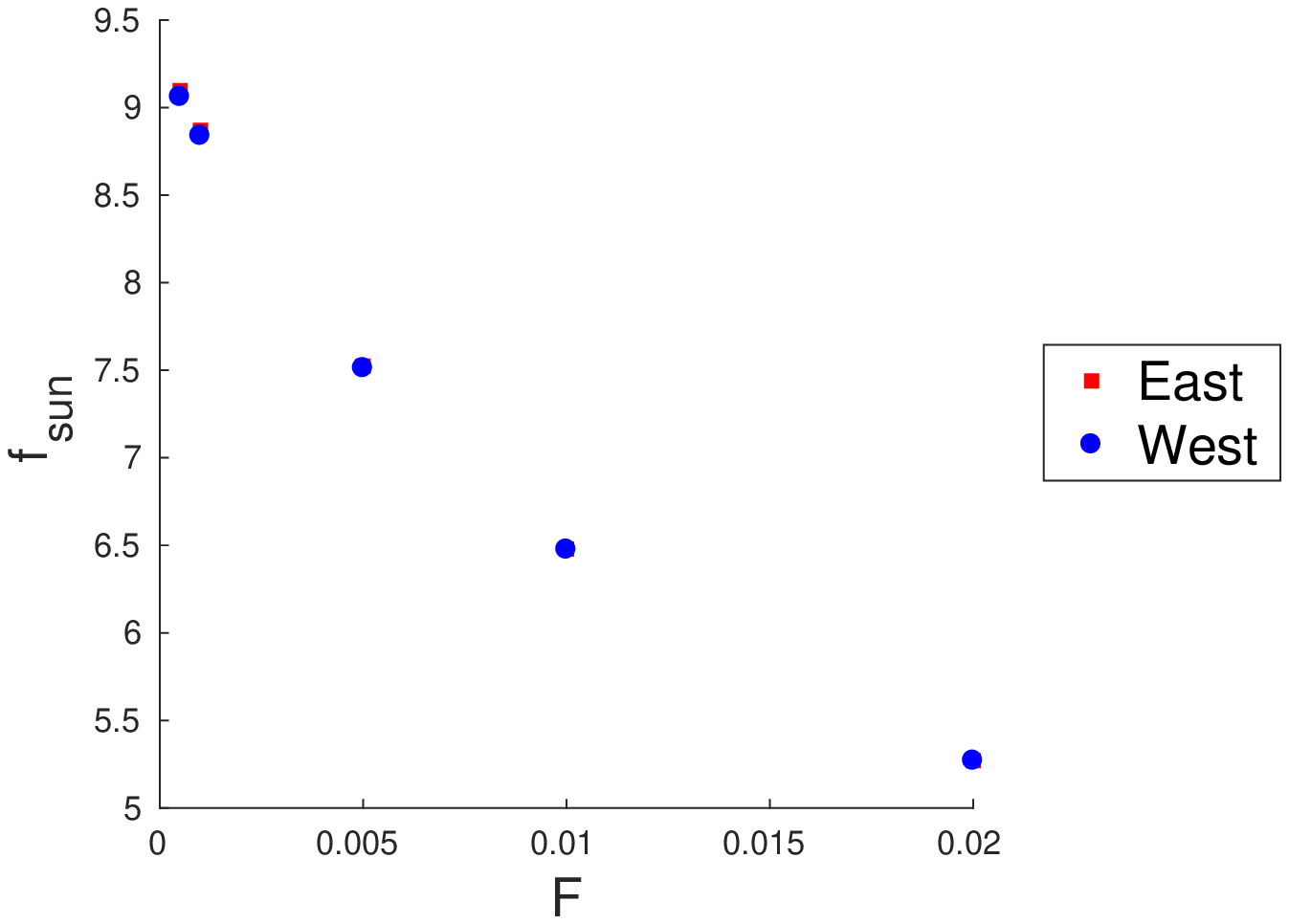}
				\caption{$f_{sun}(t)$}
			\end{subfigure}
			\begin{subfigure}{.5\textwidth}
				\centering
				\includegraphics[scale=0.5]{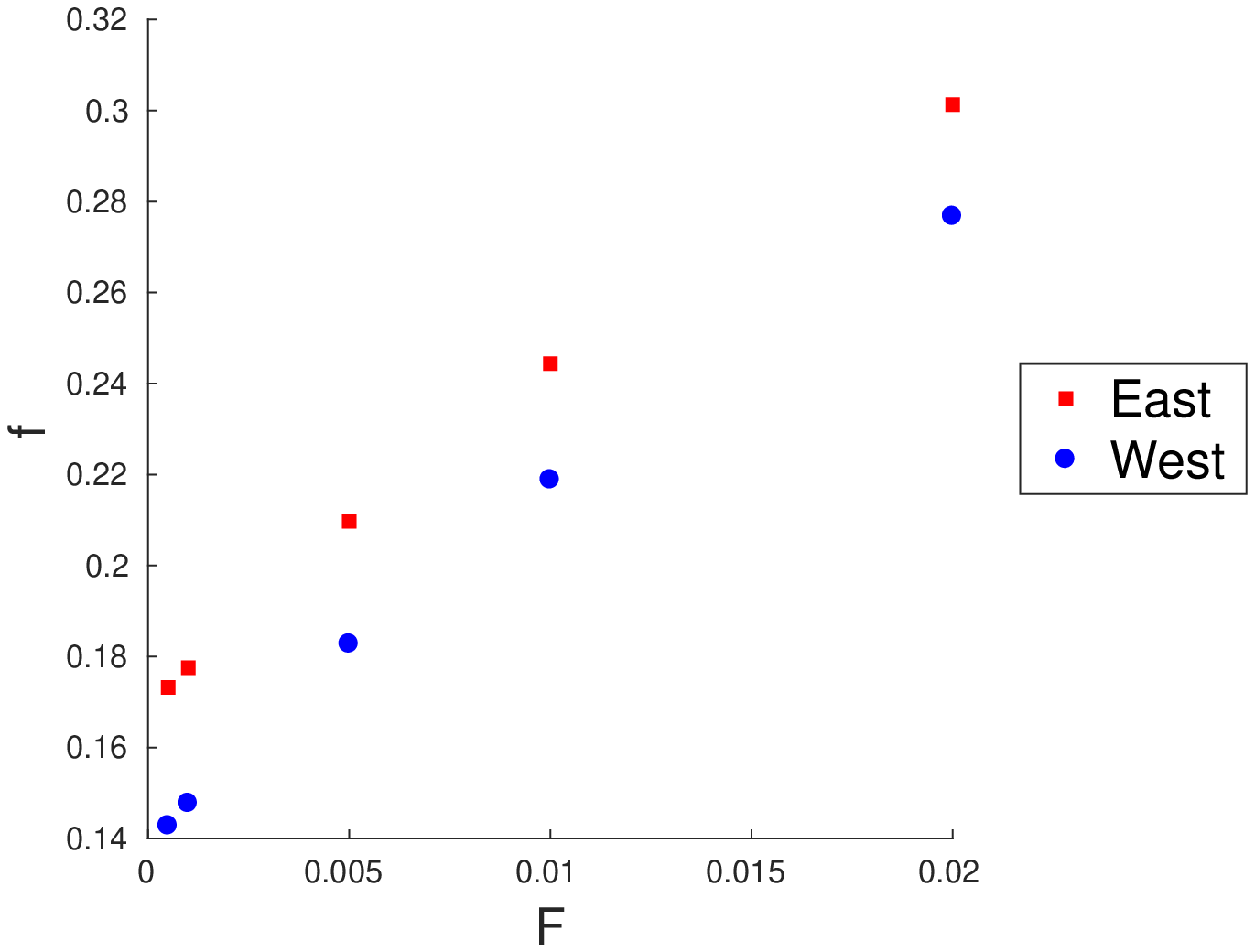}
				\caption{$f(t)$}
			\end{subfigure}
			\caption{\textit{Recovery via Ergodic Problem}: Jet lag recovery costs as a function of $F$.}
			\label{fig:Changing_F_costs}
		\end{figure}
		
In summary, Figures \ref{fig:Changing_p}, \ref{fig:Changing_omega}, \ref{fig:Changing_sigma}, \ref{fig:Changing_K}, and \ref{fig:Changing_F}  show that the recovery time is about the same for east and west travel. Figures \ref{fig:Changing_p_costs}, \ref{fig:Changing_omega_costs}, \ref{fig:Changing_sigma_costs}, \ref{fig:Changing_K_costs}, and \ref{fig:Changing_F_costs} show that the recovery costs are larger for east than west travel when $\omega_0<\omega_S$. The recovery times and recovery costs both increase as $|p|$ or $|\omega_0-\omega_S|$ increase. On the other hand, the recovery times decrease at the expense of increasing recovery costs as $\sigma$, $K$, or $F$ increase.

\subsection{Recovery Mean Field Game Problem Numerical Solution and Parameter Sensitivity}	
\hspace{5mm} Now we present the numerical results for the \textit{Recovery Mean Field Game Problem}. We set the finite time horizon to the reference value of $T=100$ days. For the reference set of the parameters, we find that the recovery times are $\tilde{\tau}^W_{9\omega_S}=\tilde{\tau}^W_{-9\omega_S}=6.17$ days and $\tilde{\tau}^z_{9\omega_S}=\tilde{\tau}^z_{-9\omega_S}=2.17$ days, where $\epsilon^W=0.01$ and $\epsilon^z=0.2$, as before. Thus, the recovery times are the same for east and west travels. Note that $\tilde{\tau}^W_{\pm 9\omega_S}>\tau^W_{\pm9\omega_S}$ and $\tilde{\tau}^z_{\pm 9\omega_S}>\tau^z_{\pm9\omega_S}$. Thus, the recovery times are larger when the oscillators solve a new mean field game problem for the recovery time period, instead of adapting the optimal control from the \textit{Ergodic Mean Field Game Problem}, as in the \textit{Recovery via Ergodic Problem}.

As in the the previous model, we find that the total recovery cost is larger for eastward travels. When comparing the recovery cost for the two models of jet lag recovery, the mean field game model has a smaller recovery cost. To summarize, for both models, the recovery times are about the same for eastward and westward travel with a larger recovery cost for eastward travel. When the oscillators re-optimize over a finite time horizon, it takes longer for them to recover from jet lag, but they accrue a smaller recovery cost in the process.

The analogous plots of Figures \ref{fig:Changing_p}-\ref{fig:Changing_omega_costs} and Figures \ref{fig:Changing_sigma}-\ref{fig:Changing_F_costs} are quite similar. One may wonder how much these results depend on the choice of the finite time horizon $T=100$ days. For the reference set of parameters, we verify that the recovery times and costs are the same for $T=50$, $100$, $150$, or $200$ days.

The main deviation between the two recovery models is what happens when $|\omega_0-\omega_S|$ is large. For the previous model, we notice a large jump in the recovery time for $\omega_0=2\pi/36$, for example. For the \textit{Recovery Mean Field Game Problem}, however, we find that the oscillators never recover. It is also interesting to note that the numerical scheme converges slowly in the large $|\omega_0-\omega_S|$ regime, requiring many more iterations to reach convergence than the other numerical results.

		\begin{figure}
			\begin{subfigure}{.5\textwidth}
				\centering
				\includegraphics[scale=0.5]{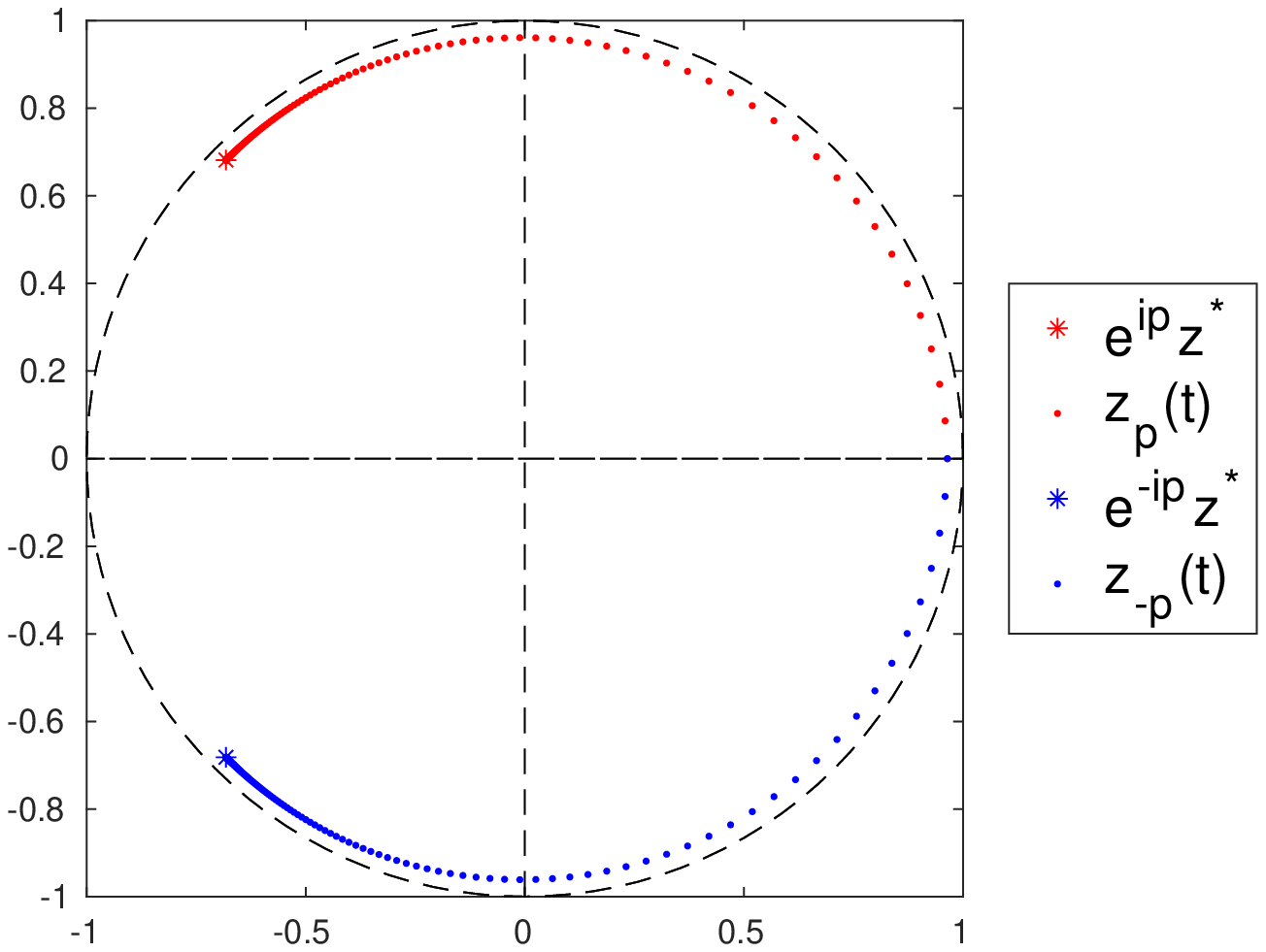}
				\caption{$\omega_0=2\pi/24.5$}
			\end{subfigure}
			\begin{subfigure}{.5\textwidth}
				\centering
				\includegraphics[scale=0.5]{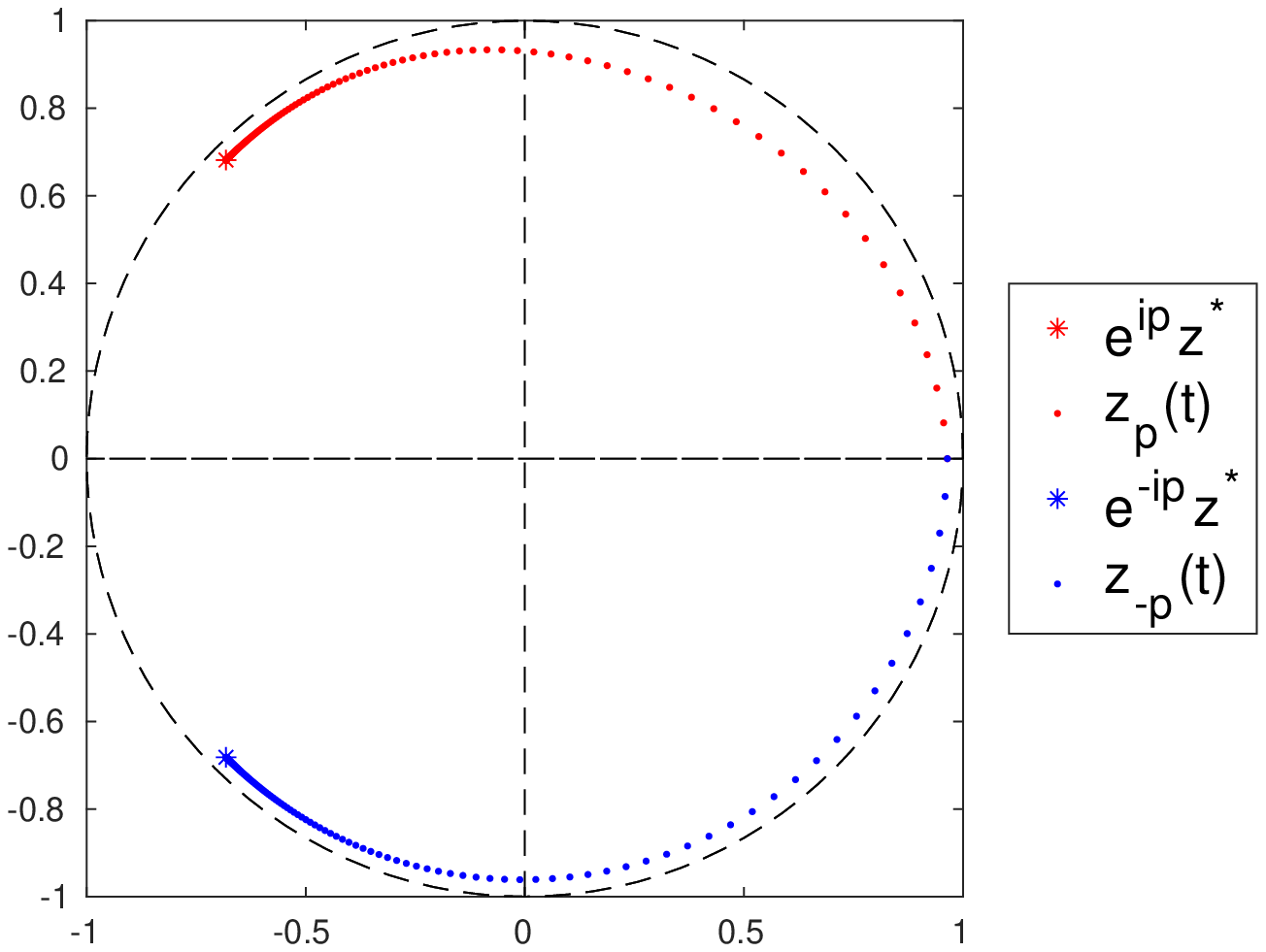}
				\caption{$\omega_0=2\pi/26.4$}
			\end{subfigure}
			\\
			\begin{subfigure}{.5\textwidth}
				\centering
				\includegraphics[scale=0.5]{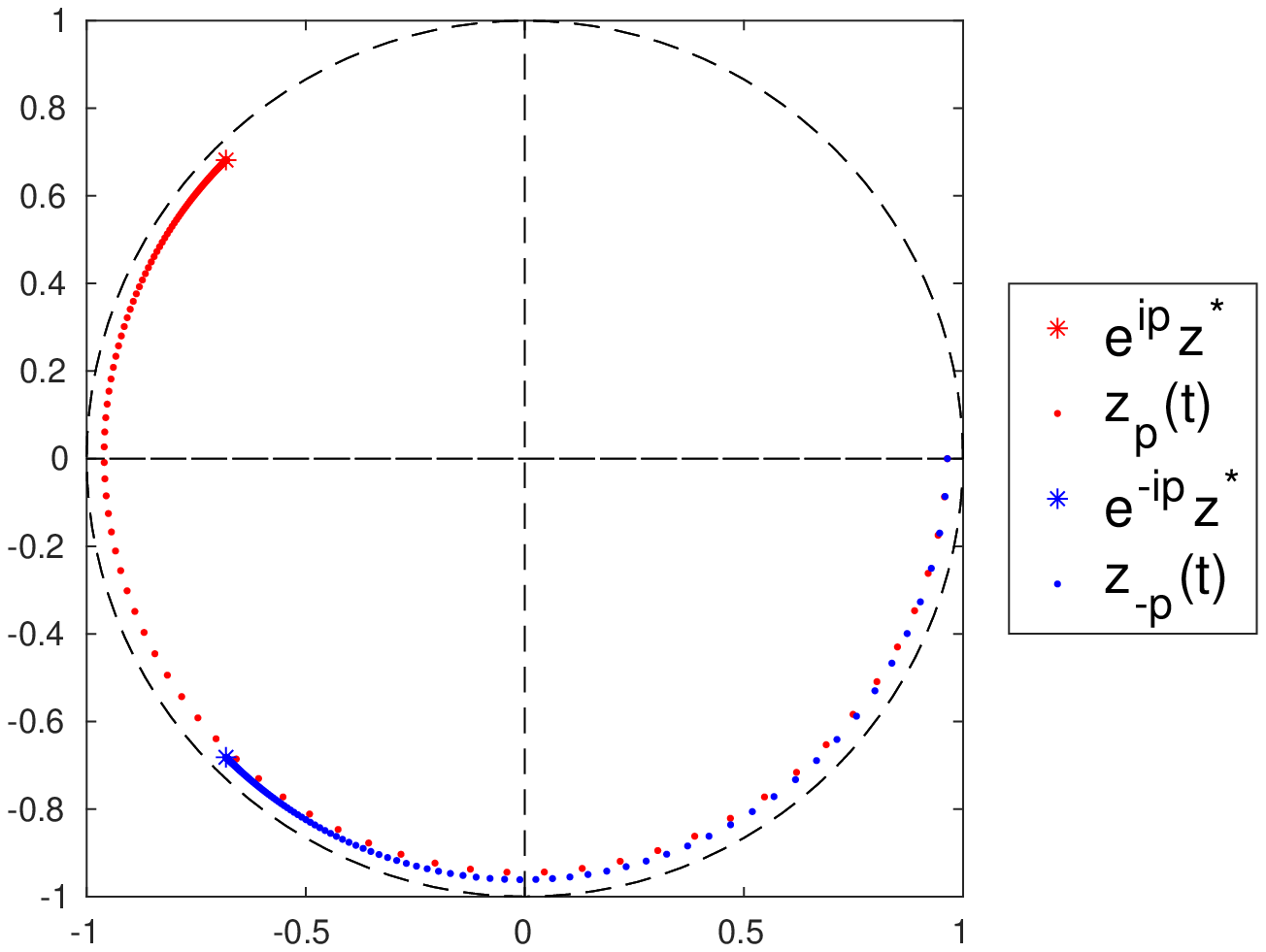}
				\caption{$\omega_0=2\pi/30$}
			\end{subfigure}
			\begin{subfigure}{.5\textwidth}
				\centering
				\includegraphics[scale=0.5]{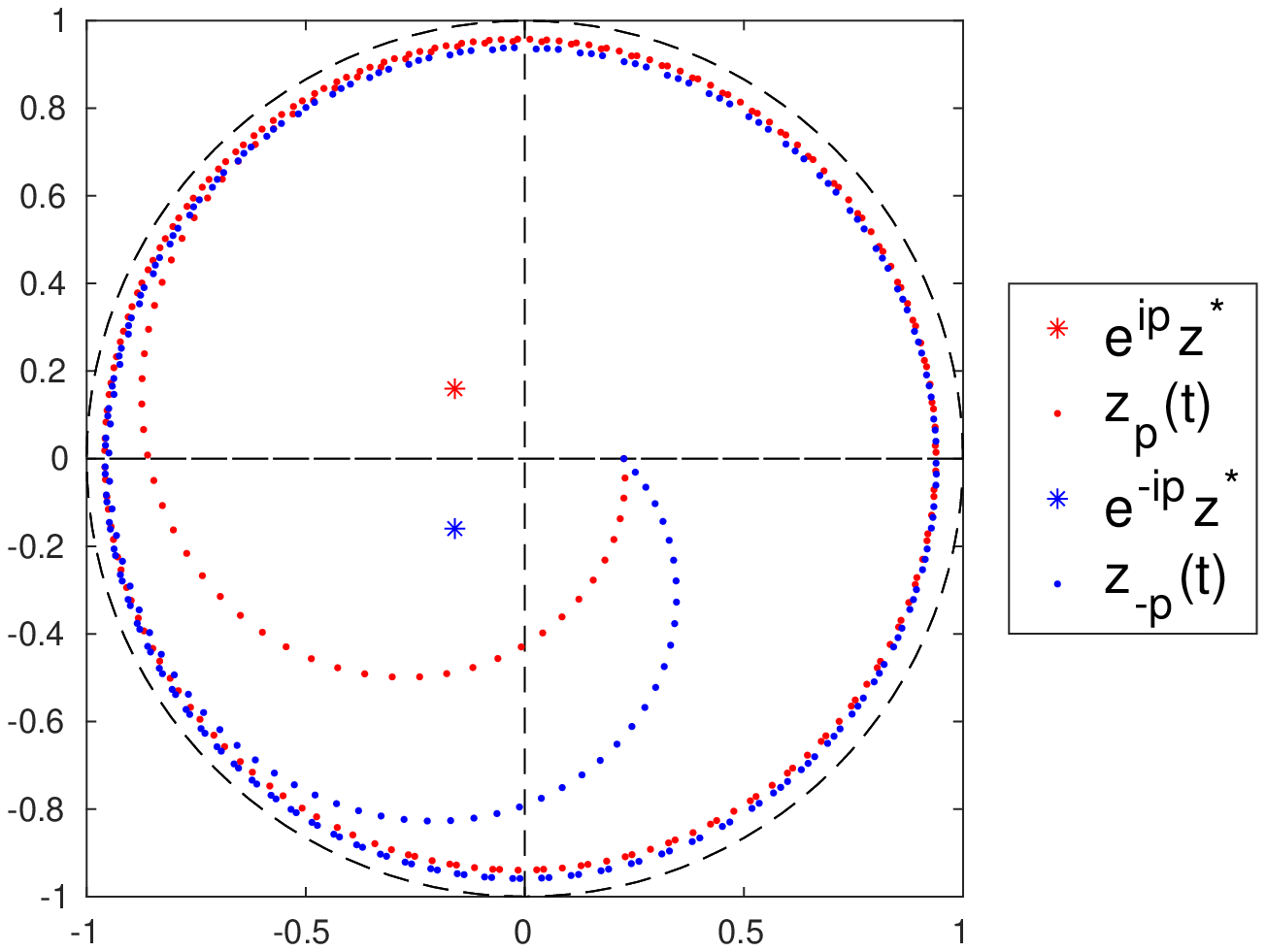}
				\caption{$\omega_0=2\pi/36$}
			\end{subfigure}
			\caption{\textit{Recovery Mean Field Game Problem}: Path $\tilde{z}_p(t)$ while recovering from jet lag after traveling $9$ time zones east (red) and west (blue) for various values of $\omega_0$. A point is plotted every hour. Note that for $\omega_0=2\pi/36$, the oscillators never recover.}
			\label{fig:z_t_tilde_omega}
		\end{figure}
Analogous plots of Figure \ref{fig:z_t_omega} are shown in Figure \ref{fig:z_t_tilde_omega} for the \textit{Recovery Mean Field Game Problem}. Note that in Figure \ref{fig:z_t_tilde_omega}(d), $|\tilde{z}_p(t)|$ increases, meaning that the oscillators are becoming more synchronized with each other. But since the angle of $\tilde{z}_p(t)$ continues to rotate, the oscillators are not able to synchronize with the natural 24 hour cycle. This result leads us to conjecture that if we keep $\omega_0=2\pi/36$ and decrease $K$ (which weighs synchronization of the oscillators with each other) and/or increase $F$ (which weighs synchronization with the natural 24 hour cycle), then the oscillators will be able to recover. We tested this conjecture by decreasing $K$ to $0$ and increasing $F$ to $0.05$. Results are shown in Figure \ref{fig:z_tilde_mixed}. For both east and west trips and both adjustments to the parameters, the oscillators are able to recover from jet lag, which confirms the conjecture.

		\begin{figure}
			\begin{subfigure}{.5\textwidth}
				\centering
				\includegraphics[scale=0.5]{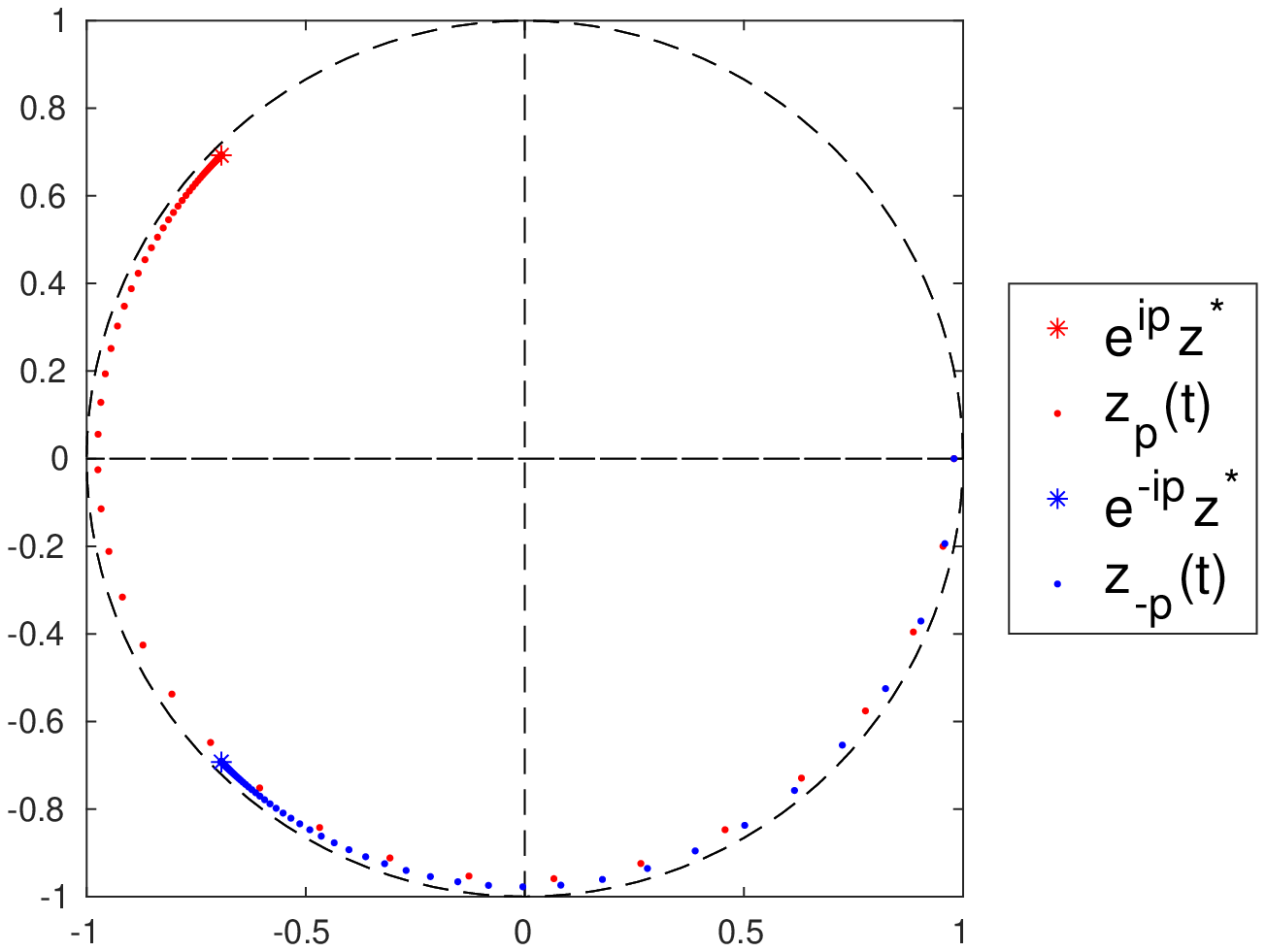}
				\caption{$\omega_0=2\pi/36$, $F=0.05$}
			\end{subfigure}
			\begin{subfigure}{.5\textwidth}
				\centering
				\includegraphics[scale=0.5]{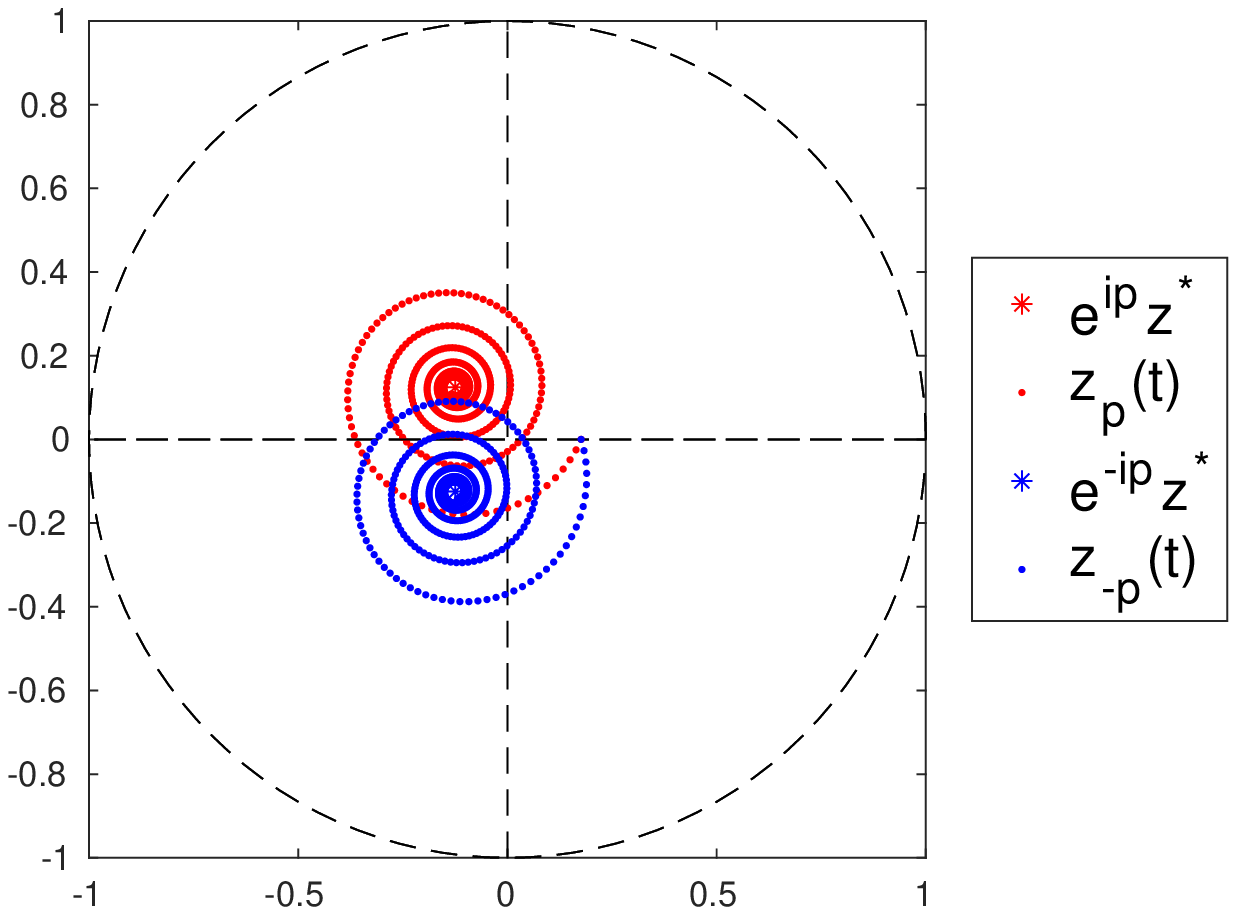}
				\caption{$\omega_0=2\pi/36$, $K=0$}
			\end{subfigure}		
			\caption{\textit{Recovery Mean Field Game Problem}: Path $\tilde{z}_p(t)$ while recovering from jet lag after traveling $9$ time zones when $\omega_0=2\pi/36$. These results confirm our conjecture that increasing $F$ and/or decreasing $K$ will allow the oscillators to recover.}
			\label{fig:z_tilde_mixed}
		\end{figure}
		
Next, we consider if this model formulation for jet lag recovery is consistent with our model formulation for the long time behavior of oscillators within a time zone, as given by the \textit{Ergodic Mean Field Game Problem}. In particular, we should expect that an individual who is entrained to their time zone but decides to solve one of the two recovery problems with $p=0$ (i.e. they do not travel to a different time zone), then they should remain synchronized to their time zone (i.e. $\tilde{\mu}_0(t,\phi)=\mu_0(t,\phi)=\mu^*(\phi),\ \forall t$). Clearly $\mu_0(t,\phi)=\mu^*(\phi),\ \forall t$ for the \textit{Recovery via Ergodic Problem}, since they will continue to use the optimal control $\alpha^*(\phi)$ which will keep them synchronized. For the reference set of parameters, we find that the oscillators also remain synchronized when using the \textit{Recovery Mean Field Game Problem}. For large values of $|\omega_0-\omega_S|$, however, the oscillators become unsynchronized when solving the \textit{Recovery Mean Field Game Problem} with $p=0$. In fact, for $\omega_0=2\pi/36$ we find that $\mathcal{W}_2(\tilde{\mu}_0(t,\phi), \mu^*(\phi))>\epsilon^W$ for all $t \in [T/4,3T/4]$. Figure \ref{fig:z_tilde_p_0} shows the path of $\tilde{z}_0(t)$ when $\omega_0=2\pi/36$. This is seemingly an inconsistency in the modeling of the \textit{Recovery Mean Field Game Problem} when $|\omega_0-\omega_S|$ is large. However, as mentioned earlier, $|\omega_0-\omega_S|$ is small for the application at hand of jet lag recovery in humans, so the model is still feasible for the parameters of interest.
		\begin{figure}
			\centering
			\includegraphics[scale=0.5]{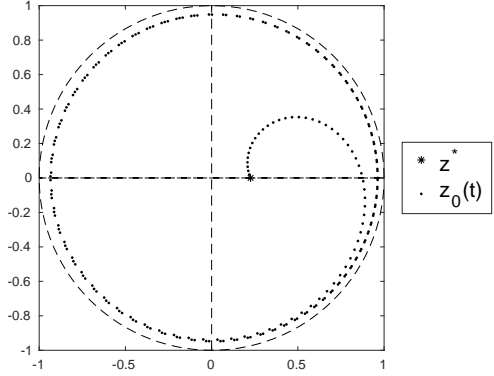}
			\caption{Path $\tilde{z}_0(t)$ after solving the \textit{Recovery Mean Field Game Problem} with $p=0$ when $\omega_0=2\pi/36$. The oscillators become unsynchronized, despite staying in the same time zone.}
			\label{fig:z_tilde_p_0}
		\end{figure}

Finally, one many wonder if the choice of the finite time horizon, $T=100$ days, affects whether the oscillators become unsynchronized when $p=0$ and $|\omega_0-\omega_S|$ is large. Surprisingly, we still find that $\mathcal{W}_2(\tilde{\mu}_0(t,\phi), \mu^*(\phi))>\epsilon^W$ for all $t \in [T/4,3T/4]$ for $T=50$, $T=100$, $T=150$, and $T=200$ days when $p=0$ and $\omega_0=2\pi/36$. This is surprising because one might expect that as $T \rightarrow \infty$, we should recover the solution to the \textit{Ergodic Mean Field Game Problem} in some sense. This is the focus of the papers of Cardaliaguet, Lasry, Lions, and Porretta \cite{cardaliaguet2012long}\cite{cardaliaguet2013long}. Note that our problem does not fit into the framework in \cite{cardaliaguet2013long} for two reasons: 1) we have an extra constant ($\omega_0-\omega_S$) in the drift, and 2) the running cost does not satisfy the monotonicity assumption. The monotonicity assumption is crucial in the proof of Theorem 3.1 in \cite{cardaliaguet2013long} and the numerical results violate the statement of Theorem 3.1 in the regime of large $|\omega_0-\omega_S|$.

\section{Conjectures}\label{sec:Conjectures}
\hspace{5mm} From the numerical results presented in the last section, we pose the following conjecture for the \textit{Ergodic Mean Field Game Problem}:
\begin{enumerate}
	\item For the reference set of parameters, there exists $l>0$ and $r>0$ such that $\omega_0-\omega_S+\alpha^*(\phi)>0$ for $\phi \in [-l,0)$, and $\omega_0-\omega_S+\alpha^*(\phi)<0$ for $\phi \in (0,r]$, and $\omega_0-\omega_S+\alpha^*(0)=0$.
\end{enumerate}
We pose the following conjectures for the \textit{Recovery via Ergodic Problem}:
\begin{enumerate}
	\setcounter{enumi}{1}
	\item If $\omega_0 < \omega_S$, then $f_{|p|} \geq f_{-|p|}$. In other words, the recovery cost is larger for the eastward trip. Similarly, if $\omega_0 > \omega_S$, then $f_{|p|} \leq f_{-|p|}$.
\end{enumerate}
When the rest of the parameters remain at the reference values, we conjecture the following behavior as we change one parameter at a time for the \textit{Recovery via Ergodic Problem}:
\begin{enumerate}
	\setcounter{enumi}{2}
	\item Changing $p$:
	\begin{enumerate}
		\item $f_p$ increases with $|p|$.
		\item For $p/\omega_S\in(0,12)$, $f_p>f_{-p}$.
	\end{enumerate}
	\item Changing $\omega_0$:
	\begin{enumerate}
		\item There is a threshold $\Omega^*$ such that if $|\omega_0-\omega_S|<\Omega^*$, oscillators will phase advance after traveling east and phase delay after traveling west. If $|\omega_0-\omega_S|>\Omega^*$, the oscillators will phase delay if $\omega_0<\omega_S$ and phase advance if $\omega_0>\omega_S$.
		\item $f_p$ increases with $|\omega_0-\omega_S|$.
	\end{enumerate}
	\item Changing $\sigma$:
	\begin{enumerate}
		\item $\tau^W_p \rightarrow 0$ and  $\tau^z_p \rightarrow 0$ as $\sigma \rightarrow \infty$.
		\item $f_p$ increases with $\sigma$.
	\end{enumerate}
	\item Changing $K$:
	\begin{enumerate}
		\item $\tau^W_p$ and $\tau^z_p$ decrease with $K$.
		\item $f_p$ increases with $K$.
	\end{enumerate}
	\item Changing $F$:
	\begin{enumerate}
		\item $\tau^W_p$ and $\tau^z_p$ decrease with $F$.
		\item $f_p$ increases with $F$.
	\end{enumerate}
\end{enumerate}
For the \textit{Recovery Mean Field Game Problem} we pose the following conjectures:
\begin{enumerate}
	\setcounter{enumi}{7}
	\item For a given $p \in [-\pi,\pi]$ and $\epsilon^W>0$, there are regimes of the parameter space where the oscillators will $\epsilon^W$-recover after traveling to a new time zone angle $p$ for a sufficiently large $T>0$, and regimes where the oscillators will not $\epsilon^W$-recover after traveling to a new time zone angle $p$ for any $T>0$. The regimes where the oscillators do recover are characterized by small $|\omega_0-\omega_S|$, small $K$, large $F$, and large $\sigma$.
	\item $\tilde{\tau}^W_p=\tilde{\tau}^W_{-p}$ and $\tilde{\tau}^z_p=\tilde{\tau}^z_{-p}$. In other words, the recovery times are the same for eastward and westward travels.
	\item For a small $\epsilon^W>0$, the oscillators do not $\epsilon^W$-recover for sufficiently large $|\omega_0-\omega_S|$.
	\item There exists $T^*>0$ such that $\tilde{\tau}^W_p$, $\tilde{\tau}^z_p$, and $\tilde{f}_p$ are the same for any $T>T^*$.
\end{enumerate}

\section{Conclusion}\label{sec:Conclusion}
\hspace{5mm} We provided a mean field game formulation for the synchronization of SCN circadian oscillators. The long time behavior of the oscillators is described by the so called \textit{Ergodic Mean Field Game Problem}, in which the oscillators optimize over an infinite horizon. Assuming that travel is immediate and the oscillators are entrained to their time zone before travel, the transitional behavior of the oscillators while they recover from jet lag is described by either the \textit{Recovery via Ergodic Problem} or the \textit{Recovery Mean Field Game Problem}. In the \textit{Recovery via Ergodic Problem}, the oscillators adapt the control they have already learned from the \textit{Ergodic Mean Field Game Problem} to resynchronize to a new time zone. In the \textit{Recovery Mean Field Game Problem}, on the other hand, the oscillator re-optimize to find a mean field game equilibrium over a finite time horizon. A finite differences approach was implemented to solve the above problems numerically. The numerics suggest that the time to recover from jet lag is about the same for eastward and westward travels. However, the cost accrued while recovering from jet lag is larger for eastward travels. This is consistent with the experience of frequent travelers who claim that it is harder to recover from jet lag after traveling east.

\end{document}